\documentclass[a4paper,10pt]{article}
\usepackage{amsmath,amssymb,latexsym,amsthm}
\usepackage{array,graphicx,epsfig,xcolor}
\usepackage{hyperref}
\usepackage{mathrsfs}

\def\R{{\mathbb R}}

\def\Ocal{{\mathcal{O}}}

\def\L{{\mathbf{L}}}
\def\ds{\displaystyle}

\def\y{{\mathbf{y}}}
\def\u{{\mathbf{u}}}
\def\n{{\mathbf{n}}}
\def\f{{\mathbf{f}}}
\def\g{{\mathbf{g}}}
\def\G{{\mathbf{G}}}
\def\X{{\mathbf{X}}}
\def\v{{\mathbf{v}}}
\def\V{{\mathbf{V}}}
\def\h{{\mathbf{h}}}
\def\H{{\mathbf{H}}}
\def\div{{\, {\rm div} \, }}
\def\rot{{\, {\rm rot}\, }}

\newcommand{\ov}{:\overset{ \mbox{\tiny def}}{=}}

\newcommand{\norm}[1]{\left\Vert#1\right\Vert}

\newtheorem{theorem}{Theorem}[section]

\newtheorem{proposition}[theorem]{Proposition}
\newtheorem{lemma}[theorem]{Lemma}

\numberwithin{equation}{section}

\title{Local controllability to trajectories for non-homogeneous $2$-d incompressible Navier-Stokes equations\footnote{This work is partially supported by the Agence Nationale de la Recherche (ANR, France), Project CISIFS number NT09-437023.}}
\author{Mehdi Badra\footnote{Laboratoire LMA, UMR CNRS 5142, Universit\'e de Pau et des Pays de l'Adour, F-64013 Pau Cedex, France. E-mail: {\tt mehdi.badra@univ-pau.fr}}, Sylvain Ervedoza\footnote{Corresponding author. Institut de Math\'ematiques de Toulouse ; UMR5219 ; Universit\'e de Toulouse ; CNRS ; UPS IMT, F-31062 Toulouse Cedex 9, France.  E-mail: {\tt sylvain.ervedoza@math.univ-toulouse.fr}.
Address: Institut de Math\'ematiques de Toulouse, Universit\'e Paul Sabatier, 118 route de Narbonne, 31062 Toulouse Cedex 9. Ph.: +33561557654}, and Sergio Guerrero\footnote{Universit\'e Pierre et Marie Curie, UMR 7598, Laboratoire Jacques-Louis Lions, F-75005 Paris, France. E-mail: {\tt guerrero@ann.jussieu.fr}}
}
\date{\today}
\begin{document}
\maketitle
\abstract{The goal of this article is to show a local exact controllability to smooth ($C^2$) trajectories for the $2$-d density dependent incompressible Navier-Stokes equations. Our controllability result requires some geometric condition on the flow of the target trajectory, which is remanent from the transport equation satisfied by the density. The proof of this result uses a fixed point argument in suitable spaces adapted to a Carleman weight function that follows the flow of the target trajectory. Our result requires the proof of new Carleman estimates for heat and Stokes equations.}
\smallskip
\\
\noindent {\bf AMS subject classifications.}
\section{Introduction}
The goal of this article is to discuss the local exact controllability property for the $2$-d non-homogeneous Navier Stokes equations.

{\bf Setting and main results.}
Let $\Omega$ be a smooth bounded domain of $\R^2$, $T>0$ and denote $(0,T) \times \Omega$ by $\Omega_T$. Let us consider a trajectory $(\overline {\sigma},\overline {\bf y})$ of the non-homogeneous Navier-Stokes equations:
\begin{equation}\label{EqEtatCible}
	 \left\{
	\begin{array}{rlll}
		\partial_t\overline \sigma+{\rm div}(\overline \sigma\, \overline {\bf y})&=& \overline{f}_\sigma &\mbox{ in } \Omega_T,
		\\
		\overline \sigma \partial_t \overline {\bf y}+ \overline \sigma (\overline {\bf y}\cdot \nabla)\overline {\bf y}-\nu \Delta \overline {\bf y}+\nabla \overline q&= & \overline{\f}_\y \quad &\mbox{ in } \Omega_T,
		\\
		{\rm div}\,\overline {\bf y}&=& 0\quad &\mbox{ in } \Omega_T,
		\\
 		(\overline{\sigma}(0), {\overline{\bf y}}(0))&=&(\overline{\sigma}_0,\overline {\bf y}_0)\quad &\mbox{ in }\Omega.
	\end{array}
	\right.
\end{equation}
Here, $\nu>0$ is the viscosity parameter and the source terms $(\overline{f}_\sigma,\overline{\f}_\y )$ are assumed to be known.\\ 
We will focus on the local exact controllability problem around the trajectory $(\overline{\sigma}, \overline{\y})$ with a control exerted on the boundary $(0,T) \times \partial \Omega$: Given $(\sigma_0,\y_0)$ close to the initial data $(\overline{\sigma}_0, \overline{\y}_0)$, find control functions $(h_\sigma, \h_\y)$ on $(0,T) \times \partial \Omega$ such that the solution $(\sigma, \y)$ of
\begin{equation}\label{EqEtat}
 \left\{
	\begin{array}{rlll}
		\partial_t\sigma+{\rm div}(\sigma {\bf y})&=& \overline{f}_\sigma &\mbox{ in } \Omega_T,
		\\
		\sigma \partial_t {\bf y}+ \sigma ({\bf y}\cdot \nabla){\bf y}-\nu \Delta {\bf y}+\nabla q&=& \overline{\f}_\y\quad &\mbox{ in } \Omega_T,
		\\
		{\rm div}\,{\bf y}&=&0\quad &\mbox{ in }  \Omega_T,
		\\
 		({\sigma}(0), {\bf y}(0))
		&=& (\overline{\sigma}_0+\rho_0, \overline {\bf y}_0+{\bf u}_0),\quad &\mbox{ in }\Omega,
	\end{array}
\right.
\end{equation}
with the boundary conditions:
\begin{eqnarray}
		\sigma & =& \overline{\sigma}+ h_\sigma  \mbox{ for } (t,x) \in (0,T) \times \partial \Omega, \mbox{ with }\, \y(t,x)\cdot \n(x) <0, \label{BordDensite}
		\\
		{\bf y} & = & \overline {\bf y}+\h_\y \mbox{ on } (0,T)\times \partial \Omega, \label{BordVitesse}
\end{eqnarray}
satisfies
\begin{equation}\label{CondFinale}
	  (\sigma(T),{\bf y}(T))=(\overline {\sigma}(T), \overline {\bf y}(T)).
\end{equation}
Our goal is to present a positive answer to this control problem under suitable assumptions on the target trajectory $(\overline{\sigma}, \overline{\y})$, and in particular one of hyperbolic nature on the flow corresponding to $\overline{\y}$. Besides, our strategy will yield a control acting on some suitable subsets of the boundary which correspond, roughly speaking, to the complement of the part of the boundary in which the scalar product of the target velocity $\overline{\y}$ with the normal vector $\n$ is positive for all time $t \in [0,T]$.
\\
Going further requires some notations. We denote by $L^2(\Omega)$, $L^\infty(\Omega)$, $H^r(\Omega)$, $H_0^r(\Omega)$ etc for $r\geq 0$, the usual Lebesgue and Sobolev spaces of scalar functions,
and we write in bold the spaces of vector-valued functions: ${\bf L}^2(\Omega)=(L^2(\Omega))^2$, ${\bf H}^r(\Omega)=(H^r(\Omega))^2$, etc. We also define
$$
	\V^1_0(\Omega) \ov \{\v \in {\bf H}^1_0(\Omega) \mid \div \v  = 0 \mbox{ in } \Omega \}.
$$
In the following, we will always assume that the target velocity $\overline{\y}$ belongs to ${\bf C}^2(\overline{\Omega_T})$. It can thus be extended into a ${\bf C}^2([0,T]\times \R^2)$ function, still denoted the same for simplicity but not necessarily divergence free outside $\Omega_T$.
This allows to define the flow $\overline{X} = \overline{X}(t,\tau, x)$ associated to that velocity $\overline{\y}$:
\begin{equation}
	\label{Def-FlowBAR}
		\forall (t,\tau,x)\in [0,T]^2\times \mathbb{R}^2,\quad  \partial_t \overline{X}(t,\tau,x)=\overline {\bf y}(t,  \overline{X}(t,\tau,x)),\quad \overline{X}(\tau,\tau,x)=x.
\end{equation}
Thus we define the outgoing subset of $ \overline{\Omega}$ for the flow $\overline{X}$ as follows:
\begin{equation}
	\label{OmegaOut}
	\Omega_{{\rm out}}^T\ov  \left \{x\in \overline{\Omega}\mid  \exists t\in (0,T)\,\mbox{ s.t. }\,\overline{X}(t, 0 ,x)\in \mathbb{R}^2\backslash \overline{\Omega} \right \}.
\end{equation}
One of our main assumptions is the following one:
\begin{equation}
	\label{OmegaEqOmegaOut}
	\overline{\Omega}=\Omega_{{\rm out}}^T.
\end{equation}
Note that this assumption does not depend on the extension $\overline{\y}$ on $[0,T] \times \R^2$ and is intrinsic. This assumption is of hyperbolic nature as it requires the time $T$ to be large enough to guarantee that all the particles that were in $\overline{\Omega}$ at time $t = 0$ have been transported by the flow outside $\overline{\Omega}$ in a time strictly smaller than $T$. Of course, this is remanent from the density equation $\eqref{EqEtat}_{(1)}$ in which the density is transported along the flow corresponding to the velocity of the fluid.
\\
 As we said, we will not require the control to be supported on the whole boundary $(0,T) \times \partial \Omega$, but only on some part of it $(0,T) \times \Gamma_c$ where $\Gamma_c=\partial\Omega\backslash \Gamma_0$ and $\Gamma_0$ (the part without control) is an open subset of $\partial \Omega$ satisfying the following conditions:
 \begin{equation}
 	\label{Cond-Gamma-c}
 	\begin{split}
	&{\rm (i).} \quad\Gamma_0 \hbox{ has a finite number of connected components,}\\
	 &{\rm (ii).}\quad  \sup_{[0,T]\times \overline{\Gamma_0}} \overline{\y}\cdot \n>0.
	 \end{split}
\end{equation}
Note that the above condition garantees the existence of $\gamma>0$ such that $\overline{\bf y}(t,x)\cdot {\bf n}(x) \geq \gamma$ for all $(t,x) \in (0,T)\times \Gamma_0$.

Our main result states as follows:
\begin{theorem}
	\label{Thm-Main}
	Let $\Omega$ be a smooth bounded domain of $\R^2$. Assume that the target trajectory $(\overline{\sigma}, \overline{\y})$ solution of \eqref{EqEtatCible} satisfies
	\begin{equation}
		\label{Conditions-TargetTrajectory}
		(\overline{\sigma}, \overline{\y}) \in C^2([0,T] \times \overline{\Omega})\times {\bf C}^2([0,T] \times \overline{\Omega}) \quad \hbox{ and } \inf_{[0,T]\times \overline{\Omega}} \overline{\sigma} >0.
	\end{equation}
	Assume that the condition \eqref{OmegaEqOmegaOut} is satisfied for the time $T$.
	
	Then there exists $\varepsilon>0$ such that for all $(\rho_0, \u_0) \in L^\infty (\Omega) \times {\bf V}^1_0(\Omega)$ satisfying
	\begin{equation}
		\label{SmallDataIni}
			\|\rho_0\|_{L^\infty(\Omega)}+ \|{\bf u}_0\|_{{\bf H}^1_0(\Omega)}\leq \varepsilon,
	\end{equation}
	there exists a controlled trajectory
	$$
		(\sigma, \y) \in L^\infty(\Omega_T) \times H^1(0,T; \L^2(\Omega)) \cap L^2(0,T; {\bf H}^2(\Omega))
	$$
	solution of \eqref{EqEtat}--\eqref{BordVitesse} satisfying the control requirement \eqref{CondFinale}.
	\\
	Besides, if $\Gamma_0$ denotes an open subset of the boundary satisfying \eqref{Cond-Gamma-c}, we may further impose $\y = \overline{\y}$ on $(0,T) \times \Gamma_0$. In particular, in that case, no boundary condition is imposed on the density on $\Gamma_0$.
\medskip
\end{theorem}

Actually, we will only prove Theorem \ref{Thm-Main} when $\Gamma_0 \neq \emptyset$. When $\Gamma_0 = \emptyset$, Theorem \ref{Thm-Main} can be proved more easily following the same lines, as the extensions arguments we will perform can be handled much more easily.

{\bf Strategy of the proof.}
The proof of Theorem \ref{Thm-Main} is based on a technical fixed-point procedure, and we briefly explain below its general strategy.
\\
Setting
\begin{equation}
	\label{Def-rho-u}
	\rho \ov \sigma - \overline{\sigma}, \quad \u \ov \y - \overline{\y},
\end{equation}
and
\begin{equation}
	\label{Def-f}
		\f(\rho,{\bf u})
		\ov
		-\rho(\partial_t {\bf u}+((\overline {\bf y}+{\bf u})\cdot \nabla){\bf u}+({\bf u}\cdot \nabla)\overline {\bf y} )-\overline {\sigma} ({\bf u}\cdot \nabla){\bf u}
		- \rho (\partial_t \overline {\bf y}+(\overline {\bf y}\cdot \nabla)\overline {\bf y}),
\end{equation}
equations \eqref{EqEtat}--\eqref{CondFinale} rewrite
\begin{equation}
	\label{EqEtat2}
	 \left\{
		\begin{array}{rlll}
			\partial_t\rho+(\overline {\bf y}+{\bf u})\cdot \nabla \rho &=& -{\bf u}\cdot \nabla \overline{\sigma}&\mbox{ in } \Omega_T,
			\\
			\overline {\sigma}\partial_t {\bf u} + \overline {\sigma} (\overline {\bf y} \cdot \nabla){\bf u}+ \overline {\sigma} ({\bf u} \cdot \nabla)\overline {\bf y}-\nu \Delta {\bf u}+\nabla p&=& \f(\rho,{\bf u})&\mbox{ in } \Omega_T,
			\\
			{\rm div}\,{\bf u}&=& 0\quad &\mbox{ in } \Omega_T,
			\\
			(\rho(0),{\bf u}(0))&=&(\rho_0,{\bf u}_0)\quad &\mbox{ in }\Omega,
		\end{array}
	\right.
\end{equation}
with the boundary conditions
\begin{equation}
	\label{CondBordHomogene}
	  {\bf u} = {\bf 0}\quad \mbox{ on } (0,T) \times \Gamma_0,
\end{equation}
and with the requirement
\begin{equation}
	\label{CondFinale2}
	(\rho(T),{\bf u}(T))=(0, {\bf 0}) \quad \mbox{ in } \Omega.
\end{equation}
To construct a solution of \eqref{EqEtat2}--\eqref{CondFinale2}, the strategy consists in finding a fixed-point to some mapping $\mathscr{F}_{(\rho_0,{\bf u}_0)}:\widehat {\bf u}\mapsto {\bf u}$
defined in such a way that $\u = \mathscr{F}_{(\rho_0,{\bf u}_0)}(\widehat {\bf u})$ is a suitable solution of:
\begin{equation}
	\label{EqUg}
	 \left\{
	\begin{array}{rlll}
		\partial_t \rho+(\overline {\bf y}+\widehat {\bf u})\cdot \nabla  \rho &= &-\widehat {\bf u}\cdot \nabla \overline{\sigma}&\mbox{in } \Omega_T,
		\\
		\overline {\sigma}\partial_t {\bf u} + \overline {\sigma} (\overline {\bf y} \cdot \nabla){\bf u}+ \overline {\sigma} ({\bf u} \cdot \nabla)\overline {\bf y}-\nu \Delta {\bf u}+\nabla p&= & \f(\rho,\widehat {\bf u}) &\mbox{in }\Omega_T,
		\\
		{\rm div}\,{\bf u}&= &0\quad &\mbox{in } \Omega_T,
		\\
		{\bf u}&= &{\bf 0}\quad &\mbox{on } (0,T)\times  \Gamma_0,
		\\
		(\rho(0),{\bf u}(0))&=&(\rho_0,{\bf u}_0)\, &\mbox{in }\Omega,
		\\
	      (\rho(T),{\bf u}(T))&=&(0,{\bf 0})\, &\mbox{in }\Omega.
	\end{array}
	\right.
\end{equation}
The mapping $\mathscr{F}_{(\rho_0,{\bf u}_0)}$ is defined in two steps. First, for a given $\widehat \u$, we define
$
	\mathscr{F}_1(\widehat {\bf u},\rho_0)\ov \rho,
$
where $\rho$ will be constructed as a suitable solution of the following control problem for the equation of the density:
\begin{equation}
	\label{EqTransport}
	 \left\{
		\begin{array}{rlll}
			\partial_t \rho+(\overline {\bf y}+\widehat {\bf u})\cdot \nabla  \rho &=&-\widehat {\bf u}\cdot \nabla \overline{\sigma}&\mbox{ in } \Omega_T,
		\\
			\rho(0)&=&\rho_0\quad &\mbox{ in }\Omega,
		\\
	        		\rho(T)&=&0\quad &\mbox{ in }\Omega.
		\end{array}
	\right.
\end{equation}
Then, we define
$
	\mathscr{F}_2({\bf f}, {\bf u}_0)\ov {\bf u},
$
where ${\bf u}$ is a suitable solution of the following control problem for the equation of the velocity:
\begin{equation}
	\label{SystemparaboliquePF}
	 \left\{
		\begin{array}{rlll}
			\overline {\sigma}\partial_t {\bf u} + \overline {\sigma} (\overline {\bf y} \cdot \nabla){\bf u}+ \overline {\sigma} ({\bf u} \cdot \nabla)\overline {\bf y}-\nu \Delta {\bf u}+\nabla p&= & {\bf f}&\mbox{ in } \Omega_T,
			\\
			{\rm div}\,{\bf u}&=&0\quad &\mbox{ in } \Omega_T,
			\\
			{\bf u}&=&{\bf 0}\quad &\mbox{ on } (0,T)\times \Gamma_0,
			\\
			{\bf u}(0)&=&{\bf u}_0\quad &\mbox{ in }\Omega,
			\\
			{\bf u}(T)&=& {\bf 0}\quad &\mbox{ in }\Omega.
		\end{array}
	\right.
\end{equation}
The mapping $\mathscr{F}_{(\rho_0,{\bf u}_0)}$ is then defined as follows:
\begin{equation}
	\label{DefF}
		\mathscr{F}_{(\rho_0,{\bf u}_0)}(\widehat {\bf u})\ov {\bf u},
		\quad \mbox{ where }
				\rho =\mathscr{F}_1(\widehat {\bf u},\rho_0),
				\mbox{ and }
				{\bf u} =\mathscr{F}_2({\bf f}(\rho,{\bf \widehat u}),{\bf u}_0).
\end{equation}
Hence our strategy decouples the control problem \eqref{EqEtat}--\eqref{CondFinale} into two control problems,  \eqref{EqTransport} for the equation of the density, and \eqref{SystemparaboliquePF} for the equation of the velocity, each of which  having different behaviors.

Indeed, on one hand, the control problem \eqref{SystemparaboliquePF} is of parabolic type, and it will be handled by using global Carleman estimates following the general approach of Fursikov and Imanuvilov \cite{FursikovImanuvilov} for the heat equations: in the case of Navier-Stokes equations, this approach has already been successfully implemented in the works \cite{Imanuvilov2001,FerCarGueImaPuel}.

On the other hand, the control problem \eqref{EqTransport} involves a transport equation. This can be easily controlled provided the time $T>0$ is large enough to allow all the particles in $\Omega$ to go outside the domain, i.e. when condition \eqref{OmegaEqOmegaOut} is satisfied.

But the problem is that we want the above mapping $\mathscr{F}_{(\rho_0,\u_0)}$ to map some convex set into itself. In order to do this, we should be able to get estimates on the above control problems in spaces that behave suitably with respect to both of them. In particular, this will lead us to introduce Carleman weights that follow the dynamics of the transport equation, that is weight functions which are transported by the flow. This strategy then follows the one recently developed in \cite{EGGP} for deriving local exact controllability results for the $1$d compressible Navier-Stokes equations around constant non-vanishing velocities.

Actually, the Carleman estimates we develop in this article also present the feature of not vanishing at time $t = 0$. This allows us to construct a solution $(\rho, \u)$ of \eqref{EqEtat2} without using any property of the Cauchy problem for the non-homogeneous Navier-Stokes equations.

%
{\bf Related references and comments.} To our knowledge, control properties for non-homo\-geneous Navier-Stokes equations have only been studied in \cite{Fernandez-Cara-2012}, which proves several optimal control results in that context for various cost functions.

For the homogeneous Navier-Stokes equations, the density is assumed to be constant and thus the equations reduce to the equations on the velocity. In that case, several local exact controllability results have been established in \cite{Imanuvilov2001,FerCarGueImaPuel} based on parabolic Carleman estimates, see e.g. \cite{FursikovImanuvilov,FernandezCaraGuerrero-2006-SICON}. Later on, several different strategies have been proposed, see for instance \cite{GuerreroPuel2,GB-G-P,ImanuvilovPuelYam-2009}. We also point out that these results also use the Carleman estimate derived in \cite{ImaPuel} for non-homogeneous elliptic problems in order to handle the pressure term.

But our problem also involves some transport phenomenon, and therefore also shares some features of the thermoelasticity equations \cite{AlbanoTataru}, the viscoelasticity models \cite{MartinRosierRouchon,ChavesRosierZuazua-2013}, and the compressible Navier-Stokes equations \cite{EGGP}. Our approach is actually close to the one developed in \cite{EGGP}. Though, the divergence free condition in the model we consider here requires a specific treatment.

In this article, we will not use any result on the Cauchy problem for \eqref{EqEtat}, as our strategy will automatically construct a trajectory $(\sigma, \y)$ solving the equations \eqref{EqEtat}. However, several results are available in the literature. We refer to the work \cite{Fernandez-Cara-2012} for several results and comments on the Cauchy problem for the non-homogeneous incompressible Navier-Stokes equations and to the references therein.

Let us also note that we will need a precise understanding of the transport equation when transported by a flow entering the domain. More precisely, we will use in an essential way the compactness result in \cite[Theorem 4]{BoyerFabrie-2007}, obtained as a consequence of \cite{Boyer-2005}.

We also underline that Theorem \ref{Thm-Main} does not state the uniqueness of the controlled trajectory $(\sigma,\y)$. This is due to the lack of regularity for the density $\sigma$ which only belongs to $L^\infty(\Omega_T)$, see \cite{Desjardins1997} for the uncontrolled case. Another limitation of this control result is that it is only valid for $2$-d geometry. This restriction comes from the treatment of the velocity equation. For that, we prove a new Carleman inequality for the Stokes equations which hardly relies on the use of the stream function of the velocity, see Section \ref{Sec-Velocity}. 

Finally note that our result also allows the use of non-trivial trajectories. For instance, if $\Gamma_0 = \emptyset$ and $(\overline{\sigma}, \overline{\y}) = (1,{\bf 0})$, one may consider the trajectory $(\overline{\sigma}^*(t), \overline{\y}^*(t)) = (1, \eta(t/T) {\bf U})$ for constant vector fields ${\bf U}$ and $\eta = \eta(t) \in [0,1]$ a bump function taking value $0$ at $t = 0$ and $t = 1$ and with $\eta = 1$ on $[1/3, 2/3]$. Note that $(\overline{\sigma}^*(t), \overline{\y}^*(t))= (1,{\bf 0}) $ at time $t = 0$ and at time $t = T$. But for $T>0$ and large ${\bf U}$, $(\overline{\sigma}^*(t), \overline{\y}^*(t)) $ satisfies \eqref{OmegaEqOmegaOut} and all the assumptions of Theorem \ref{Thm-Main}, while whatever the time $T>0$ is, the trajectory $(\overline{\sigma}(t), \overline{\y}(t)) = (1,{\bf 0})$ clearly does not satisfy \eqref{OmegaEqOmegaOut}. This suggests that the geometric condition \eqref{OmegaEqOmegaOut} may be avoided in some cases using ``return method'' type ideas, see e.g. \cite{Co2,Cor-Fur}. 

{\bf Outline.} This article is organized as follows. Section \ref{Sec-Velocity} explains how to solve the control problem \eqref{SystemparaboliquePF} by the use of Carleman estimates for the Stokes operator. Section \ref{Sec-Density} shows how to construct a controlled density satisfying \eqref{EqTransport} and to derive weighted estimates on it. Section \ref{Sec-Proof-Thm-Main} then focuses on the proof of Theorem \ref{Thm-Main} by putting together the arguments developed in Sections \ref{Sec-Velocity} and \ref{Sec-Density}. Finally, the Appendix gives the detailed proofs of some technical results.
%
\section{Controlling the velocity}\label{Sec-Velocity}
This section is dedicated to the construction of a solution of \eqref{SystemparaboliquePF}.
\subsection{Statement of the result}\label{Sect4.1}
In order to solve the control problem \eqref{SystemparaboliquePF}, we will consider \eqref{SystemparaboliquePF} in an extended domain $\Ocal$ as follows: $\Ocal$ is a smooth bounded domain of $\mathbb{R}^2$ satisfying
\begin{equation}
	\label{ExtensionOfOmega}
	\Omega\subset \Ocal,\quad \partial\Ocal \mbox{ is of class $C^2$},\quad \partial \Ocal\cap \partial \Omega\supset \Gamma_0.
\end{equation}
We then extend $(\overline{\sigma}, \overline{\y})$ on $[0,T] \times \overline{\Ocal}$, still denoted the same for simplicity, such that
\begin{equation}
	\label{Positivity}
	(\overline{\sigma}, \overline{\y}) \in C^2([0,T] \times \overline{\Ocal})\times {\bf C}^2([0,T] \times \overline{\Ocal}) \quad \mbox{ and } \, \inf_{[0,T] \times \overline{\Ocal}} \overline \sigma(t,x) > 0.
\end{equation}
Remark that this is possible due to the assumption \eqref{Conditions-TargetTrajectory}. As $\u_0 \in \V_0^1(\Omega)$, extending it by zero outside $\Omega$, we get an extension, still denoted the same, such that
\begin{equation}
	\label{CCPatu0}
	{\u}_0\in {\bf H}_0^1(\Ocal) \quad\mbox{ and }\quad \div {\bf u}_0= 0 \quad \text{ in } \Ocal.
\end{equation}
By also extending $\f$ by zero outside $\Omega$ and setting $\Ocal_T = (0,T) \times \Ocal$, $\Gamma_T = (0,T)\times\partial\Ocal$ we then consider the following system
\begin{equation}
	\label{StokesControl}
		\left\{
			\begin{array}{rlll}
				\overline\sigma  (\partial_t \u + (\overline{\y} \cdot \nabla) \u+ (\u \cdot \nabla) \overline{\y}) - \nu \Delta \u + \nabla p &=& \f + \h 1_{\Ocal \setminus \overline{\Omega}}  & \text{in } \Ocal_T,
				\\
				\div \u &=& 0 & \text{in } \Ocal_T,
				\\
				\u &=& {\bf 0}  & \text{on } \Gamma_T,
				\\
				\u(0) &=& \u_0 & \text{in } \Ocal.
			\end{array}
		\right.
\end{equation}
Here, $1_{\Ocal \setminus \overline{\Omega}}$ is the characteristic function of $\Ocal \setminus \overline{\Omega}$ and $\h \in \L^2 (\Ocal_T)$ is a control function. Note that  the presence of $1_{\Ocal \setminus \overline{\Omega}}$ in \eqref{StokesControl} implies that the action of the control is supported in $\Ocal \setminus \overline{\Omega}$.
\\
We thus intend to solve the following control problem: Given $\u_0 \in {\bf H}_0^1(\Ocal)$ satisfying \eqref{CCPatu0} and a source term $\f$ in some suitable space, find a control function
$\h \in \L^2 (\Ocal_T)$ such that the solution $\u$ of \eqref{StokesControl} satisfies
\begin{equation}
	\label{ControlReq-u}
	\u(T) = 0 \quad \text{ in } \Ocal.
\end{equation}
Indeed, if we are able to solve this control problem, the restriction of the solution $\u$ to $\Omega$ would yield a solution of the control problem \eqref{SystemparaboliquePF}.
In order to solve the control problem \eqref{StokesControl}--\eqref{ControlReq-u}, as it is classical by now, we are going to establish a suitable observability estimate for the adjoint problem
\begin{equation}
	\label{Stokes-Adjoint-Full}
		\left\{
			\begin{array}{rlll}
				- \partial_t (\overline\sigma \v)  - D (\overline \sigma \v)  \overline\y - \overline\sigma \v \div \overline\y- \nu \Delta \v + \nabla p &=& \g \quad & \text{ in } \Ocal_T,
				\\
				\div \v &=& 0 \quad & \text{ in } \Ocal_T,
				\\
				\v &=& {\bf 0} \quad & \text{ on } \Gamma_T,
			\end{array}
		\right.
\end{equation}
where $D\v:=\nabla\v+\,^t\nabla\v$ is the symmetrized gradient.

To state our result precisely, let us introduce the weight functions we will use in the Carleman estimate. We assume that we have a function $ \tilde \psi = \tilde \psi(t,x) \in C^2(\overline{\Ocal_T})$ such that
\begin{equation}
	\label{Psi}
	\tilde \psi \ov  \tilde \psi(t,x) \quad \hbox{ such that}
	\left\{
		\begin{array}{ll}
			\forall (t,x) \in \overline{\Ocal_T},\, &  \tilde \psi(t,x)  \in [0,1],
			\\
			\forall (t,x) \in \Gamma_T, \, &\partial_{\bf n}  \tilde\psi(t,x) \leq 0,
			\\
			\forall t \in [0,T], \, & \tilde \psi(t)_{|\partial \Ocal}  \text{ is constant},
			\\
			\forall t \in [0,T], \, & \inf_{\Ocal} \tilde \psi(t,\cdot) = \tilde \psi(t)_{|\partial \Ocal}.
		\end{array}
	\right.
\end{equation}
We also assume the existence of two open subsets $\tilde \omega_T \Subset \omega_T$ of $[0,T] \times (\Ocal \setminus \overline{\Omega})$ (here and in the following, the symbol $\Subset$ means that there exists a compact set $K_T$ of $[0,T] \times (\Ocal \setminus \overline{\Omega})$ such that $\tilde \omega_T \subset K_T \subset \omega_T$) and a constant $\alpha>0$ such that
\begin{equation}
	\label{Ass-Psi}
	 \inf_{\Ocal_T \setminus {\tilde \omega_T } }\{|\nabla \tilde \psi| \}\geq \alpha >0.
\end{equation}
For $m\geq 1$, we set
\begin{equation}
	\label{PsiGamma}
	\psi(t,x)\ov  \tilde \psi(t,x) + 6m.
\end{equation}
We then set $T_0>0$ and $T_1>0$ such that $T_1 \leq 1/4$ and $T_0+2 T_1 < T$ and choose a weight function in time $ \theta_{m,\mu}(t)$ depending on the parameters $m \geq 1$ and $\mu\geq 2$ defined by
\begin{equation}
	\label{ThetaMu}
	\theta_{m, \mu } \ov \theta_{m, \mu}(t)
	 \hbox{ such that}
		\left\{
			\begin{array}{l}
			\ds \forall t \in [0,T_0],\, \theta_{m, \mu}(t) = 1+ \left( 1- \frac{t}{T_0} \right)^\mu,
			\smallskip
			\\
			\ds \forall t \in [T_0, T- 2T_1], \, \theta_{m,\mu}(t) = 1,
			\smallskip
			\\
			\ds \forall t \in [T-T_1,T), \,  \theta_{m,\mu}(t) = \frac{1}{(T-t)^m},
			\smallskip
			\\
			\ds \theta_{m,\mu} \hbox{ is increasing on } [T-2T_1, T-T_1],
			\smallskip
			\\
			\ds \theta_{m,\mu} \in C^2([0,T)).
			\end{array}
		\right.
\end{equation}
For simplicity of notations in the following we omit the dependence on $m$ and $\mu$ and we simply write $\theta\ov \theta_{m,\mu}$.
We will then take the following weight functions $\varphi = \varphi(t,x)$ and $\xi = \xi(t,x)$:
\begin{equation}
	\label{Phi-Xi}
	\varphi(t,x)\ov \theta(t) \left(\lambda e^{6 \lambda (m+1)}- \exp(\lambda \psi(t,x)) \right), \quad \xi(t,x) \ov \theta (t) \exp(\lambda \psi(t,x)),
\end{equation}
where  $s,\, \lambda$ are positive parameters with $s\geq 1$, $\lambda\geq 1$ and $\mu$ is chosen as
\begin{equation}
	\label{Def-mu}
	\mu = s \lambda^2 e^{\lambda (6 m -4)},
\end{equation}
which is always bigger than $2$, thus being compatible with the condition $\theta \in C^2([0,T])$. Note that the weight functions $\varphi$ and $\xi$, depend on $s,\, \lambda, \, m$, and should rather be denoted by $\varphi_{s, \lambda,m}$, resp. $\xi_{s,\lambda,m}$, but we drop these indexes for simplicity of notations.
\\
Remark that, due to the definition of $\psi$ in \eqref{PsiGamma} and the conditions \eqref{Psi}, we have, for all $\lambda \geq 1$ and $(t,x) \in \Ocal_T$,
\begin{equation}
	\label{Phi-bounds}
	\frac{3}{4} \theta(t) \lambda e^{ 6 \lambda (m+1)} \leq \varphi(t,x) \leq \theta(t) \lambda e^{6 \lambda (m+1)}.
\end{equation}
Finally, we introduce
\begin{eqnarray}
		\widehat\varphi(t)\ov\min_{x\in\overline\Ocal}\varphi(t,x),
			& \quad &
		\varphi^*(t)\ov \max_{x\in\overline\Ocal}\varphi(t,x)=\varphi_{|\partial\Ocal}(t),
		\label{Variants-Phi}
			\\
		\widehat\xi(t)\ov\max_{x\in\overline\Ocal}\xi(t,x),
			& \quad &
		\xi^*(t)\ov\min_{x\in\overline\Ocal}\xi(t,x) = \xi_{|\partial\Ocal}(t).
		\label{Variants-Xi}
\end{eqnarray}
	
Using these weight functions, we prove the following Carleman estimate for the Stokes system \eqref{Stokes-Adjoint-Full}:
\begin{theorem}
	\label{Thm-Carleman-Stokes}
	{Assume that $\Ocal$ is a smooth bounded domain extending $\Omega$ as in \eqref{ExtensionOfOmega}, let $\omega$, $\Tilde\omega$ be two subdomains of $\Ocal\backslash \overline{\Omega}$ such that $\Tilde\omega\Subset \omega$ and set $\omega_T=[0,T]\times \omega$ and $\tilde \omega_T=[0,T]\times \Tilde \omega$.}
	\\
	Let $\tilde \psi$ as in \eqref{Psi}--\eqref{Ass-Psi} and $\psi,\, \theta, \, \varphi, \, \xi$ as in \eqref{PsiGamma}--\eqref{ThetaMu}--\eqref{Phi-Xi}.
	\\
	Then, for $m\geq 5$, there exist some constants $s_0\geq 1$, $\lambda_0 \geq 1$ and $C>0$ such that for all smooth solution $\v$ of \eqref{Stokes-Adjoint-Full} with source term $\g \in \L^2( \Ocal_T)$, for all $s \geq s_0$ and $\lambda \geq \lambda_0$,
	\begin{equation}\label{Carl-Est-Stokes-Full}
		\begin{array}{l}\displaystyle
			s^{1/2}\lambda^{-1/2}\int_\Ocal (\xi^*)^{4-2/m} |\v(0,\cdot)|^2 e^{-2s\varphi^*(0)}
			+
			s\lambda^2\iint_{\Ocal_T}\xi^{4}|\v|^2e^{-2s\varphi}
		\\
		 \noalign{\medskip}\displaystyle
			+
			s^{-1} \iint_{\Ocal_T}\xi^{2}|\nabla \v|^2e^{-2s\varphi}
			+
			s^{1/2}\lambda^{-1/2}\int_0^T(\xi^*)^{4-2/m}e^{-2s\varphi^*}\|\v\|^2_{{\bf H}^{1}(\Ocal)}
		\\
		 \noalign{\medskip}\displaystyle
			\leq
			C\left(
			s^{5/2}\lambda^{2}\iint_{{{\omega_T}}}\widehat\xi^6|\v|^2e^{2s\varphi^*-4s\widehat\varphi}
			 +s^{1/2}\lambda^{-1/2} \iint_{\Ocal_T} (\xi)^{4-2/m}|\g|^2e^{-2s\varphi}
			\right).
	\end{array}
	\end{equation}
\end{theorem}

The proof of Theorem \ref{Thm-Carleman-Stokes} is done in Sections \ref{Sec-CarlemanHeat} and \ref{Sec-Stokes}. We are first going to prove a slightly improved version of the Carleman estimates \eqref{Carl-Est-Stokes-Full} for solutions $\v$ of the simplified version of the adjoint problem \eqref{Stokes-Adjoint-Full}:
\begin{equation}
	\label{Stokes-Adjoint-Simplified}
		\left\{
			\begin{array}{rl}
				- \overline\sigma \partial_t \v - \nu \Delta \v + \nabla p = \g \quad & \text{ in } \Ocal_T,
				\\
				\div \v = 0 \quad & \text{ in } \Ocal_T,
				\\
				\v = 0 \quad & \text{ on } \Gamma_T.
			\end{array}
		\right.
\end{equation}
Our approach then consists first in taking the curl of the equation \eqref{Stokes-Adjoint-Simplified} and consider the equation of $w = \rot \v$:
\begin{equation}
	\label{Eq-Rot-V}
				- \overline\sigma \partial_t w - \nu \Delta w = \rot \g + \partial_t \v\cdot \nabla^\perp \overline\sigma \quad  \text{ in } \Ocal_T.
\end{equation}
Thus, in Section \ref{Sec-CarlemanHeat}, we derive estimates on $w$ solution of \eqref{Eq-Rot-V} in terms of the right hand side of the equation of \eqref{Eq-Rot-V} and the boundary terms. It turns out that the boundary conditions and source terms strongly depend on $\v$ itself. Hence in Section \ref{Sec-Stokes}, we explain how to estimate $\v$ in terms of $w$ by using the stream function $\zeta$ associated to $\u$, which is given by
\begin{equation}
	\label{Link-Rot-V-Zeta}
		\Delta \zeta(t) = w(t) \quad \text{ in } \Ocal_T \quad \text{ and } \quad \zeta(t) = c_i(t) \quad \displaystyle \text{ on } [0,T]\times \gamma_i\;\text{ for }\; i=1,\dots, K,
\end{equation}
where $\{\gamma_i,\; i=1,\dots, K\}$ is the family of connected components of $\partial\Ocal$ and $c_i(t)$, $i=1,\dots, K$ are some constants characterizing $\zeta(t)$ which are chosen such that, for some Lipschitz subdomain $\widehat{\omega}$ of $\Ocal\backslash \overline{\Omega}$ satisfying $\Tilde\omega\Subset \widehat\omega \Subset \omega$,
\begin{equation}\label{zetazeromean}
\int_{\widehat{\omega}}\zeta(t) =0.
\end{equation}

Among the new features of the Carleman estimate of Theorem \ref{Thm-Carleman-Stokes} with respect to those in the literature, let us point out the following facts:
\begin{itemize}
	\item The weight function in time $\theta_{m,\mu}$ in \eqref{ThetaMu} does not blow up as the time $t$ goes to $0$. However, our proof requires a strong convexity property close to $t = 0$, tuned by the choice of the parameter $\mu$ in \eqref{ThetaMu} as a suitable function of the parameters $s$ and $\lambda$, see \eqref{Def-mu}.
	\item The weight function $\psi$ depends on both the time and space variables. As we shall explain, this is not a big issue as long as we guarantee that for all $t\in [0,T]$, $\psi(t)$ is constant on the boundary $\partial \Ocal$, thus allowing to apply the Carleman inequality of \cite{ImaPuel} for elliptic equations.
\end{itemize}

Based on Theorem \ref{Thm-Carleman-Stokes}, following standard duality arguments, we prove the following control result:
\begin{theorem}
	\label{Thm-Stokes-Controlled}
	Within the setting and assumptions of Theorem \ref{Thm-Carleman-Stokes}, there exists a constant $C>0$ such that for all $s \geq s_0$ and $\lambda \geq \lambda_0$, if $\u_0$ verifies (\ref{CCPatu0}) and $\f \in {\bf L}^2(\Ocal_T)$  satisfies
	\begin{equation}
		\label{ConditionData-Stokes}
		\iint_{\Ocal_T} \xi^{-4} |\f |^2 e^{2 s \varphi} < \infty,
	\end{equation}
 	there exists a control function $\h \in \L^2( \Ocal_T)$ supported in $\omega_T$ and a controlled trajectory $\u \in \L^2( \Ocal_T)$ such that $\u$ solves the control problem \eqref{StokesControl}--\eqref{ControlReq-u} and $(\u, \h)$ satisfies the estimate
	\begin{multline}
		\label{Est-U-H-Duality-Thm}
		\|e^{\frac{3}{4}s\varphi^*}\u\|^2_{{L}^2({\bf H}^2)\cap {H}^1({\bf L}^2)}+s^{1/2}\lambda^{5/2}  \iint_{\Ocal_T} \xi^{2/m - 4} |\u|^2 e^{2 s \varphi} + s^{-3/2} \iint_{\omega_T} \widehat\xi^{-6} |\h|^2 e^{4s \widehat{\varphi} - 2s \varphi^*}
		\\
		\leq
		C\left(\iint_{\Ocal_T} \xi^{-4} |\f |^2 e^{2 s \varphi} + e^{\frac{5}{2} s \varphi^*(0,\cdot)} \|\u_0\|^2_{{\bf H}^1_0(\Ocal)}\right).
	\end{multline}
\end{theorem}
The proof of Theorem \ref{Thm-Stokes-Controlled} is given in Section \ref{Sec-Stokes-Controlled}.

\subsection{Carleman estimates for the heat equation}\label{Sec-CarlemanHeat}
The goal of this section is to show the following estimate:
 \begin{theorem}
 	\label{Thm-Carleman-H-1}
Let $\widehat{\omega_T}$ be an open subset of $\Ocal_T$ satisfying $\tilde{\omega}_T\Subset \widehat{\omega_T}$ and let $\tilde \psi$ as in \eqref{Psi}--\eqref{Ass-Psi} and $\psi,\, \theta, \, \varphi, \, \xi$ as in \eqref{PsiGamma}--\eqref{ThetaMu}--\eqref{Phi-Xi}.
	\\
	For all $M>0$, there exist constants $C>0$, $s_0$ and $\lambda_0$ such that for all $s \geq s_0$ and $\lambda \geq \lambda_0$, for all smooth functions $w $ in $ \overline{\Ocal_T}$, such that
	$$
		-\overline \sigma \partial_t w - \nu \Delta w = a_0 w +  A_1 \cdot \nabla w +g_0 + \sum_{i=1}^n b_i \partial_i  g_{i}  + b_{n+1} \partial_t g_{n+1}  \quad \hbox{ in } 	\Ocal_T,
	$$
	with $a_0 \in L^\infty( \Ocal_T)$, $A_1 \in L^\infty (0,T; {\bf W}^{1,\infty}( \Ocal))$, $g_0, \, g_{i} \in L^2 (\Ocal_T)$, and coefficients $b_{i} \in L^\infty(0,T; W^{1, \infty} (\Ocal))$,  $b_{n+1} \in W^{1, \infty}(0,T ; L^\infty(\Ocal))$ satisfying
	\begin{multline}
		\label{Reg-Bound}
		\|a_0 \|_{L^\infty(\Ocal_T)} + \| A_1 \|_{L^\infty (0,T; {\bf W}^{1,\infty}( \Ocal))}
		\\
		+\sum_{i=1}^n \|b_i\|_{L^\infty(0,T; W^{1, \infty} (\Ocal))} + \|b_{n+1} \|_{W^{1, \infty}(0,T ; L^\infty(\Ocal))}
		\leq M,
	\end{multline}
	we have
	\begin{multline}
		\label{CarlemanWeak}
		s^3 \lambda^4 \iint_{\Ocal_T} \xi^3 |w|^2 e^{-2 s \varphi}
		\leq
		C  \iint_{\Ocal_T} |g_0|^2 e^{-2 s \varphi}			
		\\
		+
		 Cs^2 \lambda^2 \iint_{\Ocal_T} \xi^{2} \big{(}\sum_{i=1}^n |g_i|^2\big{)} e^{-2 s \varphi}
		 +
		C s^4 \lambda^4  \iint_{\Ocal_T} \xi^{4} |g_{n+1}|^2 e^{-2 s \varphi}
		\\
		+
		Cs^3 \lambda^3 \int_{\Gamma_T} \xi^{3} |w|^2 e^{-2 s \varphi}
		 +
		 C s^3 \lambda^4 \iint_{\widehat{\omega_T}} \xi^3 |w|^2 e^{-2 s \varphi}.
	\end{multline}	
 \end{theorem}
The proof of Theorem \ref{Thm-Carleman-H-1} is long and is divided in three steps:
\begin{enumerate}
	\item a Carleman estimate for the heat equation with homogeneous boundary conditions and source terms in $L^2(\Ocal_T)$; see Theorem \ref{CarlemanThm};
	\item energy estimates on controlled trajectories of a heat equation with a source term in $L^2(\Ocal_T)$; see Theorem \ref{Thm-Est-Y-Carl-Norms};
	\item a duality argument.
\end{enumerate}
This proof is inspired by the ones in \cite{ImaYam}, see also \cite{FernandezCaraGuerrero-2006-SICON}. Below, we only state Theorems \ref{CarlemanThm}--\ref{Thm-Est-Y-Carl-Norms}, whose proofs are postponed to the appendix.
\begin{proof}[Proof of Theorem \ref{Thm-Carleman-H-1}]
	As said above, the proof is done in three steps.

{\bf An $L^2$-Carleman estimate.} The first result is the following $L^2$-Carleman estimate for the heat equation:
\begin{theorem}
	\label{CarlemanThm}
	Assume the setting of Theorem \ref{Thm-Carleman-H-1}. For all $m \geq 1$, there exist constants $C_0>0$, $s_0\geq 1$ and $\lambda_0\geq 1$ such that for all smooth functions
$z$ on $ \overline{\Ocal_T}$ satisfying $z = 0$ on $\Gamma_T$, for all $s \geq s_0$, $\lambda \geq \lambda_0$, we have
	\begin{multline}
		\label{CarlemanEst}
			  \int_\Ocal |\nabla z(0)|^2 e^{-2 s \varphi(0)} + s^2 \lambda^3 e^{2\lambda (6m +1)} \int_\Ocal |z(0)|^2 e^{-2 s \varphi(0)}
		  \\
			  + s \lambda^2 \iint_{\Ocal_T}  \xi |\nabla z|^2  e^{-2 s \varphi}
			 \ds + s^3 \lambda^4  \iint_{\Ocal_T} \xi^3 |z|^2 e^{-2 s \varphi} 	
		\\
			  \leq C_{0} \iint_{\Ocal_T} |(- \overline \sigma \partial_t - \nu\Delta) z|^2 e^{-2 s \varphi}
			 + C_{0} s^3 \lambda^4 \iint_{\widehat{\omega_T}} \xi^3 |z|^2 e^{-2 s \varphi}.
	\end{multline}
\end{theorem}
The proof of Theorem \ref{CarlemanThm} is given in Section \ref{Sec-Proof-Carleman-Heat}. It is rather classical except for the weight function $\varphi$, which does not blow up as $t \to 0$ and for the weight function $\psi$ which depends on both time and space variables. This introduces in the proof of Theorem \ref{CarlemanThm} several new technical issues, though our proof follows
the lines of \cite{FursikovImanuvilov}.

{\bf Estimates on a control problem.}
We then analyze the following control problem: for $f \in L^2 (\Ocal_T)$, find a control function $h \in L^2( \widehat{\omega_T})$ such that the solution $y$ of
\begin{equation}
	\label{Heat-control}
	\left\lbrace
		\begin{array}{rlll}
			\partial_t( \overline \sigma y) - \nu\Delta y &= & f  + h {1}_{ \widehat{\omega_T}}, \quad & \hbox{ in } \Ocal_T,
			\\
			y &=& 0, \quad & \hbox{ on } \Gamma_T,
			\\
			y(0,\cdot) &= & 0, \quad & \hbox{ in } \Ocal,
		\end{array}
	\right.
\end{equation}
solves the control problem:
\begin{equation}
	\label{Null-Control-Req}
	y(T, \cdot) = 0, \quad \hbox{in } \Ocal.
\end{equation}
We claim the following result:
\begin{theorem}
	\label{Thm-Est-Y-Carl-Norms}
	Assume the setting of Theorem \ref{Thm-Carleman-H-1}. For all $m\geq 1$, there exist positive constants $C>0$, $s_0 \geq 1$ and $\lambda_0 \geq 1$ such that for all $s \geq s_0$ and $\lambda \geq \lambda_0$, for all $f$ satisfying
	\begin{equation}
		\label{Conditions-f-F}
		\iint_{\Ocal_T} \xi^{-3} |f|^2 e^{2 s \varphi} < \infty,
	\end{equation}
	there exists a solution $(Y,H)$ of the control problem \eqref{Heat-control}--\eqref{Null-Control-Req} which furthermore satisfies the following estimate:
	\begin{multline}
	\label{Est-Y-Gal}
		s^3 \lambda^4  \iint_{\Ocal_T}  |Y|^2 e^{2 s \varphi}+ \iint_{\widehat{\omega_T}} \xi^{-3} |H|^2 e^{2 s \varphi}
		+
		s \lambda^2\iint_{\Ocal_T} \xi^{-2} |\nabla Y|^2 e^{2 s \varphi}+
	\\
		\frac{1}{s} \iint_{\Ocal_T} \xi^{-4} (|\partial_t Y|^2 + |\Delta Y|^2) e^{2 s \varphi}
		+
		\lambda \int_{\Gamma_T} \xi^{-3} |\partial_{\bf n} Y|^2 e^{2s \varphi}
		\leq
		C \iint_{\Ocal_T} \xi^{-3} |f|^2 e^{2 s \varphi}.
	\end{multline}
\end{theorem}
The proof of Theorem \ref{Thm-Est-Y-Carl-Norms} is given in Section \ref{Sec-Proof-Control-Heat}. Again, the proof is rather classical and is based on the duality between the Carleman estimates, which are  weighted observability estimates, and controllability, and then on energy estimates. Note however that these energy estimates have to be derived using the weight functions defined in \eqref{Psi}--\eqref{Phi-Xi}, and this introduces some novelties in the computations.

{\bf A duality argument.}
 	The proof of Theorem \ref{Thm-Carleman-H-1} then relies upon the estimate \eqref{Est-Y-Gal} on the solution $(Y,H)$ of the control problem \eqref{Heat-control}--\eqref{Null-Control-Req} for $f  = \xi^{3}  w e^{-2 s \varphi}$.
	Indeed, if $(Y,H)$ solves \eqref{Heat-control}--\eqref{Null-Control-Req} for some $f$ satisfying \eqref{Conditions-f-F}, multiplying the equation satisfied by $Y$ by $w$, we obtain
	\begin{multline}
		\iint_{\Ocal_T} w (f+ H 1_{\widehat{\omega_T} } ) + \int_{\Gamma_T}  w  \nu \partial_{\bf n} Y
		\\
		=
		\iint_{\Ocal_T} (a_0 w Y - w \div( A_1 Y) + g_0 Y - \sum_{i=1}^n g_i \partial_i ( b_i Y) - g_{n+1} \partial_t (b_{n+1} Y)).
		\label{Identity-Y-V-W}
	\end{multline}
	In particular, as $f = \xi^3w e^{-2 s \varphi} $ satisfies
	$$
		  \iint_{\Ocal_T} \xi^{-3} |f|^2 e^{2 s \varphi} = \iint_{\Ocal_T} \xi^3 |w|^2 e^{-2 s \varphi},
	$$
	according to \eqref{Est-Y-Gal} we can construct $(Y,H)$ solution of
	\begin{equation}
	\label{Heat-control-Y-w}
	\left\lbrace
		\begin{array}{rlll}
			\partial_t (\overline \sigma Y) - \nu\Delta Y &=& \xi^3 w e^{-2s \varphi} + H {1}_{\widehat{\omega_T }}, \quad & \hbox{ in } \Ocal_T,
			\\
			Y &= &0, \quad &  \hbox{ on } \Gamma_T,
			\\
			Y(0,\cdot) &=& 0, \quad & \hbox{ in } \Ocal,
			\\
			Y(T, \cdot) &=& 0, \quad &  \hbox{ in } \Ocal,		
		\end{array}
	\right.
	\end{equation}
	for which we have the estimate:
	\begin{multline}
		\label{Est-Y-W-Gal}
		s^3 \lambda^4 \iint_{\Ocal_T} |Y|^2 e^{2 s \varphi} + \iint_{\widehat{\omega_T }} \xi^{-3} |H|^2 e^{2 s \varphi}
		+
		s \lambda^2 \iint_{\Ocal_T} \xi^{-2} |\nabla Y|^2 e^{2 s \varphi}
		+
		\\
		\frac{1}{s} \iint_{\Ocal_T} \xi^{-4} (|\partial_t Y|^2 + |\nabla Y|^2) e^{2s \varphi}
		+ \lambda \int_{\Gamma_T} \xi^{-3} |\partial_{\bf n} Y|^2 e^{2s \varphi}
		\leq
		C\iint_{\Ocal_T} \xi^3 |w|^2 e^{-2 s \varphi}.
	\end{multline}
	Using then the identity \eqref{Identity-Y-V-W}, we infer
	\begin{eqnarray*}
		\lefteqn{\iint_{\Ocal_T} \xi^3 |w|^2 e^{-2 s \varphi}}
		\\
		& \leq &
		C
		\left(
			\frac{1}{s \lambda^2} \iint_{\Ocal_T} \xi^2 |w|^2 e^{-2 s \varphi}			
		\right)^{1/2}
		\left(
			s \lambda^2 \iint_{\Ocal_T}  \xi^{-2} (|Y|^2 +|\nabla Y|^2)  e^{2s \varphi}	
		\right)^{1/2}
		\\
		& + &
		C
		\left(
			\frac{1}{s^3 \lambda^4} \iint_{\Ocal_T} |g_0|^2 e^{-2 s \varphi}			
		\right)^{1/2}
		\left(
			s^3 \lambda^4 \iint_{\Ocal_T}  |Y|^2 e^{2s \varphi}	
		\right)^{1/2}
		\\
		&+&
		C \left(\frac{1}{s \lambda^2} \iint_{\Ocal_T} \xi^{2} \big{(}\sum_{i=1}^n |g_i|^2\big{)} e^{-2 s \varphi}\right)^{1/2}
		\left( 	s \lambda^2 \iint_{\Ocal_T} \xi^{-2} (|Y|^2 + |\nabla Y|^2) e^{2 s \varphi}  \right)^{1/2}
		\\
		& + &
		C \left( s \iint_{\Ocal_T} \xi^{4} |g_t|^2 e^{-2 s \varphi} \right)^{1/2}
		\left( 	\frac{1}{s} \iint_{\Ocal_T} \xi^{-4} (|Y|^2 + |\partial_t Y|^2) e^{2 s \varphi}  \right)^{1/2}
		\\
		& + &
		C \left(\frac{1}{ \lambda} \int_{\Gamma_T} \xi^{3} |w|^2 e^{-2 s \varphi} \right)^{1/2}
		\left( \lambda \int_{\Gamma_T} \xi^{-3} |\partial_{\bf n} Y|^2 e^{2s \varphi} \right)^{1/2}
		\\
		& + &
		C\left( \iint_{\widehat{\omega_T } } \xi^3 |w|^2 e^{-2 s \varphi} \right)^{1/2}
		\left( \iint_{\widehat{\omega_T }} \xi^{-3} |H|^2 e^{2 s \varphi} \right)^{1/2},
	\end{eqnarray*}
	which immediately yields the claimed result by \eqref{Est-Y-W-Gal}.
\end{proof}

 \subsection{Proof of Theorem \ref{Thm-Carleman-Stokes}}\label{Sec-Stokes}

This section aims at proving Theorem \ref{Thm-Carleman-Stokes}. This will be done in two steps.
\\
We first prove the following Carleman estimate for $\v$ solution of \eqref{Stokes-Adjoint-Simplified}:
\begin{theorem}
	\label{Thm-Stokes-Simplified}
	Within the setting and assumptions of Theorem \ref{Thm-Carleman-Stokes}, for any $m\geq 5$, there exist some constants $s_0\geq 1$, $\lambda_0 \geq 1$ and $C>0$ such that for all solution $\v$ of \eqref{Stokes-Adjoint-Simplified} with source term $\g \in \L^2(\Ocal_T)$, for all $s \geq s_0$ and $\lambda \geq \lambda_0$,
	\begin{multline}\label{Carlemanpour-v-simplified}
		s^{1/2}\lambda^{-1/2} \int_\Ocal (\xi^*)^{4-2/m} |\v(0,\cdot)|^2 e^{-2s\varphi^*(0)}
		+
		s\lambda^2\iint_{\Ocal_T}\xi^{4}|\v|^2e^{-2s\varphi}
		\\
		+
		s^{1/2}\lambda^{-1/2}\int_0^T(\xi^*)^{4-2/m}e^{-2s\varphi^*}\|\v\|^2_{{\bf H}^{1}(\Ocal)}
		+
		\iint_{\Ocal_T}\xi^{3}|\rot \v|^2e^{-2s\varphi} + s^{-1} \iint_{\Ocal_T}\xi^{2}|\nabla \v|^2e^{-2s\varphi}
		\\
		\leq
		C\left(
		s^{5/2}\lambda^{2}\iint_{\omega_T}\widehat\xi^6|\v|^2e^{2s\varphi^*-4s\widehat\varphi}
		\right.
		 +s^{-1}\lambda^{-2}\iint_{\Ocal_T}\xi^2|\g|^2e^{-2s\varphi}
		\\
		+s^{1/2}\lambda^{-1/2}\int_0^T(\xi^*)^{4-2/m}e^{-2s\varphi^*}\|\g\|^2_{{\bf H}^{-1}(\Ocal)}
		\left.
		+s^{-1/2}\lambda^{-3/2}\iint_{\Ocal_T}(\xi^*)^{3-3/m}|\g|^2e^{-2s\varphi^*}
		\right).
	\end{multline}
\end{theorem}
The proof of Theorem \ref{Thm-Stokes-Simplified} is done below in Section \ref{Sec-Proof-Thm-Stokes-Simplified}. In Section \ref{Sec-Proof-Thm-Stokes} we then explain how Theorem \ref{Thm-Stokes-Simplified} implies Theorem \ref{Thm-Carleman-Stokes}.

\subsubsection{Proof of Theorem \ref{Thm-Stokes-Simplified}}\label{Sec-Proof-Thm-Stokes-Simplified}

	Let $\v$ be a solution of \eqref{Stokes-Adjoint-Simplified} with source term $\g$. As $w = \rot \v$ satisfies \eqref{Eq-Rot-V}, the Carleman estimate \eqref{CarlemanWeak} applies to $w$: for all $s \geq s_0$ and $\lambda \geq \lambda_0$,

	\begin{multline}
		\label{Carlemanpourw}
			\iint_{\Ocal_T}\xi^{3}|w|^2e^{-2s\varphi}
			\leq
			C\left(\iint_{{ \widehat{\omega_T }}} \xi^3|w|^2 e^{-2s\varphi}+s  \iint_{\Ocal_T}\xi^4|\v|^2e^{-2s\varphi}\right.
			\\
			\displaystyle\left.
			+\lambda^{-1}\int_{\Gamma_T}\xi^3|w|^2e^{-2s\varphi}+s^{-1}\lambda^{-2}\iint_{\Ocal_T}\xi^2|\g|^2e^{-2s\varphi}
			\right).
	\end{multline}
	{Here and in the following $\widehat{\omega}_T=[0,T]\times \widehat\omega$ where $\widehat{\omega}$ is a Lipschitz subdomain  $\Ocal\backslash \overline{\Omega}$ such that $\widetilde{\omega} \Subset \widehat{\omega}\Subset {\omega}$. Note in particular that $\widetilde{\omega}_T \Subset \widehat{\omega}_T\Subset {\omega_T}$.}

Next, because $\v$ is divergence free we also have, for all $t \in (0,T)$,
	\begin{equation}
		\label{v-by-w-direct}
		-\Delta \v(t) = \rot w(t) \quad \hbox{ in } \Ocal, \qquad \v(t) = 0 \quad \hbox{ on } \partial \Ocal.
	\end{equation}
	Thus, using elliptic Carleman estimates with source term in $H^{-1}(\Ocal)$ with weight $e^{-s \varphi(t,\cdot)}$ and integrating in time,
see \cite{ImaPuel}, we immediately get
	\begin{multline}
		\label{Carleman-v-rot}
		s^{-1} \iint_{\Ocal_T}  \xi^2 |\nabla \v|^2 e^{-2s \varphi}  + s \lambda^2 \iint_{\Ocal_T} \xi^4 |\v|^2 e^{-2 s \varphi}
			\\
			\leq
		C\left(\iint_{\Ocal_T}\xi^{3}|w|^2e^{-2s\varphi}+ s \lambda^2 \iint_{\widehat{\omega_T }} \xi^4 |\v|^2 e^{-2 s \varphi}\right).
	\end{multline}
	Combined with \eqref{Carlemanpourw}, and using the fact that $w=\rot \v$ is bounded by $\partial_{\bf n} \v$ on $\Gamma_T$ (recall that $\v = 0$ on $\Gamma_T$) and that $\xi^*=\xi$ and $\varphi^*=\varphi$ on $(0,T) \times \partial \Ocal$, we immediately have that for some $s_0>1$ and $\lambda_0>1$, for all $s \geq s_0$ and $\lambda \geq \lambda_0$,
	\begin{multline}
		\label{Carleman-W-2}
s^{-1} \iint_{\Ocal_T}  \xi^2 |\nabla \v|^2 e^{-2s \varphi}  +			\iint_{\Ocal_T}\xi^{3}|w|^2e^{-2s\varphi}+ s \lambda^2 \iint_{\Ocal_T} \xi^4 |\v|^2 e^{-2 s \varphi}
			\\
			\leq
			C\left(\iint_{\widehat{\omega_T }} \xi^3|w|^2 e^{-2s\varphi} + s \lambda^2 \iint_{\widehat{\omega_T }} \xi^4 |\v|^2 e^{-2 s \varphi}\right.
			\\
			\left.
			+\lambda^{-1}\int_{\Gamma_T}(\xi^*)^3|\partial_{\bf n}\v|^2e^{-2s\varphi^*}+s^{-1}\lambda^{-2}\iint_{\Ocal_T}\xi^2|\g|^2e^{-2s\varphi}
			\right).
	\end{multline}

	We then introduce the stream function $\zeta$ associated to $\v$, i.e. $\v = \nabla^\perp \zeta$, which can be computed explicitly as the solution of \eqref{Link-Rot-V-Zeta} for some constants $c_i(t)$ due to the dimension $N = 2$, see e.g. \cite[Corollary 3.1]{RAVIART-GIRAULDFEMNSE1986}. Note that, by adding a constant to $\zeta$ if necessary, without loss of generality we can assume that \eqref{zetazeromean} is also satisfied.
Applying the elliptic Carleman estimate to the equation \eqref{Link-Rot-V-Zeta}
(see e.g. \cite{FursikovImanuvilov}),
we obtain that
	\begin{multline}\label{Carlemanpourzeta}
			s^3\lambda^4\iint_{\Ocal_T}\xi^6|\zeta|^2e^{-2s\varphi}+s\lambda^2\iint_{\Ocal_T}\xi^4|\nabla\zeta|^2e^{-2s\varphi}
			\\
			\leq C\left(\iint_{\Ocal_T}\xi^3|w|^2e^{-2s\varphi}+s^3\lambda^4\iint_{\widehat{\omega_T }}\xi^6|\zeta|^2e^{-2s\varphi}\right).
	\end{multline}
	Note that the Carleman estimate of \cite{FursikovImanuvilov} is obtained for homogeneous Dirichlet boundary conditions. But it is easily seen that it remains true for a boundary data whose tangential derivative at the boundary vanishes, which is the case for $\zeta$.
	
	Of course, estimate \eqref{Carlemanpourw} requires an observation term in $\zeta$ in $\widehat{\omega_T }$. But Poincar\'e Wirtinger inequality and condition \eqref{zetazeromean} implies, for all $t \in [0,T]$, 
	\begin{equation*}
		\int_{\widehat{\omega}} |\zeta(t,\cdot)|^2\leq C\int_{\widehat{\omega}}|\nabla\zeta(t,\cdot)|^2
		=\int_{\widehat{\omega}}|\rot \zeta(t,\cdot)|^2=\int_{\widehat{\omega}}|\v(t,\cdot)|^2,
	\end{equation*}
	and in particular:
		\begin{equation}
		\label{Poincare-t-0}
		\iint_{\widehat{\omega_T }}\xi^6|\zeta|^2e^{-2s\varphi}\leq C\iint_{\widehat{\omega_T }}\widehat{\xi}^6|\v|^2e^{-2s\widehat{\varphi}}.
			\end{equation}
Let us stress the fact that the $2$-d assumption is also used at this stage since \eqref{Poincare-t-0} relies on the identity $|\nabla\zeta(t,\cdot)|^2=|\rot \zeta(t,\cdot)|^2$.

	Next, we use \eqref{Carlemanpourzeta} to derive suitable weighted energy estimates for $\v$, hence for $\partial_{\bf n} \v$ on the boundary $\partial \Ocal$. But since we do not have any estimate on the pressure in the Stokes equation \eqref{Stokes-Adjoint-Simplified}, we are reduced to derive energy estimates for $\v$ with weight functions independent of $x$.

{\bf Estimates in $L^2(0,T;{\bf H}^1(\Ocal))$.}
	We set $(\v_a,p_a)\ov\theta_1(t)(\v,p)$ with
	$$
		\theta_1(t)\ov s^{1/4}\lambda^{-1/4}(\xi^*)^{2-1/m}e^{-s\varphi^*(t)}.
	$$
	Using
	\begin{equation}\label{EstDtPhistar}
		 \partial_t\varphi^*  \leq    C \lambda (\xi^*)^{1+1/m}\hbox{ in }\Ocal_T.
	\end{equation}
	and explicit computations, we get
	\begin{equation}
		\label{Estimate-dt-theta-1}
		\theta'_1 \geq - C s^{5/4} \lambda^{3/4} (\xi^*)^3 e^{-s \varphi^*(t)}.
	\end{equation}
	The pair $(\v_a, p_a)$ satisfies
	\begin{equation}
		\label{systemedev*}
		\left\{\begin{array}{rlll}
			-\overline\sigma \partial_t \v_a-\nu \Delta \v_a+\nabla p_a & = &\theta_1 \g-\overline\sigma \theta_1' \v & \hbox{ in } \Ocal_T,
			\\
			\div \v_a&=&0& \hbox{ in }\Ocal_T,
			\\
			\v_a & =& 0 & \hbox{ on }\Gamma_T,
			\\
			\v_a(T)& =& 0 & \hbox{ in }\Ocal.
		\end{array}
	\right.
	\end{equation}
	We want to obtain an estimate of the $L^2(\H^1_0)$-norm of $\v_a$, so we multiply the partial differential
equation in \eqref{systemedev*} by $\v_a$, we integrate in $\Ocal_T$ and we integrate by parts. This yields:
	\begin{multline}
		\label{Estimation-V-1}
		\frac{1}{2} \| \sqrt{\overline\sigma(0, \cdot)} \v_a(0, \cdot)\|_{{\bf L}^2(\Ocal)}^2 + \nu \|\v_a\|^2_{L^2(0,T;{\bf H}^1_0(\Ocal))}
		=
		\iint_{\Ocal_T}\theta_1 \g\cdot\v_a
		\\
		-
		\iint_{\Ocal_T}\overline\sigma \theta'_1\v \cdot \v_a
		-
		\frac{1}{2}\iint_{\Ocal_T}\partial_t\overline\sigma \,|\v_a|^2 .
	\end{multline}
	First, we remark that
	\begin{equation}
		\label{barça}
		\left| \iint_{\Ocal_T}\theta_1 \g \cdot \v_a\right|
		\leq
		\frac{\nu}{4}\iint_{\Ocal_T}|\nabla \v_a |^2
		+
		C \int_0^T |\theta_1|^2\|\g\|_{{\bf H}^{-1}(\Ocal)}^2.
	\end{equation}

	 We then focus on the second term of \eqref{Estimation-V-1} and use \eqref{Estimate-dt-theta-1}
	\begin{align*}
		\lefteqn{-\iint_{\Ocal_T}\overline\sigma \theta_1'\v \cdot \v_a}
		\\
		& \leq
		C s^{3/2} \lambda^{1/2} \iint_{\Ocal_T} (\xi^*)^{5-1/m} \v \cdot \nabla^\perp \zeta   e^{-2s \varphi^*(t)}
		\\
		& =
		- C s^{3/2} \lambda^{1/2} \iint_{\Ocal_T} (\xi^*)^{5-1/m} \rot \v \, \zeta e^{-2s \varphi^*(t)}
		\\
		& \leq C s^{5/2} \lambda^{3/2} \iint_{\Ocal_T} (\xi^*)^6 |\zeta|^2 e^{-2s \varphi^*(t)}
+ \frac{\nu s^{1/2} \lambda^{-1/2}}{4} \iint_{\Ocal_T} (\xi^*)^{4-2/m} |\nabla \v|^2 e^{-2 s
		\varphi^*(t)}
		\\
		& \leq C s^{5/2} \lambda^{3/2} \iint_{\Ocal_T} (\xi^*)^6 |\zeta|^2 e^{-2s \varphi^*(t)} + \frac{\nu }{4} \iint_{\Ocal_T}  |\nabla \v_a|^2.
	\end{align*}
	The last term can be handled similarly:
	\begin{align*}
			\left| \frac{1}{2}\iint_{\Ocal_T}\partial_t\overline\sigma \,|\v_a|^2 \right|
			& \leq
			C s^{1/2} \lambda^{-1/2} \iint_{\Ocal_T} (\xi^*)^{4-2/m} |\v|^2 e^{-2 s \varphi^*}
			\\
			&\leq
			C s^{5/2} \lambda^{3/2} \iint_{\Ocal_T} (\xi^*)^6 |\zeta|^2 e^{-2s \varphi^*(t)} + \frac{\nu }{4} \iint_{\Ocal_T}  |\nabla \v_a|^2.
	\end{align*}

	Plugging these three last estimates in \eqref{Estimation-V-1}, we obtain
	\begin{multline}
		\label{L2H1}
		\|  \v_a(0, \cdot)\|_{{\bf L}^2(\Ocal)}^2 + \|\v_a\|^2_{L^2(0,T;{\bf H}_0^1(\Ocal))}
		\\
		\leq C\left(s^{5/2} \lambda^{3/2}\iint_{\Ocal_T}(\xi^*)^6|\zeta|^2e^{-2s\varphi^*}
		+\|\theta_1 {\bf g}\|_{L^2(0,T;{\bf H}^{-1}(\Ocal))}^2\right).
	\end{multline}

	{\bf Estimate in $L^2(0,T;{\bf H}^2(\Ocal))$.}
	Let us now set $(\v_b,p_b)\ov\theta_2(t)(\v,p)$ with
	$$
		\theta_2(t)\ov  s^{-1/4}\lambda^{-3/4}(\xi^*)^{3/2-3/(2m)}e^{-s\varphi^*(t)},
	$$
	for which explicit computations yield:
	\begin{equation}
		\label{Estimate-dt-theta-2}
		\theta'_2 \geq - C s^{3/4} \lambda^{1/4} (\xi^*)^{\tfrac{5}{2}-\tfrac{1}{2m}} e^{- s \varphi^*}
	\end{equation}
	This pair $(v_b, p_b)$ satisfies
	\begin{equation}\label{systemedev-hat}
		\left\{\begin{array}{rlll}
			-\overline\sigma \partial_t \v_b-\Delta \v_b+ \nabla p_b
			&=&
			\theta_2 \g-\overline\sigma \theta_2' \v & \hbox{ in }\Ocal_T,
				\\
			\div \v_b&=& 0& \hbox{ in }\Ocal_T,
				\\
			\v_b &= & 0 & \hbox{ on }\Gamma_T,
				\\
			\v_b(T)&=& 0 & \hbox{ in }\Ocal.
		\end{array}
		\right.	
	\end{equation}
	We then multiply the partial differential equation in \eqref{systemedev-hat} by $(-\Delta\v_b+\nabla p_b)/\overline{\sigma}$, we integrate in $\Ocal_T$ and we integrate by parts:
	\begin{multline}
			\frac{1}{2}\int_{\Ocal}|\nabla \v_b(0,\cdot)|^2
+\iint_{\Ocal_T}\frac{1}{ \overline{\sigma}} \left|-\Delta \v_b+\nabla p_b\right|^2
%
	\\
	 = 	
			\iint_{\Ocal_T}\frac{\theta_2}{\overline{\sigma}} \g\,(-\Delta\v_b+\nabla p_b)
			-\iint_{\Ocal_T}\theta_2 \theta_2' |\nabla \v|^2.
			\label{estimationvhat}
	\end{multline}
Using \eqref{Positivity} we can estimate the first term as follows:
	\begin{equation}\label{estimationg}
		\left|\iint_{\Ocal_T}\frac{\theta_2}{\overline{\sigma}} \g\,(-\Delta\v_b+\nabla p_b)\right|
		\leq \frac{1}{4}\iint_{\Ocal_T}\frac{1}{\overline{\sigma}}|-\Delta\v_b+\nabla p_b \,|^2+C \|\theta_2 \g\|^2_{{\bf L}^2(\Ocal_T)}.
	\end{equation}
	For the second term, remark that by \eqref{Estimate-dt-theta-2}, we have
	$$
		\theta_2\theta_2' \geq - C s^{1/2} \lambda^{-1/2} (\xi^*)^{4-2/m} e^{- 2 s \varphi^*} = - C \theta_1^2,
	$$
	thus yielding
	$$
		-\iint_{\Ocal_T}\theta_2 \theta_2' |\nabla \v|^2 \leq C \norm{ \theta_1 \v}_{L^2(0,T; {\bf H}^1(\Ocal))}^2 =
C \norm{ \v_a}_{L^2(0,T; {\bf H}^1(\Ocal))}^2.
	$$
	
	Therefore, using the above estimate and \eqref{estimationg} into \eqref{estimationvhat}, we obtain
	\begin{eqnarray}
		\| \v_b \|^2_{L^2(0,T;{\bf H}^2(\Ocal))}
		&\leq & C\iint_{\Ocal_T}|-\Delta \v_b+\nabla p_b\,|^2
		\notag\\
		& \leq &
		C\left(\|\theta_2 \g\|^2_{{\bf L}^2(\Ocal_T)}+\|\v_a\|^2_{L^2(0,T;{\bf H}^1(\Ocal))}\right),
		\label{L2H2}
	\end{eqnarray}
	where we have used the classical $H^2$-estimate for the stationary Stokes system, see e.g. \cite[Theorem IV.5.8]{Boyer-Fabrie-Book}.

	{\bf Global Estimate on $\v$ and its normal derivative.}
	Since $\v = 0$ on $\Gamma_T$, classical estimates yield
	$$
		\norm{\partial_{\bf n} \v(t,\cdot)}_{{\bf L}^2(\partial \Ocal)}^2 \leq C \left(\norm{ \v(t,\cdot)}_{{\bf H}_0^1(\Ocal)}  \norm{ \v(t,\cdot)}_{{\bf H}^2(\Ocal)} + \norm{ \v(t,\cdot)}_{{\bf H}_0^1(\Ocal)}^2\right),
	$$
	and in particular, using the fact that $\theta_2 (t)\leq \theta_1(t)$ for all $t \in (0,T)$,
	\begin{multline*}
		\norm{\lambda^{-1/2} (\xi^*)^{\tfrac{7}{4} - \tfrac{5}{4m}} \partial_{\bf n} \v e^{-s \varphi^*}(t, \cdot)}_{{\bf L}^2(\partial\Ocal)}^2
		\\
		\leq C \left(\norm{ \theta_1 \v(t,\cdot)}_{{\bf H}_0^1(\Ocal)}  \norm{ \theta_2 \v(t,\cdot)}_{{\bf H}^2(\Ocal)} + \norm{\theta_1 \v(t,\cdot)}_{{\bf H}_0^1(\Ocal)}^2\right).
	\end{multline*}
	Putting together (\ref{L2H1}) and (\ref{L2H2}) with this last estimate, using \eqref{Carlemanpourzeta} and \eqref{Poincare-t-0} to estimate the term in $\zeta$ and taking into account that $m\geq 5$, we deduce that
	\begin{multline}\label{estimationderiveenormale}
		\|  \v_a(0, \cdot)\|_{{\bf L}^2(\Ocal)}^2 +\|\theta_1 \v\|^2_{L^2(0,T;{\bf H}_0^1(\Ocal))}+\|\theta_2 \v\|^2_{L^2(0,T;{\bf H}^2(\Ocal))}+
		\lambda^{-1}\norm{ (\xi^*)^{3/2} \partial_{\bf n} \v e^{-s \varphi^*}}_{{\bf L}^2(\Gamma_T)}^2
		\\
		\leq C\left(s^{-1/2}\lambda^{-5/2}\iint_{\Ocal_T}\xi^3|w|^2 e^{-2s\varphi}+s^{5/2}\lambda^{3/2}\iint_{\widehat{\omega_T}}\widehat \xi^6|\v|^2 e^{-2s\widehat \varphi}\right.
		\\	
		\left.+\|\theta_1 \g\|_{L^2(0,T;{\bf H}^{-1}(\Ocal))}^2+\|\theta_2  \g\|^2_{{\bf L}^2(\Ocal_T)}\right).
	\end{multline}

	{\bf Elimination of the boundary term.}
	We come back to the Carleman inequality \eqref{Carleman-W-2} and we combine it with \eqref{estimationderiveenormale}: for $s$ large enough,
\begin{multline}
		\label{Carleman-W-2-bis}
		\|  \v_a(0, \cdot)\|_{{\bf L}^2(\Ocal)}^2 +\|\theta_1 \v\|^2_{L^2(0,T;{\bf H}_0^1(\Ocal))}+\|\theta_2 \v\|^2_{L^2(0,T;{\bf H}^2(\Ocal))}\\
s^{-1} \iint_{\Ocal_T}  \xi^2 |\nabla \v|^2 e^{-2s \varphi}  +			\iint_{\Ocal_T}\xi^{3}|w|^2e^{-2s\varphi}+ s \lambda^2 \iint_{\Ocal_T} \xi^4 |\v|^2 e^{-2 s \varphi}
			\\
			\leq
			C\left(\iint_{\widehat{\omega_T }} \xi^3|w|^2 e^{-2s\varphi} + s^{5/2}\lambda^{2}\iint_{\widehat{\omega_T}}\widehat \xi^6|\v|^2 e^{-2s\widehat \varphi}\right.
			\\
			\left.
			+\|\theta_1 \g\|_{L^2(0,T;{\bf H}^{-1}(\Ocal))}^2+\|\theta_2  \g\|^2_{{\bf L}^2(\Ocal_T)}+s^{-1}\lambda^{-2}\iint_{\Ocal_T}\xi^2|\g|^2e^{-2s\varphi}
			\right).
	\end{multline}

{\bf Removing the observation on $w$.}
We now estimate the local term in $|w|^2$. For this purpose, we recall that $\widehat {\omega_T}=[0,T]\times \widehat{\omega} \Subset \omega_T=[0,T] \times \omega$ and we consider a positive function $\chi\in C^2(\overline\Ocal)$ such that
$$
		\chi=1\hbox{ in } \widehat{\omega},
		\qquad
		\chi=0 \hbox{ in }\Ocal \setminus \omega.
$$
Using
\begin{equation}
	\label{Comparison-w}
		\iint_{\widehat{\omega}_T}\xi^3|w|^2 e^{-2s\varphi}	\leq \iint_{\widehat{\omega}_T}\widehat\xi^3|w|^2 e^{-2s\widehat\varphi},
\end{equation}
we are reduced to estimate the right hand side of \eqref{Comparison-w}:
\begin{align*}
	\iint_{\widehat{\omega}_T}\widehat\xi^3|w|^2 e^{-2s\widehat\varphi}
	&	\leq \iint_{{\omega_T}}\chi \widehat\xi^3|w|^2 e^{-2s\widehat\varphi}
		\leq \iint_{{\omega_T}}\chi \widehat\xi^3 |\nabla \v|^2 e^{-2s\widehat\varphi}
	\notag\\
	&
		= -\iint_{{\omega_T}}\chi \widehat\xi^3 \Delta \v\,\v e^{-2s\widehat\varphi}
		+\frac{1}{2}\iint_{{\omega_T}}\Delta\chi \widehat\xi^3  \,|\v|^2 e^{-2s\widehat\varphi}.
	\notag\\
	& \leq
		\varepsilon s^{-1/2} \lambda^{-3/2}\iint_{\Ocal_T}(\xi^*)^{3-3/m}|\Delta \v|^2e^{-2s\varphi^*}
	\notag\\
	& \qquad \qquad +C_\varepsilon s^{1/2} \lambda^{3/2} \iint_{{\omega_T}}(\xi^*)^{-3+3/m}\widehat\xi^6 |\v|^2e^{2s\varphi^*-4s\widehat\varphi},
\end{align*}
where the last estimate follows from Young's identity and where $\varepsilon>0$.

Using the last above inequality in \eqref{Carleman-W-2-bis} with $\varepsilon$ small enough and recalling the definition of $\theta_2$, we get in particular
\begin{multline}\label{Carlemanpourw2}
	\displaystyle s^{1/2}\lambda^{-1/2}\int_\Ocal (\xi^*)^{4-2/m} |\v(0,\cdot)|^2 e^{-2s\varphi^*}+s^{1/2} \lambda^{-1/2}\int_0^T(\xi^*)^{4-2/m}e^{-2s\varphi^*}\|\v\|^2_{{\bf H}^{1}(\Ocal)}\\
	\displaystyle s^{-1} \iint_{\Ocal_T}  \xi^2 |\nabla \v|^2 e^{-2s \varphi}+s\lambda^2\iint_{\Ocal_T}\xi^{4}|\v|^2e^{-2s\varphi}
		\\ \noalign{\medskip}\displaystyle
	+\iint_{\Ocal_T}\xi^{3}|\rot \v|^2e^{-2s\varphi}
	\leq
	C\left(
	s^{5/2}\lambda^{2}\iint_{{{\omega_T}}}\widehat\xi^6|\v|^2e^{2s\varphi^*-4s\widehat\varphi}
	\right.
		\\ \noalign{\medskip}\displaystyle
	\left. +s^{-1}\lambda^{-2}\iint_{\Ocal_T}\xi^2|\g|^2e^{-2s\varphi}
	+s^{1/2}\lambda^{-1/2}\int_0^T(\xi^*)^{4-2/m}e^{-2s\varphi^*}\|\g\|^2_{{\bf H}^{-1}(\Ocal)}
	\right.
		\\ \noalign{\medskip}\displaystyle
	\left.
	+s^{-1/2}\lambda^{-3/2}\iint_{\Ocal_T}(\xi^*)^{3-3/m}|\g|^2e^{-2s\varphi^*}
	\right).
\end{multline}
This concludes the proof of Theorem \ref{Thm-Stokes-Simplified}.

\subsubsection{Proof of Theorem \ref{Thm-Carleman-Stokes}}\label{Sec-Proof-Thm-Stokes}

	Let $\v$ be a smooth solution of \eqref{Stokes-Adjoint-Full} with source term $\g$. Then $\v$ is a solution of \eqref{Stokes-Adjoint-Simplified} with source term
	$$
		\tilde \g = \g  + \partial_t \overline{\sigma}\,  \v +  D (\overline{\sigma} \v) \overline{\y} + \overline{\sigma} \v \div (\overline{\y}).
	$$
	Applying Theorem \ref{Thm-Stokes-Simplified} to $\v$ with source term $\tilde \g$, for all $s \geq s_0$ and $\lambda \geq \lambda_0$ we get
	\begin{multline}
	\label{Carlemanpour-v-simplified-Bis}
		s^{1/2}\lambda^{-1/2}\int_\Ocal (\xi^*)^{4-2/m} |\v(0,\cdot)|^2 e^{-2s\varphi^*(0)}
		+
		s\lambda^2\iint_{\Ocal_T}\xi^{4}|\v|^2e^{-2s\varphi}
		\\
		+
		s^{1/2}\lambda^{-1/2}\int_0^T(\xi^*)^{4-2/m}e^{-2s\varphi^*}\|\v\|^2_{{\bf H}^{1}(\Ocal)}
		+
		\iint_{\Ocal_T}\xi^{3}|\rot \v|^2e^{-2s\varphi}
		+
		 s^{-1} \iint_{\Ocal_T}\xi^{2}|\nabla \v|^2e^{-2s\varphi}
		\\
			\leq
		C\left(
		s^{5/2}\lambda^{2}\iint_{{{\omega_T}}}\widehat\xi^6|\v|^2e^{2s\varphi^*-4s\widehat\varphi}
		\right.
		 +s^{-1}\lambda^{-2}\iint_{\Ocal_T}\xi^2|\tilde \g|^2e^{-2s\varphi}
		\\
		+
		s^{1/2}\lambda^{-1/2}\int_0^T(\xi^*)^{4-2/m}e^{-2s\varphi^*}\|\tilde \g\|^2_{{\bf H}^{-1}(\Ocal)}
		\left.
		+
		s^{-1/2}\lambda^{-3/2}\iint_{\Ocal_T}(\xi^*)^{3-3/m}|\tilde \g|^2e^{-2s\varphi^*}
		\right)
	\end{multline}
	and we are thus reduced to estimate the last terms of the inequality.
	
	But we have
	\begin{align*}
			& s^{-1}\lambda^{-2}\iint_{\Ocal_T}\xi^2|\tilde \g|^2e^{-2s\varphi}
			\leq
			C \big{(} s^{-1}\lambda^{-2}\iint_{\Ocal_T}\xi^2|\g|^2e^{-2s\varphi}
			\\
		&\qquad \qquad
			+
			s^{-1}\lambda^{-2}\iint_{\Ocal_T}\xi^2|\v|^2e^{-2s\varphi}
			+
			s^{-1}\lambda^{-2}\iint_{\Ocal_T}\xi^2|\nabla \v|^2e^{-2s\varphi}\big{)},
		\\
		&
			s^{-1/2}\lambda^{-3/2}\iint_{\Ocal_T}(\xi^*)^{3-3/m}|\tilde \g|^2e^{-2s\varphi^*}
			\leq
			C \big{(}s^{-1/2}\lambda^{-3/2}\iint_{\Ocal_T}(\xi^*)^{3-3/m}|\g|^2e^{-2s\varphi^*}
			\\
		&\qquad \qquad
			+
			s^{-1/2}\lambda^{-3/2}\iint_{\Ocal_T}(\xi^*)^{3-3/m}|\v|^2e^{-2s\varphi^*}
			+
			s^{-1/2}\lambda^{-3/2}\iint_{\Ocal_T}(\xi^*)^{3-3/m}|\nabla \v|^2e^{-2s\varphi^*}\big{)},
	\end{align*}
	in which all the terms in $\v,\, \nabla\v$ can be absorbed by the left-hand side of \eqref{Carlemanpour-v-simplified-Bis} for $s$ and $\lambda$ large enough.
	
	We also have, for all $t \in (0,T)$,
	$$
		\| \tilde \g(t) \|_{{\bf H}^{-1}(\Ocal)}^2 \leq C \|  \g(t,\cdot) \|_{{\bf L}^2(\Ocal)}^2  + C \| \v(t,\cdot)\|_{{\bf L}^2(\Ocal)}^2.
	$$
	Hence
	\begin{multline}
		s^{1/2}\lambda^{-1/2}\int_0^T(\xi^*)^{4-2/m}e^{-2s\varphi^*}\|\tilde \g\|^2_{{\bf H}^{-1}(\Ocal)}
		\leq
		C s^{1/2}\lambda^{-1/2}\iint_{\Ocal_T} (\xi^*)^{4-2/m}e^{-2s\varphi^*}|\g|^2
		\\
		+
		C s^{1/2}\lambda^{-1/2}\iint_{\Ocal_T} (\xi^*)^{4-2/m}e^{-2s\varphi^*}|\v|^2.
	\end{multline}
	Plugging these last estimates in \eqref{Carlemanpour-v-simplified-Bis}, we obtain \eqref{Carl-Est-Stokes-Full}
for $s$ and $\lambda$ large enough.

\subsection{Proof of Theorem \ref{Thm-Stokes-Controlled}}\label{Sec-Stokes-Controlled}

 	We use the following simplified form of \eqref{Carl-Est-Stokes-Full}: for all $s \geq s_0$ and $\lambda \geq \lambda_0$ and all smooth solutions $\v$ of \eqref{Stokes-Adjoint-Full} with source term $\g$:
	\begin{equation}\label{Carl-Est-Stokes-Full-bis}
		\begin{array}{l}\displaystyle
			\int_\Ocal (\xi^*)^{4-2/m} |\v(0,\cdot)|^2 e^{-2s\varphi^*(0)}
			+s^{1/2}\lambda^{5/2}\iint_{\Ocal_T}\xi^{4}|\v|^2e^{-2s\varphi}
			\\ \noalign{\medskip}\displaystyle
			\leq
			C\left(
			s^{2}\lambda^{5/2}\iint_{{{\omega_T}}}\widehat\xi^6|\v|^2e^{2s\varphi^*-4s\widehat\varphi}
			 + \iint_{\Ocal_T}\xi^{4- 2/m} |\g|^2e^{-2s\varphi}
			\right).
	\end{array}
	\end{equation}
	Easy density arguments then show that this result extends to all solutions $\v$ of \eqref{Stokes-Adjoint-Full} with source term $\g \in L^2(\Ocal_T)$ and final data $\v(T) = \v_T \in \V^1_0(\Omega) $.
	
	We then follow the proof of Theorem \ref{Thm-Est-Y-Carl-Norms} and introduce the functional $J_{St}$ defined by
	\begin{multline}
		\label{Def-J-Stokes}
		J_{St}(\v_T, \g) \ov \frac{1}{2} \iint_{\Ocal_T}\xi^{4- 2/m} |\g|^2e^{-2s\varphi}+ \frac{s^{2}\lambda^{5/2}}{2} \iint_{{{\omega_T}}}\widehat\xi^6|\v|^2e^{2s\varphi^*-4s\widehat\varphi}
		\\
		-\iint_{\Ocal_T} \f \cdot \v - \int_\Ocal \u_0( \cdot) \cdot \v (0, \cdot),
	\end{multline}
defined for data $(\v_T, \g) \in \V^1_0(\Omega) \times L^2(\Ocal_T)$, where $\v$ solves \eqref{Stokes-Adjoint-Full} with $\v(T) = \v_T$.

	We then need to define the functional $J_{St}$ on the set $\X_{St,obs} \ov\overline{\X_{St,obs}^0}^{\norm{\cdot}_{St, obs}}$, where
	\begin{equation}
		\label{Def-X-St-Obs}
		\X_{St,obs}^0 \ov \big{\{}(\v_T, \g) \in \V^1_0(\Omega) \times L^2(\Ocal_T)\}
	\end{equation}
	and the norm $\norm{( \v_T,\g) }_{St, obs}$ is defined by
	$$
	 \norm{(\v_T,\g) }_{St, obs}^2 \ov \iint_{\Ocal_T}\xi^{4- 2/m}  |\g |^2 e^{-2s\varphi}+ s^{2}\lambda^{5/2}\iint_{{{\omega_T}}}\widehat\xi^6|\v|^2e^{2s\varphi^*-4s\widehat\varphi},
	$$
	where $\v$ is the corresponding solution to \eqref{Stokes-Adjoint-Full}.

	According to \eqref{Carl-Est-Stokes-Full-bis}, the functional $J_{St}$ can be extended by continuity on $\X_{St, obs}$ if $\f$ satisfies \eqref{ConditionData-Stokes}. The functional $J_{St}$ then has a unique minimizer on $\X_{St, obs}$, that we denote $(\V_T,\G)$ and corresponds to a solution $\V$ of \eqref{Stokes-Adjoint-Full}. We get, for all smooth solution $\v$ of \eqref{Stokes-Adjoint-Full} corresponding to a source term $\g$,
	\begin{multline}
		0 =  \iint_{\Ocal_T}\xi^{4- 2/m} \G \cdot \g e^{-2s\varphi}+ s^{2}\lambda^{5/2} \iint_{{{\omega_T}}}\widehat\xi^6\V \cdot \v e^{2s\varphi^*-4s\widehat\varphi}
		\\
		-\iint_{\Ocal_T} \f \cdot \v - \int_\Ocal \u_0(\cdot) \cdot \v (0, \cdot).
		\label{Weak-Formulation-Stokes}
	\end{multline}
	In particular, setting
	\begin{equation}
		\label{Identification-U-H}
		\u =  \xi^{4-2/m} \G e^{-2s \varphi}, \qquad \h = -s^{2}\lambda^{5/2} \widehat\xi^6 \V e^{2s \varphi^*-4s \widehat{\varphi}} 1_{\omega_T},
	\end{equation}
	we obtain a solution in the sense of transposition of the control problem \eqref{StokesControl}--\eqref{ControlReq-u} with a control term acting only on $\omega_T$.
	
	Besides, using again the Carleman estimate \eqref{Carl-Est-Stokes-Full-bis} and the fact that $J_{St}(\V_T, \G) \leq J_{St} (0,0) = 0$, one immediately derives that
	\begin{equation}
		\label{Bound-V-obs}
		\norm{(\V_T, \G)}_{obs}^2
		\leq
		\frac{C}{s^{1/2} \lambda^{5/2}} \iint_{\Ocal_T} \xi^{-4} |\f |^2 e^{2 s \varphi} +  C\int_\Ocal (\xi^*)^{2/m-4} |\u_0|^2 e^{2 s \varphi^*(0)}.
	\end{equation}
	Hence, using \eqref{Identification-U-H}, the controlled trajectory $(\u, \h)$ satisfies
	\begin{multline}
		\label{Est-U-H-Duality}
		\iint_{\Ocal_T} \xi^{2/m - 4} |\u|^2 e^{2 s \varphi} + \frac{1}{s^{2}\lambda^{5/2}} \iint_{\omega_T} \widehat\xi^{-6} |\h|^2 e^{4s \widehat{\varphi} - 2s \varphi^*}
		\\
		\leq
		\frac{C}{s^{1/2} \lambda^{5/2}} \iint_{\Ocal_T} \xi^{-4} |\f |^2 e^{2 s \varphi} +  \int_\Ocal (\xi^*)^{2/m-4} |\u_0|^2 e^{2 s \varphi^*(0)}.
	\end{multline}
		
	Finally, we can then derive $H^1(\L^2) \cap L^2(\H^2)$ estimates on $\u$ by applying regularity results for Stokes equations to the system satisfied by $e^{\frac{3}{4} s \varphi^*} \u$. The computations are left to the reader.
%
%
\section{Controlling the density}\label{Sec-Density}
%
%
This section is devoted to explain how to solve the control problem \eqref{EqTransport}. As we said in the introduction, the main difficulty is that we need to provide a controlled trajectory that can be estimated with the use of the weight functions introduced in Section \ref{Sec-Velocity}.
\subsection{Basic properties of the flow}\label{Subsec-Flows}
Let $\overline{\y}$ be the extension of $\overline{\y}$ on $[0,T] \times \R^2$ and $\overline{X}$  the corresponding flow, defined in \eqref{Def-FlowBAR}. As $\overline{\y} \in {\bf C}^2([0,T] \times \R^2)$, the flow $\overline{X}$ is continuous with respect to the variables $(t, \tau, x) \in [0,T]^2 \times \R^2$.
\\
We first discuss the stability of property \eqref{OmegaEqOmegaOut}:
\begin{lemma}\label{LemSortieEps}
	Assume that $\overline{\y} \in {\bf C}^2([0,T] \times \R^2)$, and that the flow $\overline{X}$ defined by \eqref{Def-FlowBAR} satisfies \eqref{OmegaEqOmegaOut}.
	\\
	There exist $\varepsilon>0$, $T_0^* >0$ and $T_1^*>0$ such that for all $T_0 \in (0,T_0^*)$, for all $T_1 \in (0,T_1^*)$ and for all $x \in \overline{\Omega}$, there exists $t \in [T_0, T - 2 T_1]$ such that $d(\overline{X}(t, T_0, x), \Omega) \geq 2 \varepsilon$.
\end{lemma}
\begin{proof}
	The proof is done by contradiction. Assume it is false. Then for all $\varepsilon >0$, there exist $T_0^\varepsilon >0$ and $T_1^\varepsilon$ such that $T_0^{\varepsilon},\, T_1^\varepsilon$ converge to $0$ as $\varepsilon \to 0$, and an $x_\varepsilon$ in $\overline{\Omega}$ such that
	\begin{equation}
		\label{Contradiction-flow}
		\forall t \in [T_0^\varepsilon, T - 2 T_1^\varepsilon], \quad d(\overline{X}(t, T_0^\varepsilon, x_\varepsilon), \Omega) < 2\varepsilon.
	\end{equation}
	But $x_\varepsilon$ is bounded in $\overline{\Omega}$. Hence, up to a subsequence, it converges to some $\overline{x}$ in $\overline{\Omega}$. As the flow $\overline{X}$ is continuous in $[0,T]^2 \times \R^2$ and the distance function is continuous, for each $t \in (0, T)$, one could then pass to the limit in \eqref{Contradiction-flow}:
	$$
		\forall t \in (0, T) , \quad d(\overline{X}(t, 0, \overline{x}), \Omega) = 0.
	$$
	This is of course in contradiction with \eqref{OmegaEqOmegaOut}.
\end{proof}
For $\widehat{\bf u}\in L^2(0,T;{\bf H}^2(\mathbb{R}^2))$ we denote by $\widehat X$ the flow defined by
\begin{equation}\label{FlowXannexe}
	\partial_t \widehat X(t,\tau,x)=(\overline {\bf y}+\widehat{\bf u})(t,  \widehat X(t,\tau,x)),\quad \widehat X(\tau,\tau,x)=x.
\end{equation}
We then show that, provided $\widehat \u$ is small enough, the property \eqref{OmegaEqOmegaOut} also holds for $\widehat X$:
\begin{lemma}\label{SortieEps2}
	Under the setting of Lemma \ref{LemSortieEps}, there exists $\varsigma>0$ such that for all $\widehat \u \in L^2(0,T;{\bf H}^2(\mathbb{R}^2))$, satisfying
	\begin{equation}
		\| \widehat {\bf u} \|_{L^2(0,T;{\bf L}^\infty(\mathbb{R}^2))}\leq 2\varsigma,
		\label{hypR}
	\end{equation}
	the flow $\widehat X$ defined by \eqref{FlowXannexe} satisfies the following property: for all $T_0 \in (0,T_0^*)$, for all $T_1 \in (0,T_1^*)$ and for all $x \in \overline{\Omega}$, there exists $t \in [T_0, T -  2T_1]$ such that $d({\widehat X}(t, T_0, 	x), \Omega) \geq \varepsilon$.
\end{lemma}
\begin{proof}
Set $L=\|\nabla \overline {\bf y} \|_{L^\infty(0,T;{\bf L}^\infty(\Omega))}$. For $\tau, t \in [0,T]^2$ with $t \geq \tau$ and $x \in \R^2$, we have:
\begin{align*}
	&
	|\widehat X(t, \tau,x)-\overline{X}(t, \tau,x)| = | \widehat X(t, \tau, x) - \widehat X(\tau, \tau, x) + \overline{X}(\tau, \tau, x) - \overline{X}(t, \tau, x)|
	\\
	&
	=
	\left|
		\int_\tau^t \left (\partial_t \widehat X(t',\tau,x)-\partial_t \overline{X}(t',\tau,x)\right){\rm d}t'	
	\right|
	\\
	&
	=
	\left|
		\int_\tau^t \widehat {\bf u}(t',\widehat X(t',\tau,x))+\overline {\bf y}(t',  \widehat X(t',\tau,x))-\overline {\bf y}(t',  \overline X(t',\tau,x)) {\rm d}t'
	\right|
	\\
	&
	\leq
	|t-\tau|^{1/2} \|\widehat {\bf u}\|_{L^2(\tau,t;{\bf L}^\infty(\R^2))}
	+
	L \int_\tau^t |\widehat X(t',\tau,x)-\overline{X}(t',\tau,x)| {\rm d}t'.
\end{align*}
Then Gronwall's Lemma yields for all $t\in [0,T]$ and $x \in \R^2$:
\begin{equation}\label{EstXbarXaux}
	|\widehat X(t,\tau,x)-\overline{X}(t,\tau,x)| \leq T^{1/2} e^{L T} \|\widehat {\bf u}\|_{L^2(\tau,t;{\bf L}^\infty(\R^2))}.
\end{equation}
According to Lemma \ref{LemSortieEps}, Lemma \ref{SortieEps2} thus holds by setting $\varsigma= T^{-1/2} e^{-L T}\varepsilon/2$ in \eqref{hypR}.
\end{proof}

\subsection{Construction of the controlled density}\label{Sec-ControlledDensity}
%
In this section, we assume that
\begin{equation}
	\label{Ass-On-U}
	\widehat\u \in L^2(0,T; {\bf H}^2(\R^2)) \hbox{ and } \norm{\widehat \u }_{L^2(0,T; {\bf L^\infty}(\R^2))} \leq 2 \varsigma,
\end{equation}
where $\varsigma$ is given by Lemma \ref{SortieEps2}. We then choose $T_0 \in (0,T_0^*)$ and $T_1 \in (0,T_1^*)$, where $T_0^*$, $T_1^*$ are given by Lemma \ref{SortieEps2}.

The construction of the controlled density $\rho$ solution of \eqref{EqTransport} is then done as in \cite{EGGP}: we construct a forward solution $\rho_f$
and a backward solution $\rho_b$ of the transport equation in \eqref{EqTransport}
and we glue these two solutions according to the characteristics of the flow.
\\
Indeed, we define $\rho_f$ as the solution of
\begin{equation}\label{Eq-rho-f}
	 \left\{
	\begin{array}{rlll}
		\partial_t\rho_f+(\overline {\bf y}+\widehat {\bf u})\cdot \nabla \rho_f &=&- \widehat {\bf u}\cdot \nabla\overline{\sigma }&\mbox{ in } \Omega_T,
		\\
	        	\rho_f(t,x) &=& 0 &\mbox{ for } t \in (0,T),\, x \in \partial \Omega,
		\\
		& & &  \hbox{ with }\left( \overline {\bf y}(t,x)+\widehat {\bf u}(t,x)\right)\cdot {\bf n}(x) <0,
		\\
        		\rho_f(0)&=& \rho_0\quad &\mbox{ in }\Omega,
	\end{array}
	\right.
\end{equation}
and $\rho_b$ as the solution of
\begin{equation}\label{Eq-rho-b}
	 \left\{
	\begin{array}{rlll}
		\partial_t\rho_b+(\overline {\bf y}+\widehat {\bf u})\cdot \nabla \rho_b &=& - \widehat {\bf u}\cdot \nabla\overline{\sigma }&\mbox{ in } \Omega_T,
		\\
	        \rho_b&=& 0&\mbox{ for } t \in (0,T),\, x \in \partial \Omega,
	        \\
	        & & & \hbox{ with }\left( \overline {\bf y}(t,x)+\widehat {\bf u}(t,x)\right)\cdot {\bf n}(x) >0,
	        \\
	        \rho_b(T)&=& 0\quad &\mbox{ in }\Omega.
	\end{array}
	\right.
\end{equation}
We also introduce $\chi$ the solution of
\begin{equation}\label{Eq-chi}
	 \left\{
	\begin{array}{rlll}
		\partial_t\chi+(\overline {\bf y}+\widehat {\bf u})\cdot \nabla \chi &=&0&\mbox{ in } \Omega_T,
		\\
	        \chi&=& 1_{t \in (0,T_0)}(t)&\mbox{ for } t \in (0,T),\, x \in \partial \Omega,
	        \\
	        & & & \hbox{ with }\left( \overline {\bf y}(t,x)+\widehat {\bf u}(t,x)\right)\cdot {\bf n}(x) <0,
	        \\
	        \chi(0)&= &1\quad &\mbox{ in }\Omega.
	\end{array}
	\right.
\end{equation}

We finally define $\rho(t,x)$ as follows,
\begin{equation}
	\label{def-rho}
	\rho(t,x)\ov (1-\chi(t,x)) \rho_b(t,x)+\chi(t,x)\rho_f(t,x).
\end{equation}
It is easy to check that this function $\rho$ satisfies the transport equation $\eqref{EqTransport}_{(1)}$ and the required initial condition $\eqref{EqTransport}_{(2)}$. The final condition $\rho(T) = 0$ in $\eqref{EqTransport}_{(3)}$ is satisfied due to the properties of the flow proved in Lemma \ref{SortieEps2}, which guarantees that $\chi(T) = 0$.

In the next subsections, we describe how to get estimates on the function $\rho$ constructed in \eqref{def-rho} in the weighted spaces adapted to the Carleman estimates derived in Section \ref{Sec-Velocity}.
\subsection{Explicit description of the density}
To begin with, let us remark that the function $\chi$ is explicitly given by:
\begin{equation}\label{ExprChi}
	\chi(t,x)=
	\left\{
		\begin{array}{cl}
			1&\mbox{ if $t<T_0$}
			\\
			1&\mbox{ if $t\geq T_0$ and }X(\tau,t,x)\in \Omega\mbox{ for all $\tau\in [T_0,t]$,}
			\\
			0&\mbox{ else,}
		\end{array}
	\right.
\end{equation}
so that from Lemma \ref{SortieEps2} we have in particular
\begin{equation}
	\label{Chi=0-T-2T1-T}
	\chi(t,x)=0 \hbox{ and } \rho(t,x) = \rho_b(t,x) \quad \hbox{ for } (t,x)\in [T-2T_1, T]\times \Omega.
\end{equation}
We also give explicit expressions for $\rho_f$ and $\rho_b$. In order to do that, for $t\in [0,T]$, we introduce
\begin{equation}
	\label{DefOmega0T}
	\begin{split}
		\Omega_{[0]}(t)&\ov \{x\in \Omega \mid \widehat X(\tau,t,x)\in \Omega\quad \mbox{ for all }\tau \in [0,t]  \}
		\\
		\Omega_{[T]}(t)&\ov \{x\in \Omega \mid \widehat X(\tau,t,x)\in \Omega\quad \mbox{ for all }\tau \in [t,T]  \}
	\end{split}
\end{equation}
and for all $(t,x)\in [0,T]\times \Omega$:
\begin{equation}\label{Deftintout}
\begin{split}
	t_{\rm in}(t,x)&\ov \sup\{\tau \in [0, t)\mid \widehat X(\tau,t,x)\in \partial\Omega\},
	\\
	t_{\rm out}(t,x)&\ov \inf\{\tau \in (t, T]\mid \widehat X(\tau,t,x)\in \partial\Omega\}.
\end{split}
\end{equation}
In the above definitions, we use the convention $\sup\emptyset =0$ and $\inf\emptyset =T$. This way, $t_{\rm in}(t,x)=0$ iff $x\in \Omega_{[0]}(t)$ and
$t_{\rm out}(t,x)=T$ iff $x\in \Omega_{[T]}(t)$.

Using these notations, $\rho_f$ and $\rho_b$ are explicitly given by
\begin{eqnarray}
	\label{rho-f-explicit}
	\rho_f(t,x)
	& =& \ds
		\left\{
			\begin{array}{ll}
				\displaystyle \rho_0(\widehat X(0,t,x))-\int_{0}^{t} (\widehat {\bf u}\cdot\nabla \overline{\sigma})(\tau, \widehat X(\tau,t,x) ){\rm d\tau} &\displaystyle \mbox{ if }x\in \Omega_{[0]}(t),
				\\
				\displaystyle - \int_{t_{\rm in}(t,x)}^{t} (\widehat {\bf u}\cdot\nabla \overline{\sigma})(\tau, \widehat X(\tau,t,x) ){\rm d\tau} &\displaystyle \mbox{ else,}
			\end{array}
		\right.
	\\
	 \label{rho-b-explicit}
	\rho_b(t,x)
	& = &
	\int_{t_{\rm out}(t,x)}^{t} (\widehat {\bf u}\cdot\nabla \overline{\sigma})(\tau, \widehat X(\tau,t,x) ){\rm d\tau} \quad \mbox{ for } x\in \Omega.
\end{eqnarray}
We are now in position to derive weighted estimates on $\rho$.

\subsection{Weighted estimates on the density}

In order to derive weighted estimates on $\rho$ based on the Carleman weights $\psi$, $\theta$, $\varphi$, $\xi$ described in \eqref{Psi}--\eqref{PsiGamma}--\eqref{ThetaMu}--\eqref{Phi-Xi}, we will need some further assumptions.

{\bf Assumptions on the weights.}
We assume that $T_0$ and $T_1$ in the definition of $\theta$ in \eqref{ThetaMu} satisfy
\begin{equation}
	\label{Assumption-T-0-1}
	T_0 \in (0,T_0^*), \quad T_1 \in (0,T_1^*),
\end{equation}
where $T_0^*$ and $T_1^*$ are given by Lemma \ref{LemSortieEps}.

We also assume that the function $\psi$ in \eqref{Psi} satisfies the transport equation
\begin{equation}
	\label{Eq-Psi-Transport}
	\partial_t \psi + \overline{\bf y} \cdot \nabla \psi = 0 \quad \hbox{ in } \Omega_T.
\end{equation}

{\bf Assumptions on $\widehat \u$.}
In order to derive estimates on $\rho$, we shall assume that $\widehat \u$ is in a weighted Sobolev space. According to Theorem \ref{Thm-Stokes-Controlled}, it is natural to assume
\begin{align}
	\label{Assumption-hat-u-L2-L2}
	& \xi^{- 2} \widehat \u e^{s \varphi} \in L^2(0,T; {\bf L}^2 (\Omega)),
	\\
	\label{Assumption-hat-u-L2-H2}
	 & \widehat \u e^{3s \varphi^*/4} \in L^2(0,T; {\bf H}^2 (\Omega)) \hbox{ with } \norm{\widehat \u e^{3s \varphi^*/4}}_{L^2(0,T; {\bf H}^2 (\Omega))} \leq \varsigma.
\end{align}

{\bf Extension of $\widehat\u$.} To fit into the setting of Section \ref{Sec-ControlledDensity}, we extend $\widehat \u$ on $[0,T]\times \R^2$ that we still denote the same: $\widehat \u = {\bf E}(\widehat \u)$, where ${\bf E}$ denotes an extension from ${\bf H}^2(\Omega)$ to ${\bf H}^2(\R^2)$ such that $\norm{ {\bf E} (\v)}_{{\bf H}^2(\R^2)} \leq 2 \norm{\v}_{{\bf H}^2(\Omega)}$ for all $\v \in {\bf H}^2(\Omega)$. This allows us to define the flow $\widehat X$ by \eqref{FlowXannexe} for $(t,\tau, x) \in [0,T]^2 \times \R^2$.

Note that, for $s$ large enough, this last assumption is stronger than \eqref{Ass-On-U} and is thus perfectly compatible with the construction of Section \ref{Sec-ControlledDensity}, as it implies in particular that
\begin{equation}
	\label{Norm-L-infini}
	\norm{\theta \widehat\u}_{L^2(0,T; {\bf L}^\infty(\R^2))}\leq c \varsigma e^{- c_0 s \lambda},
\end{equation}
where $c_0 >0$ is independent of $s$ and $\lambda$. For the following we suppose that $s\geq s_0$ and $\lambda\geq 1$ with $s_0$ large enough such that \eqref{Ass-On-U} and \eqref{Norm-L-infini} are satisfied.

{\bf On the flows $\widehat X$ and $\overline{X}$.} We first establish a lemma on the closeness of $\widehat X$ to $\overline{X}$.

\begin{lemma}\label{Lemma-Close-Flows}
	There exists $c>0$ independent of $s$ and $\lambda$ such that for all $(\tau,t)\in [0,T]^2$ and $x\in \mathbb{R}^2$:
	\begin{equation}
		\label{EstXbarX}
		|\widehat X(\tau,t,x)-\overline{X}(\tau,t,x)| \leq c \varsigma  e^{- c_0 s \lambda}.
	\end{equation}
	Moreover, if $T_0\leq t\leq \tau\leq T$, we also have
	\begin{equation}
		\label{EstXbarX0}
		\theta (t) |\widehat X(\tau,t,x)-\overline{X}(\tau,t,x)| \leq c \varsigma e^{-c_0 s \lambda}.
	\end{equation}
\end{lemma}
\begin{proof}
	Estimate \eqref{EstXbarX} is an immediate consequence of \eqref{EstXbarXaux} and \eqref{Norm-L-infini}.  From \eqref{EstXbarXaux}, we also have
	$$
		\theta (t) |\widehat X(\tau,t,x)-\overline{X}(\tau,t,x)| \leq T^{1/2} e^{L T} \theta (t) \|\widehat {\bf u}\|_{L^2(t,\tau;{\bf L}^\infty(\R^2))},
	$$
	where $L=\|\nabla \overline {\bf y} \|_{L^\infty(0,T; {\bf L}^\infty(\R^2))}$. Using the fact that $\theta$ is increasing on $[T_0, T]$,
	\begin{equation}\nonumber
		\theta (t) |\widehat X(\tau,t,x)-\overline{X}(\tau,t,x)| \leq T^{1/2} e^{L T} \|\theta \widehat {\bf u}\|_{L^2(t,\tau;{\bf L}^\infty(\R^2))},
	\end{equation}
	for all $T_0\leq t\leq \tau\leq T$, which concludes the proof of Lemma \ref{Lemma-Close-Flows} by \eqref{Norm-L-infini}.
\end{proof}
{\bf On the weight functions.} Here, we shall deeply use the fact that $\psi$ is assumed to solve the transport equation \eqref{Eq-Psi-Transport}, thus implying in particular that
\begin{equation}
	\label{IdentityPsi}
	\forall (t,\tau, x) \in [0,T]^2 \times \R^2,
	\quad \psi(t, \overline{X}(t, \tau, x)) = \psi(\tau, x).
\end{equation}
We then show the following lemma:

\begin{lemma}\label{Lemma-Phi-Vs-Weights}
	There exist $c_1>0$, $c_2>0$ and $c_3>0$ independent of $s$ and $\lambda$, and $s_0>1$ such that for all $s\geq s_0$, $\lambda\geq 1$, the following inequalities hold:
	 \begin{enumerate}
	 	\item For all $t \in  [0,T-2T_1]$, $\tau\in [0,t]$ and $x\in \mathbb{R}^2$,
			\begin{eqnarray}
				\varphi(t,x)-\varphi(\tau,\widehat X(\tau,t,x)) &\leq& c_1 \varsigma   e^{- c_2s \lambda  },
				\label{EstVarphi}
				\\
				\frac{\xi(\tau,\widehat X(\tau,t,x))}{\xi(t,x)} &\leq& 2 e^{c_1 \varsigma  e^{-c_2 s \lambda } }.
				\label{EstVarphi2}
			\end{eqnarray}
		\item For all $t \in  [T_0,T]$, $\tau\in [t,T]$ and $x\in \mathbb{R}^2$,
			\begin{eqnarray}
				\varphi(t,x)-\varphi(\tau,\widehat X(\tau,t,x)) &\leq& c_1 \varsigma  e^{- c_2s\lambda} - c_3 (\theta(\tau) - \theta(t)),
				\label{EstVarphi3}
				\\
				\frac{\xi(\tau,\widehat X(\tau,t,x))}{\xi(t,x)} &\leq& \frac{\theta(\tau)}{\theta(t)}e^{c_1 \varsigma   e^{- c_2s\lambda}}.
				\label{EstVarphi4}
			\end{eqnarray}
	\end{enumerate}
\end{lemma}

\begin{proof}
	We focus on the proof of item {\it 2}, the first one being similar and easier because $\theta$ takes value in $[1,2]$ close to $t = 0$. Estimate \eqref{EstVarphi3} follows from the following computations: for $T_0 \leq t \leq \tau \leq T$,
	\begin{eqnarray*}
		\lefteqn{\varphi(t,x)-\varphi(\tau,\widehat X(\tau,t,x))}
		\\
		& = &\theta(t) \left(\lambda e^{6 \lambda (m+1)} - e^{\lambda \psi(t,x)} \right) - \theta (\tau)  \left(\lambda e^{6 \lambda (m+1)} - e^{\lambda \psi(\tau,\widehat X(\tau,t, x))} \right)
		\\
		& = & \theta(t) \left( e^{\lambda \psi(\tau,\widehat X(\tau,t, x))} - e^{\lambda \psi(t,x)}\right) + (\theta(t)-\theta(\tau))\left(\lambda e^{6 \lambda (m+1)}-e^{\lambda \psi(\tau,\widehat X(\tau,t,x))}\right)
		\\
		& \leq & \theta(t) \left( e^{\lambda \psi(\tau,\widehat X(\tau,t, x))} - e^{\lambda \psi(t,x)}\right)- c_3 (\theta(\tau) - \theta(t)),
	\end{eqnarray*}
	for some $c_3 >0$, where we used in the last estimate that $\theta$ is increasing on $[T_0, T]$. We then use \eqref{IdentityPsi} and \eqref{EstXbarX0}:
	\begin{multline*}
		|
			 \theta(t) \left( e^{\lambda \psi(\tau,\widehat X(\tau,t, x))} - e^{\lambda \psi(t,x)}\right)
		|
		=
		\theta(t) \left|e^{\lambda \psi(\tau,\widehat X(\tau,t,x))}- e^{\lambda \psi(\tau,\overline{X}(\tau,t,x))}\right|
		\\
		\leq c \theta(t) \lambda \|\nabla \psi \|_{\infty}e^{\lambda ({6m+1})} |\widehat X(\tau,t,x)-\overline{X}(\tau,t,x)|
		\leq c_1 \varsigma e^{-c_2 s \lambda},
	\end{multline*}
	for $s$ large enough, as announced in \eqref{EstVarphi3}. Next, by construction we have
	\begin{eqnarray}
		\frac{\xi(\tau,\widehat X(\tau,t,x))}{\xi(t,x)}
		& = &
		\frac{\theta(\tau)}{\theta(t)} e^{\lambda(\psi(\tau,\widehat X(\tau,t,x))- \psi(\tau,\overline{X}(\tau,t,x)))}
		\nonumber
		\\
		&\leq & \frac{\theta(\tau)}{\theta(t)} e^{\lambda \|\nabla \psi \|_{\infty}  |\widehat X(\tau,t,x)-\overline{X}(\tau,t,x)|},\label{EqxiXxi}
	\end{eqnarray}
	which immediately yields \eqref{EstVarphi4} by \eqref{EstXbarX0}.
\end{proof}
We immediately deduce from Lemma \ref{Lemma-Phi-Vs-Weights} the following:
\begin{proposition}
	\label{Prop-Aleph}
	Introducing the weight function
	\begin{equation}
		\label{Aleph}
		\aleph(t,x) \ov (\xi(t,x))^{-2} e^{s \varphi(t,x)},
	\end{equation}
	there exist $s_0\geq 1$ and $c>0$ independent of $s$ and $\lambda$ such that for all $\lambda \geq 1$, $s\geq s_0$, for all $(\tau,t, x) \in  [0,T]\times [0,T]\times \Omega$ satisfying $\tau\leq t\leq T-2T_1$ or $T_0\leq t\leq \tau$,
	\begin{equation}
		\aleph (t, x)\leq c \aleph (\tau, \widehat X(\tau,t,x)).
		\label{Est-Aleph-flow}
	\end{equation}
\end{proposition}

\begin{proof}
	If $\tau\leq t\leq T-2T_1$ then \eqref{Est-Aleph-flow} follows immediately from \eqref{EstVarphi} and \eqref{EstVarphi2}.
	
	 If $T_0\leq t\leq \tau$ then \eqref{Est-Aleph-flow} follows from \eqref{EstVarphi3} and \eqref{EstVarphi4}:
	$$
		\aleph (t, x)\leq \left(\frac{\theta^2(\tau)e^{-c_3 s \theta(\tau)}}{\theta^2(t)e^{-c_3 s \theta(t)}}\right) e^{c_1 \varsigma (s+2) e^{-s\lambda c_2}} \aleph (\tau, \widehat X(\tau,t,x)).
	$$
	But, for $s\geq 2/c_3$, the function $x\mapsto x^2 e^{- c_3 s x}$ is decreasing on $[1,+\infty)$ and then, since $\theta$ is increasing on $[T_0, T]$,
	$
		\theta^2(\tau)e^{-c_3 s \theta(\tau)} \leq \theta^2(t)e^{-c_3 s \theta(t)}.
	$
\end{proof}

{\bf On the controlled trajectory $\rho$.} We now derive estimates on the controlled trajectory $\rho$ given by Section \ref{Sec-ControlledDensity}:
\begin{theorem}
	\label{Thm-Est-Rho}
	Let $\psi$, $\theta$, $\varphi$, $\xi$ are defined in \eqref{Psi}--\eqref{PsiGamma}--\eqref{ThetaMu}--\eqref{Phi-Xi} and assume \eqref{Assumption-T-0-1}, \eqref{Eq-Psi-Transport}. Further assume that $\widehat{\u}$ satisfies \eqref{Assumption-hat-u-L2-L2} and \eqref{Assumption-hat-u-L2-H2} with $s\geq s_0$, $\lambda\geq 1$ and $s_0$ large enough such that \eqref{Ass-On-U} and \eqref{Norm-L-infini} are satisfied.
	
	There exists $c>0$ independent of $s$, $\lambda$ and $\widehat{\u}$ such that the solution $\rho$ given by Section \ref{Sec-ControlledDensity} satisfies
	\begin{equation}
		\label{Est-Rho-L-2}
		\norm{\aleph \rho}_{L^2(\Omega_T)} \leq C \left( \norm{ \aleph \widehat{\u}}_{L^2(0,T;{\bf L}^2(\Omega))} + e^{s \varphi^*(0)} \norm{\rho_0}_{L^2(\Omega)} \right),
	\end{equation}
	where $\aleph$ is given by \eqref{Aleph}, and
	\begin{multline}
		\label{Est-Rho-L-infty}
		\norm{e^{s \lambda e^{ 6\lambda(m+1)} \theta(t)/2} \rho}_{L^\infty(\Omega_T)} \leq C \left( \norm{ e^{s \lambda e^{ 6\lambda(m+1)} \theta(t)/2} \widehat{\u}}_{ L^2(0,T; {\bf L}^\infty(\Omega))}\right.
		\\
		\left.+ e^{s \lambda e^{6\lambda(m+1)} } \norm{\rho_0}_{L^\infty(\Omega)} \right).
	\end{multline}
\end{theorem}

\begin{proof}
	The proof of Theorem \ref{Thm-Est-Rho} follows from the precise description of $\rho_f$ and $\rho_b$ given in \eqref{rho-f-explicit}--\eqref{rho-b-explicit}.
	
	Let us begin with the proof of estimate \eqref{Est-Rho-L-infty}.
	On one hand, as $t\mapsto s \lambda e^{6\lambda(m+1)} \theta(t)$ is non-increasing on $(0,T-2T_1)$, from \eqref{rho-f-explicit} we get, for all $(t,x) \in (0,T-2T_1)\times \Omega$
	\begin{multline*}
		e^{s \lambda e^{6\lambda(m+1)} \theta(t)} |\rho_f(t,x)|^2 \leq 2 e^{ s \lambda e^{6\lambda(m+1)} \theta(t)}\|\rho_0\|_{L^\infty(\Omega)}^2
		\\
		 +
		 2\|\nabla \overline{\sigma}\|_{{\bf L}^\infty(\Omega_T)}^2 \int_{0}^{t} e^{s \lambda e^{6\lambda(m+1)} \theta(\tau)} \|\widehat {\bf u}(\tau,\cdot )\|_{{\bf L}^\infty(\Omega)}^2 {\rm d\tau}.
	\end{multline*}
	On the other hand, using that $t\mapsto s \lambda e^{6\lambda(m+1)} \theta(t)$ is non-decreasing on $(T_0,T)$, from \eqref{rho-b-explicit}, similarly, we have, for all $(t,x) \in (T_0,T) \times \Omega$,
	$$
		e^{s \lambda e^{6\lambda(m+1)} \theta(t)} |\rho_b(t,x)|^2 \leq
		 \|\nabla \overline{\sigma}\|_{{\bf L}^\infty(\Omega_T)}^2 \int_{t}^{T} e^{s \lambda e^{6\lambda(m+1)} \theta(\tau)} \|\widehat {\bf u}(\tau,\cdot )\|_{{\bf L}^\infty(\Omega)}^2 {\rm d\tau}.
	$$
	Together with the fact that the solution $\chi$ of \eqref{Eq-chi} takes value in $[0,1]$ on $\Omega_T$ and the properties \eqref{ExprChi}, these two estimates easily yield \eqref{Est-Rho-L-infty}.

	We then focus on the proof of \eqref{Est-Rho-L-2}, that mainly relies on the two following estimates: for all time $t \in (0,T-2T_1)$, we get
	\begin{equation}
		\label{est-rho-f-L-2}
			\int_{\Omega} |\rho_f(t)|^2\aleph ^2(t){\rm d}x
			\leq
			C\left(e^{2s\varphi^*(0)} \int_{\Omega} |\rho_0|^2{\rm d}x
			+ \iint_{\Omega_T} |{\bf \widehat u}|^2\aleph ^2 {\rm d}x  {\rm d\tau} \right),
	\end{equation}
	and for all time $t \in (T_0, T)$,
	\begin{equation}
		\label{est-rho-b-L-2}
		\int_{\Omega}|\rho_b(t)|^2\aleph ^2(t) {\rm d}x
		\leq c  \iint_{\Omega_T} |{\bf \widehat u}|^2\aleph ^2 {\rm d}x  {\rm d\tau}.
	\end{equation}
	Indeed, once estimates \eqref{est-rho-f-L-2}--\eqref{est-rho-b-L-2} are proved, we can bound the $L^2(\Omega_T)$-norm of $\aleph \rho$ by the sum of the $L^\infty((0,T-2T_1); L^2(\Omega))$-norm of $\rho_f$ and of the $L^\infty((T_0,T); L^2(\Omega))$-norm of $\rho_b$, and estimate \eqref{Est-Rho-L-2} immediately follows.
	
	Let us first present the proof of \eqref{est-rho-f-L-2}. We fix $t \in [0, T - 2T_1]$. From \eqref{rho-f-explicit} and \eqref{Est-Aleph-flow} we deduce that, for $x\in \Omega_{[0]}(t)$,
	\begin{multline*}
		 |\rho_f(t,x)|^2\aleph ^2(t,x)
		 \\
		 \leq
		 C\left(|\rho_0(\widehat X(0,t,x))|^2\aleph ^2(0,\widehat X(0,t,x))
		 +
		 \int_{0}^{t} |\widehat {\bf u}(\tau, \widehat X(\tau,t,x) )|^2\aleph ^2(\tau,\widehat X(\tau,t,x)) {\rm d\tau}\right),
	\end{multline*}
	whereas for $x\in \Omega\backslash \Omega_{[0]}(t)$,
	$$
		|\rho_f(t,x)|^2\aleph ^2(t,x)
		\leq
		C\int_{t_{\rm in}(t,x)}^{t} |\widehat {\bf u}(\tau, \widehat X(\tau,t,x) )|^2\aleph ^2(\tau,\widehat X(\tau,t,x)) {\rm d\tau}.
	$$
	Combining these two estimates, for all $t \in (0, T - 2T_1)$ we get:
	\begin{multline}\label{estRho-1}
		\int_{\Omega}|\rho_f(t,x)|^2\aleph ^2(t,x) {\rm d}x
		\leq C
		\int_{\Omega_{[0]}(t)} |\rho_0(\widehat X(0,t,x))|^2\aleph ^2(0,\widehat X(0,t,x)){\rm d}x
		\\
		+
		C \int_{0}^{t} \int_{\Omega}{\bf 1}_{[t_{\rm in}(t,x),t]}(\tau)|\widehat {\bf u}(\tau, \widehat X(\tau,t,x) )|^2\aleph ^2(\tau,\widehat X(\tau,t,x)) {\rm d}x{\rm d\tau} .
	\end{multline}
	Since $\overline{\bf y}+\widehat{\bf
u}$ is divergence free in $\Omega_T$, the Jacobian of $x\mapsto \widehat X(t,\tau, x)$ equals $1$ identically. Therefore,
	\begin{eqnarray*}
	\int_{\Omega_{[0]}(t)} |\rho_0(\widehat X(0,t,x))|^2\aleph ^2(0,\widehat X(0,t,x)){\rm d}x
	& = &
	\int_{\widehat X(0,t,\Omega_{[0]}(t))} |\rho_0(x)|^2\aleph ^2(0,x) {\rm d}x
	\\
	&\leq & \int_{\Omega} |\rho_0(x)|^2\aleph ^2(0,x){\rm d}x.
	\end{eqnarray*}
	Similarly, we get
	$$
		\int_{0}^{t} \int_{\Omega}{\bf 1}_{[t_{\rm in}(t,x),t]}(\tau)|\widehat {\bf u}(\tau, \widehat X(\tau,t,x) )|^2\aleph ^2(\tau,\widehat X(\tau,t,x)) {\rm d}x{\rm d\tau}
		\leq \int_{0}^t \int_\Omega |\widehat {\bf u}(\tau, x )|^2\aleph ^2(\tau,x) {\rm d\tau}{\rm d}x
	$$
	Estimate \eqref{est-rho-f-L-2} then follows from \eqref{estRho-1}.
	
	The proof of \eqref{est-rho-b-L-2} is based on \eqref{rho-b-explicit} and follows the same lines. It is therefore left to the reader.
\end{proof}
\section{Proof of Theorem \ref{Thm-Main}}\label{Sec-Proof-Thm-Main}

We are now in position to prove Theorem \ref{Thm-Main}. The idea is to construct suitable convex sets which are invariant by the mapping $\mathscr{F} = \mathscr{F}_{(\rho_0,\u_0)}$ in \eqref{DefF} and relatively compact for a topology making $\mathscr{F}$ continuous. In all this section, we assume the assumptions of Theorem \ref{Thm-Main}.

\subsection{Main steps of the proof of Theorem \ref{Thm-Main}}

In the introduction, we introduced formally a mapping $\mathscr{F}$. We are now in position to define it precisely.

In order to do this, the first step in the proof of Theorem \ref{Thm-Main} is to construct a weight function $\tilde \psi$ which is suitable for both Section \ref{Sec-Velocity} and Section \ref{Sec-Density}, i.e. suitable in the same time for controlling the velocity equation and the density equation. We claim the following result, proved in Section \ref{Subsec-Weight}:

\begin{lemma}
	\label{Lem-Weight-Function}
	Let $\Omega$ be a smooth bounded domain. Further assume the regularity condition \eqref{Conditions-TargetTrajectory} on $(\overline{\sigma}, \overline{\y})$, the geometric condition \eqref{OmegaEqOmegaOut} and condition \eqref{Cond-Gamma-c}.
	
	Then one can find a smooth ($C^2$) bounded domain $\mathcal{O}$ satisfying \eqref{ExtensionOfOmega} such that there exists a $C^2(\overline{\mathcal{O}_T})$-function $\tilde \psi$ satisfying the transport equation \eqref{Eq-Psi-Transport} and satisfying assumptions \eqref{Psi} to \eqref{Ass-Psi} for $\omega_T=[0,T]\times \omega$ and $\tilde \omega_T=[0,T]\times \tilde \omega$ { where $\omega$, $\Tilde\omega$ are two subdomains of $\Ocal\backslash \overline{\Omega}$ such that $\Tilde\omega\Subset \omega$}.
\end{lemma}

Next, we take $T_0^*$, $T_1^*$ and $\varsigma >0$ given by Lemma \ref{SortieEps2} and fix $T_0 \in (0,T_0^*)$ and $T_1 \in (0,T_1^*)$. We then use the function $\psi$, $\theta$, $\varphi$ and $\xi$ given by \eqref{PsiGamma}, \eqref{ThetaMu}, \eqref{Phi-Xi} for $m\geq 5$, $s\geq s_0$, $\lambda\geq \lambda_0$, and the notations given in \eqref{Variants-Phi}--\eqref{Variants-Xi}. Moreover, we suppose that $s_0$, $\lambda_0$ are large enough given by Theorem \ref{Thm-Stokes-Controlled} and Theorem \ref{Thm-Est-Rho}. Now, we define the spaces ${\bf X}_{s,\lambda}$ and $Y_{s,\lambda}$ depending on positive parameters $s\geq s_0$ and $\lambda\geq \lambda_0$ as follows:
\begin{align}
	\label{Def-X-s-lambda}
	{\bf X}_{s, \lambda}
		\ov
	\{
		\u \in & \L^2(\Omega_T), \quad \hbox{ with } \div(\u ) = 0\;\mbox{ in }\Omega_T,
		\\
		& s^{1/4}  \xi^{1/m- 2} e^{s \varphi}\u \in \L^2(\Omega_T)	,
		\notag	
		\\
		&
		e^{3s\varphi^*/4} \u \in L^2(0,T; {\bf H}^2(\Omega))\cap H^1(0,T; \L^2(\Omega))
		\notag
	\},
\end{align}
endowed with the norm
$$
	\norm{\u}_{{\bf X}_{s,\lambda}}^2\ov
	\|e^{3s\varphi^*/4}\u\|^2_{{L}^2({\bf H}^2)\cap {H}^1({\bf L}^2)}+s^{1/2}  \norm{\xi^{1/m-2} e^{s \varphi} \u}_{\L^2(\Omega_T)}^2,
$$
and
\begin{equation}
	\label{Def-Y-s-lambda}
	{ Y}_{s, \lambda}
		\ov
	\{
		\rho \in  L^\infty(\Omega_T), \hbox{ with }
		\xi^{-2} e^{s \varphi} \rho \in L^2(\Omega_T)	
		\, \hbox{ and }
		e^{s\lambda e^{6\lambda(m+1)} \theta/2} \rho \in L^\infty(\Omega_T)
		\notag
	\},
\end{equation}
endowed with the norm
$$
	\norm{\rho}_{Y_{s,\lambda}}\ov
	\|\xi^{-2} e^{s\varphi}\rho\|_{L^2(\Omega_T)}+ \norm{e^{s\lambda e^{6\lambda(m+1)} \theta/2} \rho}_{L^\infty(\Omega_T)}.
$$
We also introduce the space ${\bf F}_{s,\lambda}$ defined by
\begin{multline*}
	{\bf F}_{s,\lambda} \ov \{ {\bf f} \in L^2(0,T; \L^2(\Omega)), \hbox{ with }\xi^{-2} {\bf f}e^{s \varphi} \in L^2(0,T; {\bf L}^2(\Omega))\}
	\\
	\hbox{endowed with the norm}
	\norm{\bf f}_{{\bf F}_{s,\lambda}} \ov \norm{ \xi^{-2} {\bf f}e^{s \varphi} }_{L^2(\L^2)}.
\end{multline*}
Note that, in the above definitions as well as in the following results, we keep the dependence in both parameters $\lambda$ and $s$ to be consistent with notations of Section \ref{Sec-Velocity}. However, only the dependence in $s$ will be needed in this section.

We then derive the following results.
\begin{theorem}[On the mapping $\mathscr{F}_1$]
	\label{Thm-Def-F-1}
	Fix $\rho_0 \in L^\infty(\Omega)$. For all $\widehat{\u}\in {\bf X}_{s, \lambda}$ with $\norm{\widehat \u}_{{\bf X}_{s, \lambda}} \leq \varsigma$, the construction in Section \ref{Sec-ControlledDensity} yields $\rho = \mathscr{F}_1(\widehat \u, \rho_0)$ solution of the control problem \eqref{EqTransport}. Besides, $\rho \in Y_{s, \lambda}$ and for some constant $C$ independent of $s\geq s_0$ and $\lambda \geq \lambda_0$,
	\begin{equation}
		\label{Estimate-F-1}
		\norm{ \rho}_{Y_{s, \lambda}} \leq C \left(\frac{1}{s^{1/4} }\norm{\widehat \u}_{{\bf X}_{s, \lambda}} +  e^{s \varphi^*(0)} \norm{\rho_0}_{L^\infty(\Omega)}\right).
	\end{equation}
	Furthermore, the application $\mathscr{F}_1$ satisfies the following compactness property: If $\widehat{\u}_n$ is a sequence of functions in ${\bf X}_{s, \lambda}$ with $\norm{\widehat \u_n}_{{\bf X}_{s, \lambda}} \leq \varsigma$ which weakly converges to some $\widehat\u$ in $X_{s,\lambda}$, the corresponding sequence $\rho_n =  \mathscr{F}_1(\widehat \u_n, \rho_0)$ strongly converges to $\mathscr{F}_1(\widehat \u, \rho_0)$ in all $L^q(\Omega_T)$ for $q \in [1, \infty)$.
\end{theorem}
The proof of Theorem \ref{Thm-Def-F-1} is done in Section \ref{Sec-Proof-F-1}. Let us point out that the compactness property stated in Theorem \ref{Thm-Def-F-1} is of primary importance for our result and follows from \cite[Theorem 4]{BoyerFabrie-2007}.

We then focus on the study of the mapping $\mathscr{F}_2$:
\begin{theorem}[On the mapping $\mathscr{F}_2$]
	\label{Thm-Def-F-2}
	We can define a bounded linear mapping $\mathscr{F}_2: {\bf F}_{s, \lambda}\times {\bf V}^1_0(\Omega)\to {\bf X}_{s, \lambda}$ such that for all $\u_0 \in {\bf V}^1_0(\Omega)$ and ${\bf f} \in {\bf F}_{s, \lambda}$, $\u = \mathscr{F}_2 ({\bf f}, \u_0)$ solves the control problem \eqref{SystemparaboliquePF} and satisfies, for some constant $C>0$ independent of $s \geq s_0$ and $\lambda\geq \lambda_0$,
	\begin{equation}
		\label{Estimate-F-2}
		\norm{ \u}_{{\bf X}_{s, \lambda}} \leq C \left(\norm{\bf f}_{{\bf F}_{s,\lambda}} +  e^{\frac{5}{4} s \varphi^*(0)} \norm{\u_0}_{{\bf H}^1_0(\Omega)}\right).
	\end{equation}
\end{theorem}
Theorem \ref{Thm-Def-F-2} is a direct consequence of Theorem \ref{Thm-Stokes-Controlled}: the mapping $\mathscr{F}_2$ is obtained by restricting the controlled trajectory given by Theorem \ref{Thm-Stokes-Controlled} to $(0,T) \times \Omega$. Of course, this depends on the extension $\Ocal$ of $\Omega$, but this choice is done once for all. Estimate \eqref{Estimate-F-2} is then a rewriting of Theorem \ref{Thm-Stokes-Controlled} by taking into account that $\f$ and $\u_0$ are extended by zero outside $\Omega$.

We are then able to derive the following properties on the mapping $\mathscr{F}$ in \eqref{DefF}, whose proof is postponed to Section \ref{Sec-Proof-F}:
\begin{theorem}
	\label{Thm-Prop-F}
	Let $\rho_0 \in L^\infty(\Omega)$ and $\u_0 \in {\bf V}^1_0(\Omega)$.
	
	Then for all $s \geq s_0$ and $\lambda \geq \lambda_0$ the mapping $\mathscr{F}$ in \eqref{DefF} is well-defined for all $\widehat{\u}\in {\bf X}_{s, \lambda}$ with $\norm{\widehat \u}_{{\bf X}_{s, \lambda}} \leq \varsigma$. Besides, for all $\widehat{\u}\in {\bf X}_{s, \lambda}$ with $\norm{\widehat \u}_{{\bf X}_{s, \lambda}} \leq \varsigma$, $\u = \mathscr{F}(\widehat\u)$ belongs to ${\bf X}_{s, \lambda}$, and satisfies, for some constant $C_0$ independent of $s$ and $\lambda$,
	\begin{multline}
		\norm{ \u}_{{\bf X}_{s, \lambda}} \leq
		C_0 \left( \frac{1}{s^{1/4} } \norm{\widehat \u}_{{\bf X}_{s, \lambda}} + \norm{\widehat \u}_{{\bf X}_{s, \lambda}}^2 \right.
		\\
		\left.
		+
		e^{s \varphi^*(0)} \norm{\rho_0}_{L^\infty(\Omega)}+ e^{2s \varphi^*(0)} \norm{\rho_0}_{L^\infty(\Omega)}^2 + e^{\frac{5}{4}s \varphi^*(0)} \norm{\u_0}_{{\bf H}^1_0(\Omega)}\right).
		\label{Estimate-f-Prop}
	\end{multline}
	Moreover, if $\widehat{\u}_n$ is a sequence of functions in ${\bf X}_{s, \lambda}$ with $\norm{\widehat \u_n}_{{\bf X}_{s, \lambda}} \leq \varsigma$ which weakly converges to some $\widehat\u$ in $\X_{s,\lambda}$, the corresponding sequence $\u_n =  \mathscr{F}(\widehat \u_n)$ strongly converges to $\u = \mathscr{F}(\widehat \u)$ in $L^2(0,T;\L^2(\Omega))$.
\end{theorem}

We may then conclude the proof of Theorem \ref{Thm-Main}. For $R \in (0, \varsigma)$, we introduce the closed convex set
$$
	\X_{s,\lambda}^{R} = \{ \u \in \X_{s, \lambda} \hbox{ with } \norm{\u}_{\X_{s,\lambda}}\leq R \}.
$$
We then choose $R$ small enough such that $C_0 R \leq 1/4$, where $C_0$ is the constant in \eqref{Estimate-f-Prop}, $\lambda = \lambda_0$ and $s\geq s_0$ large enough to guarantee $C_0\leq s^{1/4}/4$. We then get from \eqref{Estimate-f-Prop} that for all $\widehat\u \in \X_{s,\lambda_0}^R$, $\u = \mathscr{F}(\widehat\u)$ satisfies
$$
	\norm{ \u}_{{\bf X}_{s, \lambda_0}} \leq \frac{R}{2} + C_0 \left(e^{s \varphi^*(0)} \norm{\rho_0}_{L^\infty(\Omega)}+ e^{2s \varphi^*(0)} \norm{\rho_0}_{L^\infty(\Omega)}^2 + e^{\frac{5}{4} s \varphi^*(0)} \norm{\u_0}_{{\bf H}^1_0(\Omega)}\right).
$$
Thus, choosing $\varepsilon>0$ sufficiently small in \eqref{SmallDataIni}, we can guarantee that the mapping $\mathscr{F}$ maps $\X_{s,\lambda_0}^R$ to itself.

We then check that the set $\X_{s,\lambda_0}^R$ is compact in $L^2(0,T; \L^2(\Omega))$ as $H^1(0,T; \L^2(\Omega)) \cap L^2(0,T; {\bf H}^2(\Omega))$ is compactly embedded in $L^2(0,T; \L^2(\Omega))$ due to Rellich's compactness theorem and Aubin-Lions' theorem.

Besides, the mapping $\mathscr{F}$ is continuous on $\X_{s,\lambda_0}^R$ endowed with the $L^2(0,T; \L^2(\Omega))$-topology from Theorem \ref{Thm-Prop-F}. Indeed, if  $\widehat{\u}_n$  is a sequence of functions in $\X_{s,\lambda_0}^R$ which strongly converges to $\widehat \u$ in $L^2(0,T; \L^2(\Omega))$, it necessarily weakly converges in $\X_{s,\lambda_0}^R$. Thus, from the last item of Theorem \ref{Thm-Prop-F}, $\u_n = \mathscr{F}(\widehat\u_n)$ strongly converges to $\u$ in $L^2(0,T; \L^2(\Omega))$.

Schauder's fixed point theorem then implies the existence of a fixed point to the mapping $\mathscr{F}$, and concludes the proof of Theorem \ref{Thm-Main}.

\subsection{Proof of Lemma \ref{Lem-Weight-Function}}\label{Subsec-Weight}

	We do it in several steps.
	
	{\bf Construction of $\mathcal{O}$.}
	In a neighborhood of $\Gamma_c$, according to Assumption \eqref{Cond-Gamma-c}, there exists a $C^2$ extension $\mathcal{O}$ of $\Omega$ such that
	\begin{itemize}
		\item $\Omega \subset \mathcal{O}$;
		\item $ \Gamma_0 \subset \partial \Omega \cap \partial \mathcal{O}$ and for all $t \in (0,T)$ and $x \in \partial \Omega \cap \partial \mathcal{O}$, $\overline{\y}(t,x) \cdot \n \geq \gamma/2$;
		\item $\partial \mathcal{O} \cap \partial \Omega$ and $\mathcal{O} \setminus \overline{\Omega}$ have a finite number of connected components.
	\end{itemize}

{Let $\omega$, $\Tilde\omega$ be two subdomains of $\Ocal\backslash \overline{\Omega}$ such that $\Tilde\omega\Subset \omega$ and fix $d_0 = {\rm dist}(\tilde\omega, \Omega)$.}

%
	
	{\bf Construction of an extension $\overline {\bf y}_e$ of $\overline{\y}$ in $\mathcal{O}$.}
	We then construct an extension $\overline {\bf y}_e\in {\bf C}^2([0,T]\times \mathbb{R}^2)$ of $\overline {\bf y}$ outside $\Omega_T$ (i.e $\overline {\bf y}_e\equiv \overline {\bf y}$ in $\Omega_T$)
	satisfying
	\begin{align}
		\label{Hyp-y-e}
		\norm{\overline{\y}_e}_{{\bf C}^2([0,T]\times \overline{\Ocal})} <\infty,
		\qquad
		\inf_{[0,T]\times \partial\mathcal{O}} \overline{\bf y}_e\cdot {\bf n} >0,
		\\
		\label{Smallness-y-e}
		 \hbox{ and } \quad  \overline{\y}_e\equiv 0 \quad \mbox{ in }\quad (0,T)\times \Tilde\omega.
	\end{align}
	Before going into the detailed construction of $\overline{\y}_e$, let us remark that $\overline{\y}_e$ cannot be divergence free as it would not be compatible with the condition $ \inf_{[0,T]\times \partial\mathcal{O}} \overline{\bf y}_e\cdot {\bf n} >0$.
	\\
	In order to construct such extension $\overline{\y}_e$, we proceed as follows. First, we consider any extension of $\overline\y$ in ${\bf C}^2([0,T]\times \mathbb{R}^2)$. By continuity, there exists $d_1 >0$ such that for all $(t, x) \in (0,T) \times \partial\mathcal{O} \hbox{ with } d(x, \Omega) <d_1$, $\overline{\y}(t,x) \cdot {\bf n} \geq \gamma/3$. We also introduce a function ${\bf m}$ in ${\bf C}^2([0,T] \times \R^2)$ such that ${\bf m} \cdot {\bf n} = 1$ on the whole boundary $\partial \Ocal$ and ${\bf m}\equiv 0$ in $\Tilde\omega$, and a smooth non-negative cut-off function $\eta= \eta(x)$ taking value $1$ in $\overline{\Omega}$ and $0$ for all $x \in \Ocal$ with $d(x, \Omega) > \min\{d_0,d_1\}$, and we then consider
	$$
		\overline{\y}_e(t,x) = \eta(x) \overline\y(t,x) + (1- \eta(x)) {\bf m}(x).
	$$
	This function indeed belongs to ${\bf C}^2([0,T]\times \mathbb{R}^2)$. Besides,
	$$
			\inf_{[0,T]\times \partial\mathcal{O}} \overline{\bf y}_e\cdot {\bf n} \geq \min\left\{ \frac{\gamma}{3}, 1\right\},
	$$
	and \eqref{Smallness-y-e} is trivially satisfied as ${\bf m}\equiv 0$ and $\eta\equiv 0$ in $\Tilde\omega$.
	
	{\bf Construction of $\tilde \psi$.} We then construct a function $\widehat\psi_T= \widehat\psi_T(x)$ such that
	\begin{itemize}
		\item $\widehat\psi_T$ is a non-negative $C^2(\overline{\Ocal})$ function;
		\item The critical points of $\widehat\psi_T$ all belong to $\Tilde\omega$;
		\item $\widehat\psi_T$ satisfies the following conditions on the boundary $\partial\mathcal{O}$:
			\begin{equation}
				\label{Boundary-Conditions-Psi-T}
				\left\{	
					\begin{array}{l}
						\widehat\psi_T(x) = 0 \hbox{ on } \partial \Ocal,
						\\
						\overline{\y}_e(T,x) \cdot \nabla \widehat\psi_T(x) = -1  \hbox{ on } \partial \Ocal,
						\\
						\partial_t \overline{\y}_e(T,x) \cdot \nabla \widehat\psi_T(x) -(\overline{\y}_e(T,x) \cdot \nabla)^2 \widehat\psi_T(x) = 0  \hbox{ on } \partial \Ocal.
					\end{array}
				\right.
			\end{equation}
		\item $\inf_{\Ocal} \widehat \psi_T = (\widehat\psi_T)_{|\partial \Ocal} = 0$.
	\end{itemize}
	Note that such function exists according to the construction of Fursikov and Imanuvilov in \cite{FursikovImanuvilov} suitably modified to handle the conditions on the first and second order derivatives on the boundary of $\Ocal$. This can be done easily following the lines of \cite[Appendix III]{TWbook}.
	
	We then consider the solution $\widehat\psi$ of
	\begin{equation}
		\label{Hat-Psi-Eq}
			\left\{
				\begin{array}{ll}
					\partial_t \widehat\psi + \overline{\y}_e \cdot \nabla \widehat\psi = 0 \quad &\hbox{ in } \Ocal_T,
					\\
					\widehat\psi(t,x) = t- T\quad  & \hbox{ on } \Gamma_T,
					\\
					\widehat\psi(T ) = \widehat\psi_T & \hbox{ in } \Ocal.
				\end{array}
			\right.
	\end{equation}
	Note that this problem is well-posed as, by construction, $\overline{\y}_e(t,x) \cdot {\bf n}>0$ for all $(t,x) \in (0,T) \times \partial \Ocal$. We then want to check that
	\begin{itemize}
		\item $\partial_\n \widehat \psi(t,x) \leq 0$ for $(t,x) \in (0,T) \times \partial \Ocal$;
		\item $\widehat \psi$ belongs to $C^2([0,T]\times \overline{\Ocal})$;
		\item For all $t \in [0,T]$, the critical points of $\widehat \psi(t,\cdot)$ belong to $\Tilde\omega$;
		\item For all $t\in [0,T]$, $\inf_{\Ocal} \widehat \psi(t,\cdot) = \widehat \psi(t)_{|\partial \Ocal}$;
	\end{itemize}
	
	Indeed, providing these properties are true, one can choose $a >0$ and $b \in \R$ such that the function $\tilde \psi = a \widehat\psi + b$ is suitable for Lemma \ref{Lem-Weight-Function}.
	
	Using the equation \eqref{Hat-Psi-Eq} and the fact that tangential derivatives of $\widehat\psi$ vanish due to the boundary conditions, we get, for all $(t,x) \in (0,T)\times \partial \Ocal$,
	$$
		\overline{\y}_e(t,x) \cdot \n\,  \partial_{\bf n} \widehat\psi(t,x) = - \partial_t \widehat \psi(t,x) =-1.
	$$
	Using \eqref{Hyp-y-e}, we thus deduce that
\begin{equation}\label{DerNormaleNegative}
\forall (t,x)\in (0,T) \times \partial\Ocal,\quad  \partial_{\bf n} \widehat\psi(t,x)\leq \frac{-1}{\displaystyle \inf_{[0,T]\times \partial\Ocal}\overline{\y}_e(t,x) \cdot \n}<0.
\end{equation}

	
	To describe more precisely the function $\widehat\psi$, we will introduce the flow $\overline{X}_e$ corresponding to $\overline{\y}_e$, i.e. the solution of
	\begin{equation}
	\label{Def-FlowBAR-e}
		\forall (t,\tau,x)\in [0,T]^2\times \mathbb{R}^2,\quad  \partial_t \overline{X}_e(t,\tau,x)=\overline {\bf y}_e(t,  \overline{X}_e(t,\tau,x)),\quad \overline{X}_e(\tau,\tau,x)=x.
	\end{equation}

	The fact that $\widehat\psi \in C^2([0,T] \times \overline{\Ocal})$ follows from the following lemma, whose proof is postponed to Appendix \ref{Appendix-Reg-Weight}:
	\begin{lemma}
		\label{Lem-Reg-Weight}
		Under the above assumptions, $\widehat\psi \in C^2([0,T] \times \overline{\Ocal})$.
	\end{lemma}	
	
	We then have to check that the critical points of $\widehat \psi(t,\cdot)$ all belong to $\Tilde\omega$.
	
	We first remark that \eqref{DerNormaleNegative} implies that
there is no critical point on the boundary $\partial \Ocal$. We then remark that $\nabla \widehat\psi$ solves the equation
	\begin{equation}
		\label{Eq-hat-psi-CritPoints}
		\partial_t \nabla \widehat \psi + (\overline{\y}_e \cdot \nabla) \nabla \widehat \psi + D\overline{\y}_e \nabla \widehat\psi = 0 \quad \hbox{ in } \Ocal_T.
	\end{equation}
	From the equation \eqref{Eq-hat-psi-CritPoints}, if the point $x_c$ is a critical point for $\widehat\psi(t_c, \cdot)$, then for all $t$ in a neighborhood around $t_c$, $\overline{X}_e(t, t_c, x_c)$ is a critical point for $\widehat\psi(t, \cdot)$. Note that this neighborhood actually correspond to the set $I_c$ of time $t \in [0,T]$ such that the trajectory $\tau \mapsto \overline{X}_e(\tau, t_c, x_c)$ stays in $\Ocal$ for $\tau$ between $t$ and $t_c$.
	
	Since there is no critical point on the boundary $\partial \Ocal$ and thanks to conditions \eqref{Hyp-y-e}, for all time $t_c \in [0,T]$, the critical points $x_c$ of $\widehat\psi(t_c, \cdot)$ are linked by a trajectory $\tau \mapsto \overline{X}_e(\tau, t_c, x_c)$ to a critical point $x_{c,T}$ of $\widehat\psi_T$, that is $x_c = \overline{X}_e(t_c, T, x_{c,T})$. By construction of $\widehat\psi_T$, $x_{c,T}$ necessarily belongs to $\Tilde\omega$. But, according to condition \eqref{Smallness-y-e}, as long as $\overline{X}_e(t,T, x_{c,T})\in \Tilde\omega$,
	$$
		\partial_t \overline{X}_e(t, T, x_{c,T}) =0,
	$$
so that $\overline{X}_e(t, T, x_{c,T})=x_{c,T}$ for all $t\in [0,T]$. This implies that the set of critical points of $\widehat\psi(t,\cdot)$ is invariant through the flow $\overline{X}_e$ and is then included in $\Tilde\omega$.

	We finally check the condition $ \inf_{\Ocal} \widehat \psi(t,\cdot) = \widehat \psi(t)_{|\partial \Ocal}$ for all $t \in [0,T]$ by contradiction. If this were wrong, there would exist $t \in [0,T]$ and $x_t\in \Ocal$ such that $x_t \in \hbox{Argmin}\widehat\psi(t,\cdot)$. Thus, $x_t$ would be a critical point, and as above, $\overline{X}_e(T,t,x_t)$ would belong to $\Ocal$ and be a critical point of $\widehat\psi_T$. Following, $\widehat\psi(t,x_t) = \widehat\psi_T(\overline{X}_e(T, t,x_t))$ would be larger than $0$ due to the assumption on $\widehat\psi_T$. But from the boundary conditions, it follows that $\inf_\Ocal \widehat\psi(t)$ cannot be strictly smaller than $ \widehat \psi(t)_{|\partial \Ocal}$, which is negative for all time $t\in [0,T)$.
	
\subsection{Proof of Theorem \ref{Thm-Def-F-1}}\label{Sec-Proof-F-1}

	According to Section \ref{Sec-Density}, the construction in Section \ref{Sec-ControlledDensity}  yields $\rho = \mathscr{F}_1(\widehat \u, \rho_0)$ solution of the control problem \eqref{EqTransport} for $\widehat\u$ satisfying \eqref{Ass-On-U}. This condition is indeed satisfied for $\widehat\u \in {\bf X}_{s, \lambda}$ with $\norm{\widehat \u}_{{\bf X}_{s, \lambda}} \leq \varsigma$, see \eqref{Assumption-hat-u-L2-L2}--\eqref{Assumption-hat-u-L2-H2}--\eqref{Norm-L-infini}.
	
	Theorem \ref{Thm-Est-Rho} immediately provides estimate \eqref{Estimate-F-1}, as $\lambda e^{6\lambda (m+1)} \theta/2 \leq 3\varphi^*/4$, see \eqref{Phi-bounds}.
	
	We then focus on the proof of the compactness property. According to the construction in Section \ref{Sec-ControlledDensity}, we introduce $\rho_{f,n}$ the solution of
	\begin{equation}\label{Eq-rho-f-n}
	 \left\{
	\begin{array}{rlll}
		\partial_t\rho_{f,n}+(\overline {\bf y}+\widehat {\bf u}_n)\cdot \nabla \rho_{f,n} &=&- \widehat {\bf u}_n\cdot \nabla\overline{\sigma }&\mbox{ in } \Omega_T,
		\\
	        	\rho_{f,n}(t,x) &=&0&\mbox{ for } t \in (0,T),\, x \in \partial \Omega,
		\\
		& &  &\hspace{-5ex} \hbox{ with }\left( \overline {\bf y}(t,x)+\widehat {\bf u}_n(t,x)\right)\cdot {\bf n}(x) <0,
		\\
        		\rho_{f,n}(0)&=&\rho_0\quad &\mbox{ in }\Omega,
	\end{array}
	\right.
	\end{equation}
	$\rho_{b,n}$ the solution of
	\begin{equation}\label{Eq-rho-b-n}
	 \left\{
	\begin{array}{rlll}
		\partial_t\rho_{b,n}+(\overline {\bf y}+\widehat {\bf u}_n)\cdot \nabla \rho_{b,n} &=&- \widehat {\bf u}_n\cdot \nabla\overline{\sigma }&\mbox{ in } \Omega_T,
		\\
	        \rho_{b,n}&=& 0&\mbox{ for } t \in (0,T),\, x \in \partial \Omega,
	        \\
	        & &  & \hspace{-5ex} \hbox{ with } \left( \overline {\bf y}(t,x)+\widehat {\bf u}_n(t,x)\right)\cdot {\bf n}(x) >0,
	        \\
	        \rho_{b,n}(T)&= & 0\quad &\mbox{ in }\Omega,
	\end{array}
	\right.
	\end{equation}
	and $\chi_n$ the solution of
	\begin{equation}\label{Eq-chi-n}
	 \left\{
	\begin{array}{rlll}
		\partial_t\chi_n+(\overline {\bf y}+\widehat {\bf u}_n)\cdot \nabla \chi_n &=& 0&\mbox{ in } \Omega_T,
		\\
	        \chi_n &= &1_{t \in (0,T_0)}(t)&\mbox{ for } t \in (0,T),\, x \in \partial \Omega,
	        \\
	        & & & \hspace{-5ex} \hbox{ with }\left( \overline {\bf y}(t,x)+\widehat {\bf u}_n(t,x)\right)\cdot {\bf n}(x) <0,
	        \\
	        \chi_n(0)&= & 1\quad &\mbox{ in }\Omega.
	\end{array}
	\right.
	\end{equation}
	Since $\widehat {\bf u}_n$ is a bounded sequence of $H^1(0,T; \L^2(\Omega))\cap L^2(0,T; {\bf H}^2(\Omega))$, which is compact in $L^2(0,T; \L^2(\Omega))$, up to a subsequence still denoted the same for simplicity, $\widehat {\bf u}_n$ converge to $\widehat{\u} $ weakly in $H^1(0,T; \L^2(\Omega))\cap L^2(0,T; {\bf H}^2(\Omega))$ and strongly in $L^2(0,T; \L^2(\Omega))$. Then \cite[Theorem 4]{BoyerFabrie-2007} applies and for all $q\in [1,+\infty)$ the sequence
$\chi_n$ strongly converges towards $\chi$ in $L^q(\Omega_T)$ solution of \eqref{Eq-chi}.

	Next, to pass to the limit in \eqref{Eq-rho-f-n}, we notice that $\sigma_{f,n}\ov \overline{\sigma}+\rho_{f,n}$ solves
	\begin{equation}
		\label{Eqsigmafn}
	 \left\{
		\begin{array}{rlll}
			\partial_t\sigma_{f,n}+(\overline {\bf y}+\widehat {\bf u}_n)\cdot \nabla \sigma_{f,n} &=& 0&\mbox{ in } \Omega_T,
			\\
			\sigma_{f,n}(t,x) &=& 0&\mbox{ for } t \in (0,T),\, x \in \partial \Omega,
			\\
			& & &  \hbox{ with }\left( \overline {\bf y}(t,x)+\widehat {\bf u}_n(t,x)\right)\cdot {\bf n}(x) <0,
			\\
		        \sigma_{f,n}(0)&=& \rho_0\quad &\mbox{ in }\Omega.
		\end{array}
	\right.
	\end{equation}
	Thus, by applying again \cite[Theorem 4]{BoyerFabrie-2007} we deduce that, for all $q\in [1,+\infty)$, the sequence $\sigma_{f,n}$ is strongly convergent in $L^q(\Omega_T)$ to the solution $\sigma_{f}$ of
	\begin{equation}\label{Eqsigmaf}
	 \left\{
		\begin{array}{rlll}
			\partial_t\sigma_{f}+(\overline {\bf y}+\widehat {\bf u})\cdot \nabla \sigma_{f} &=&0&\mbox{ in } \Omega_T,
			\\
			\sigma_{f}(t,x) &=&0&\mbox{ for } t \in (0,T),\, x \in \partial \Omega,
			\\
			& & &  \hbox{ with }\left( \overline {\bf y}(t,x)+\widehat {\bf u}(t,x)\right)\cdot {\bf n}(x) <0,
			\\
		        \sigma_{f}(0)&=& \rho_0\quad &\mbox{ in }\Omega.
		\end{array}
	\right.
	\end{equation}
	It follows that $\rho_{f,n}$ strongly converges in all $L^q (\Omega_T)$ for $q \in [1, \infty)$ to $\rho_f = \sigma_f - \overline{\sigma}$, which solves \eqref{Eq-rho-f} by construction.
	
	Of course, the same can be done to show that $\rho_{b,n}$ strongly converges in all $L^q(\Omega_T)$ for $q \in [1, \infty)$ to the solution $\rho_b$ of \eqref{Eq-rho-b}. Consequently, the sequence $\rho_n =  \mathscr{F}_1(\widehat \u_n, \rho_0)$ converges to  $\rho =  \mathscr{F}_1(\widehat \u, \rho_0)$ in $L^q(\Omega_T)$ for all $q \in [1, \infty)$.
	
%
	
 $\u_n$ and $\u$ are uniformly bounded in ${\bf X}_{s, \lambda}$, so the convergence of $\u_n$ to $\u$ actually is weak in ${\bf X}_{s, \lambda}$.

\subsection{Proof of Theorem \ref{Thm-Prop-F}}\label{Sec-Proof-F}

	Let $\rho_0 \in L^\infty(\Omega)$, $\u_0 \in \V^1_0(\Omega)$ and $\widehat\u \in {\bf X}_{s, \lambda}$ with $\norm{\widehat\u}_{{\bf X}_{s, \lambda}} \leq \varsigma$.
	
	According to Theorem \ref{Thm-Def-F-1}, $\rho = \mathscr{F}_1(\widehat\u,\rho_0)$ belongs to $Y_{s,\lambda}$ and is bounded in that space by \eqref{Estimate-F-1}. Thus, according to Theorem \ref{Thm-Def-F-2}, for $\mathscr{F}$ to be well-defined, we have to check that ${\bf f}(\rho, \widehat\u)$ given in \eqref{Def-f} belongs to ${\bf F}_{s,\lambda}$, and we will get estimates on $\u = \mathscr{F}(\widehat\u)$ from an estimate of ${\bf f}(\rho, \widehat\u)$ in ${\bf F}_{s,\lambda}$ according to \eqref{Estimate-F-2}. We thus estimate ${\bf f}(\rho, \widehat\u)$ in ${\bf F}_{s,\lambda}$ term by term from estimates on $\rho \in Y_{s,\lambda}$ and $\widehat\u \in \X_{s, \lambda}$.
	\\
	We easily check
	\begin{align*}
		\lefteqn{
		\norm{\xi^{-2} e^{s\varphi} \rho(\partial_t {\widehat \u}+(\overline {\bf y}+{\widehat\u})\cdot \nabla {\widehat\u} +\widehat\u \cdot \nabla \overline\y)}_{L^2(\L^2)}
		}
		\\
		\leq
		&	
		\norm{e^{s\lambda e^{6\lambda(m+1)}\theta/2} \rho}_{L^\infty}
		\norm{\xi^{-2} e^{s\varphi- s\lambda e^{6\lambda(m+1)}\theta /2} (\partial_t {\widehat \u}+((\overline {\bf y}+{\widehat\u})\cdot \nabla)\widehat\u +\widehat\u \cdot \nabla \overline\y)}_{L^2(\L^2)}
		\\
		\leq
		&
		C\norm{\rho}_{Y_{s,\lambda}} \norm{e^{3s\varphi^*/4} \widehat\u}_{L^2({\bf H}^2)\cap H^1(\L^2)}
		\norm{\xi^{-2} e^{s\varphi- s\lambda e^{6\lambda(m+1)}\theta /2- 3s\varphi^*/4}}_{L^\infty},
	\end{align*}
	where we used that $\overline {\bf y}+{\widehat\u}$ is bounded in $L^\infty(0,T; \L^4(\Omega))$ due to Sobolev's embedding as $\widehat\u $ belongs to $L^2(0,T;{\bf H}^2(\Omega))\cap H^1(0,T; \L^2(\Omega))$ and is of norm bounded by $\varsigma$, and that
	$$
		\norm{e^{3s\varphi^*/4} \nabla \widehat\u}_{L^2(\L^4)} \leq C\norm{e^{3s\varphi^*/4} \widehat\u}_{L^2({\bf H}^2)\cap H^1(\L^2)}.
	$$
	According to \eqref{Phi-bounds}, $s\varphi- s\lambda e^{6\lambda(m+1)}\theta/2- 3s\varphi^*/4 \leq - s \varphi/4$, and thus there exists some constant $C$ independent of $s$ and $\lambda$ such that
	$$
		\norm{\xi^{-2} e^{s\varphi- s\lambda e^{6\lambda(m+1)}\theta/2- 3s\varphi^*/4}}_{L^\infty} \leq C.
	$$
	Following,
	\begin{equation}
	\label{Estimate-term-f-1}
		\norm{\xi^{-2} e^{s\varphi} \rho(\partial_t {\widehat \u}+(\overline {\bf y}+{\widehat\u})\cdot \nabla{\widehat\u} )}_{L^2(\L^2)}
		\leq C\norm{\rho}_{Y_{s,\lambda}} \norm{ \widehat\u}_{\X_{s, \lambda}}.
	\end{equation}
	
	Next, we estimate $\overline {\sigma} (\widehat \u \cdot \nabla)\widehat \u$. Similarly as above, we write
	\begin{align}
		\norm{\xi^{-2} e^{s\varphi}  \overline {\sigma} \widehat \u \cdot \nabla\widehat \u}_{L^2(\L^2)}
		\leq
		&
		C
		\norm{e^{3 s \varphi^*/4}\widehat\u}_{L^\infty(\L^4)}
		 \norm{e^{3 s \varphi^*/4}\nabla\widehat\u}_{L^2(\L^4)}
		 \norm{\xi^{-2} e^{s \varphi- 3s \varphi^*/2}}_{L^\infty}
		 \notag
		 \\
		\leq
		&
		C \norm{ \widehat\u}_{\X_{s, \lambda}}^2.
		\label{Estimate-term-f-2}
	\end{align}
	
	Last, we estimate $\rho (\partial_t \overline {\bf y}+(\overline {\bf y}\cdot \nabla)\overline {\bf y})$:
	\begin{equation}
		\norm{\xi^{-2} e^{s\varphi} \rho (\partial_t \overline {\bf y}+(\overline {\bf y}\cdot \nabla)\overline {\bf y})}_{L^2(\L^2)}
		\leq
		C \norm{\xi^{-2} e^{ s \varphi}\rho}_{L^2}
		\leq
		C \norm{\rho}_{Y_{s, \lambda}}.
		\label{Estimate-term-f-3}
	\end{equation}
	
	Putting estimates \eqref{Estimate-term-f-1}--\eqref{Estimate-term-f-3} together, we obtain:
	\begin{equation}
		\label{Estimate-f-rho-u}
		\norm{\f(\rho,\widehat\u)}_{{\bf F}_{s, \lambda}}
		=  \norm{\xi^{-2} e^{s\varphi} \f(\rho,\widehat\u)}_{L^2(\L^2)}
		\leq
		C  (\norm{\rho}_{Y_{s, \lambda}}+ \norm{\rho}_{Y_{s,\lambda}}^2 + \norm{ \widehat\u}_{\X_{s, \lambda}}^2).
	\end{equation}	
	Combined with estimates \eqref{Estimate-F-1} and \eqref{Estimate-F-2}, this yields the well-posedness of the mapping $\mathscr{F}$ for $\widehat\u \in \X_{s,\lambda}$ with $\norm{\widehat\u}_{\X_{s, \lambda}} \leq \varsigma$ and the estimate \eqref{Estimate-f-Prop}.
	
	We now focus on the last part of Theorem \ref{Thm-Prop-F}. Let $\widehat\u_n$ is a sequence of $\X_{s, \lambda}$ with $\norm{\widehat\u_n}_{\X_{s,\lambda}} \leq \varsigma$ which weakly converges to $\widehat\u$. Note that this weak convergence implies that $\norm{\widehat\u}_{\X_{s,\lambda}} \leq \varsigma$, so that $\mathscr{F}(\widehat\u)$ is well-defined.
	
	Besides that, according to Theorem \ref{Thm-Def-F-1}, the sequence $\rho_n = \mathscr{F}_1(\widehat\u_n, \rho_0)$ strongly converges in all $L^q(\Omega_T)$ with $q< \infty$ to $\rho = \mathscr{F}_1(\widehat\u, \rho_0)$ and the sequence $\rho_n$ is uniformly bounded in $Y_{s,\lambda}$.
	
	We then have to check that $\f(\rho_n, \widehat\u_n)$ weakly converges in ${\bf F}_{s, \lambda}$ to $\f(\rho, \widehat\u)$. But \eqref{Estimate-f-rho-u} shows that the sequence $\f(\rho_n, \widehat\u_n)$ is bounded in ${\bf F}_{s,\lambda}$, and thus we only need to prove that the sequence $\f(\rho_n, \widehat\u_n)$ weakly converges in $\mathcal{D}'(\Omega_T)$ to $\f(\rho, \widehat\u)$. To obtain this convergence result in $\mathcal{D}'(\Omega_T)$, as $\rho_n$ strongly converges to $\rho$ in all $L^q(\Omega_T)$ with $q <\infty$ and $\widehat\u_n$ weakly converges to $\widehat\u$ in $H^1(0,T; \L^2(\Omega))\cap L^2(0,T;{\bf H}^2(\Omega))$, we only have to focus on the convergence of the term $(\overline\sigma+\rho_n) \widehat\u_n \cdot \nabla \widehat\u_n$. But, using the compactness of $H^1(0,T; \L^2(\Omega))\cap L^2(0,T;{\bf H}^2(\Omega))$ in $L^4(0,T; \L^4(\Omega))$, we have the convergences
	\begin{align*}
		\overline\sigma + \rho_n & \underset{n\to\infty}\longrightarrow  \overline\sigma+\rho \quad  && \hbox{ strongly in } L^q(\Omega_T), \, q \in [1, \infty),
		\\
		\widehat\u_n  &\underset{n\to\infty}\longrightarrow  \widehat \u \quad && \hbox{ strongly in } L^4(0,T; \L^4(\Omega)),
		\\
		\nabla \widehat\u_n &\underset{n\to\infty}\longrightarrow  \nabla\widehat \u \quad && \hbox{ weakly in } L^2(0,T; \L^2(\Omega)),
	\end{align*}
	so that, choosing $q =4$ for instance, we obtain the weak convergence of $(\overline\sigma+\rho_n) \widehat\u_n \cdot \nabla \widehat\u_n$ to $(\overline\sigma+\rho) \widehat\u \cdot \nabla \widehat\u$.
	
	Following, $\f(\rho_n, \widehat\u_n)$ weakly converges in ${\bf F}_{s, \lambda}$ to $\f(\rho, \widehat\u)$ and, since $\mathscr{F}_2:{\bf F}_{s, \lambda} \times {\bf V}^1_0(\Omega)  \to {\bf X}_{s, \lambda}$ is a linear bounded operator, we obtain that $\u_n = \mathscr{F}(\widehat\u_n)=\mathscr{F}_2(\f(\rho_n, \widehat\u_n), \u_0)$ weakly converges to $\mathscr{F}_2(\f(\rho, \widehat\u) , \u_0) = \mathscr{F}(\widehat\u)=\u$ in $\X_{s, \lambda}$. Finally, as $\X_{s,\lambda}$ is compact in $L^2(0,T; \L^2(\Omega))$, $\u_n$ strongly converges to $u$ in $L^2(0,T; \L^2(\Omega))$.
	
	
\appendix
\section{Proofs of Theorems \ref{CarlemanThm} and \ref{Thm-Est-Y-Carl-Norms}}\label{Section-Proofs}
For simplicity, we make the proof of Theorems \ref{CarlemanThm} and \ref{Thm-Est-Y-Carl-Norms} for $\nu$ of equal to $1$.This can be done without loss of generality by replacing $\overline{\sigma}$ and $f$ by $\overline{\sigma}/\nu$ and $f/\nu$ if needed.

\subsection{Proof of Theorem \ref{CarlemanThm}}\label{Sec-Proof-Carleman-Heat}

	Let $z$ be a smooth function on $[0,T] \times \overline \Ocal$ satisfying $z = 0 $ on $(0,T) \times \partial \Ocal$ and set
	
	\begin{equation}
		\label{TheSourceTerm}
				f\ov - \overline \sigma \partial_t z - \Delta z, \quad  (t,x) \in (0,T) \times \Ocal,
	\end{equation}
	
	Set then
	\begin{equation}
		\label{ConjugateVariable}
		w  = e^{-s \varphi} z.
	\end{equation}
	According to the definition of $\theta$ in \eqref{ThetaMu}, $w$ satisfies
	\begin{equation}
		\label{ConditionsOnV}
		w(T, x) = 0, \quad \nabla w(T, x) = 0, \quad x \in \Ocal,
	\end{equation}
	in addition to the conditions $w(t, x) = 0$ on $(0,T) \times \partial \Ocal$.
	
	Besides, with $f$ as in \eqref{TheSourceTerm}, $w$ satisfies
	\begin{align*}
		e^{-s \varphi} f
		&
		= e^{-s \varphi}\left( - \overline \sigma \partial_t z - \Delta z\right)
				= e^{-s \varphi}\left(- \overline \sigma \partial_t (e^{s \varphi} w) - \Delta (e^{s \varphi} w) \right) = P_\varphi w,
	\end{align*}
	where the operator $P_\varphi$ is given by
	\begin{equation}
	\label{ConjugateOperator}
		P_\varphi w =
		- \overline \sigma \partial_t w - s \overline \sigma \partial_t \varphi w - \Delta w - 2 s \nabla \varphi\cdot \nabla w - s^2 |\nabla \varphi|^2 w  - s \Delta \varphi w.
	\end{equation}
	
	We now set $P_1$, $P_2$ and $R$ the operators:
	\begin{eqnarray}
		P_1 w & = &
		- \overline \sigma \partial_t w  - 2 s \nabla \varphi\cdot \nabla w  + 2 s \lambda^2 |\nabla \psi|^2 \xi w,
		\label{P_1}
		\\
		P_2 w & = &
		 - \Delta w   - s \overline \sigma \partial_t \varphi w-  s^2 |\nabla \varphi|^2 w,
		\label{P_2}
		\\
		R w & = &
		s \lambda \Delta  \psi \xi w - s \lambda^2 |\nabla \psi|^2\xi w,
		\label{R}
	\end{eqnarray}
	so that
	$$
		P_\varphi = P_1 + P_2 + R.
	$$
	We then use that $P_1 w + P_2 w = f e^{-s \varphi} - Rw$ and then
	\begin{multline}
		\label{Square-Pi}
		\iint_{\Ocal_T} |P_1 w|^2
		+
		\iint_{\Ocal_T} |P_2 w|^2
		+
		2 \iint_{\Ocal_T} P_1 w P_2 w
		\\
		=
		\iint_{\Ocal_T} |f e^{- s\varphi} - R w|^2
		\leq 2 \iint_{\Ocal_T} |f|^2 e^{-2 s \varphi} + 2 \iint_{\Ocal_T} |R w|^2.
	\end{multline}
	The main part of the proof then consists in computing the scalar product of $P_1 w$ with $P_2 w$ and estimate it from below.

	{\bf Computations.} We write
	$$
		\iint_{\Ocal_T} P_1 w P_2 w = \sum_{i, j = 1}^3 I_{ij},
	$$
	where $I_{i,j}$ is the scalar product of the $i$-th term of $P_1 w$ with the $j$-th term of $P_2 w$.
	
	{\it Computation of $I_{11}$.}
	\begin{eqnarray}
		I_{11}
		& = &
		 \iint_{\Ocal_T} \overline \sigma \partial_t w \Delta w
		=
		- \iint_{\Ocal_T} \overline \sigma \partial_t \left( \frac{|\nabla w|^2}{2}\right) - \iint_{\Ocal_T} \partial_t w \nabla \overline \sigma \cdot \nabla w
		\notag\\
		& = &
		 \frac{1}{2} \int_\Ocal \overline \sigma(0) |\nabla w(0)|^2 + \frac{1}{2}\iint_{\Ocal_T} \partial_t \overline \sigma |\nabla w|^2 - \iint_{\Ocal_T} \partial_t w \nabla \overline \sigma \cdot \nabla w.
		\label{I-1-1}
	\end{eqnarray}
	
	{\it Computation of $I_{12}$.}
	\begin{align}
		I_{12}
		&=
		s \iint_{\Ocal_T} \overline \sigma^2 \partial_t w \partial_t \varphi w
		\notag\\
		&=
		- \frac{s}{2}\int_\Ocal \overline \sigma^2(0) \partial_t \varphi(0) |w(0)|^2   - \frac{s}{2} \iint_{\Ocal_T} \overline \sigma^2 \partial_{tt} \varphi |w|^2
		\label{I-1-2}
		- s \iint_{\Ocal_T} \overline \sigma \partial_t \overline \sigma \partial_{t} \varphi |w|^2.
	\end{align}

	{\it Computation of $I_{13}$.}
	\begin{eqnarray}
		I_{13}
		& = &
		s^2 \iint_{\Ocal_T} \overline \sigma \partial_t w  |\nabla \varphi|^2 w \notag
		\\
		& =&
		-\frac{s^2}{2}  \int_\Ocal  \overline \sigma(0) |\nabla \varphi(0)|^2|w(0)|^2 -\frac{s^2}{2} \iint_{\Ocal_T} \overline \sigma \partial_t\left( |\nabla \varphi|^2 \right) |w|^2
		\label{I-1-3}
		\\
		& & - \frac{s^2}{2} \iint_{\Ocal_T} \partial_t \overline \sigma |\nabla \varphi|^2 |w|^2.\notag
	\end{eqnarray}

	{\it Computation of $I_{21}$.}
	\begin{eqnarray}
		I_{21}
		& = &
		2 s \iint_{\Ocal_T} \nabla \varphi\cdot \nabla w  \Delta w
		\notag
		\\
		& = &
		2 s \int_{\Gamma_T} \partial_{\bf n} \varphi |\partial_{\bf n} w|^2  -2s  \iint_{\Ocal_T} \nabla \left( \nabla \varphi\cdot \nabla w\right) \cdot \nabla w
		\notag
		\\
		& = &
		2 s \int_{\Gamma_T} \partial_{\bf n} \varphi |\partial_{\bf n} w|^2  -2s  \iint_{\Ocal_T} D^2\varphi( \nabla w, \nabla w)  - s  \iint_{\Ocal_T} \nabla \varphi\cdot \nabla \left( |\nabla w|^2 \right)
		\notag
		\\
		& = &
		s \int_{\Gamma_T} \partial_{\bf n} \varphi |\partial_{\bf n} w|^2  -2s  \iint_{\Ocal_T} D^2\varphi( \nabla w, \nabla w) + s  \int_\Ocal \Delta \varphi  |\nabla w|^2.
		\label{I-2-1}
	\end{eqnarray}

	{\it Computation of $I_{22}$.}
	\begin{eqnarray}
		I_{22}
		& = &
		2 s^2 \iint_{\Ocal_T} \overline \sigma \nabla \varphi\cdot \nabla w  \partial_t \varphi w
		=
		- s^2 \iint_{\Ocal_T} \div (\overline \sigma \partial_t \varphi \nabla \varphi) |w|^2
		\notag\\
		& = &
		- s^2 \iint_{\Ocal_T} \overline \sigma \div ( \partial_t \varphi \nabla \varphi) |w|^2
		- s^2 \iint_{\Ocal_T} \nabla \overline \sigma \cdot \nabla \varphi  \partial_t \varphi |w|^2.
		\label{I-2-2}
	\end{eqnarray}

	{\it Computation of $I_{23}$.}
	\begin{equation}
		I_{23}
		 =
		2 s^3 \iint_{\Ocal_T} \nabla \varphi\cdot \nabla w  |\nabla \varphi|^2 w
		=
		- s^3 \iint_{\Ocal_T} \div\left( |\nabla \varphi|^2 \nabla \varphi\right) |w|^2.
		\label{I-2-3}
	\end{equation}	

	{\it Computation of $I_{31}$.}
	\begin{align}
		I_{31}
		&=
		 - 2 s \lambda^2 \iint_{\Ocal_T} |\nabla  \psi|^2 \xi w \Delta w
		\notag
		\\
		& =
		  2 s \lambda^2 \iint_{\Ocal_T} |\nabla  \psi|^2 \xi  |\nabla w|^2 + 2 s \lambda^2 \iint_{\Ocal_T} \nabla(|\nabla  \psi|^2 \xi ) w \cdot \nabla w.
		 \label{I-3-1}
	\end{align}

	{\it Computation of $I_{32}$.}
	\begin{equation}
		I_{32}
		=
		- 2 s^2 \lambda^2 \iint_{\Ocal_T} \overline \sigma |\nabla  \psi|^2 \xi \partial_t \varphi |w|^2.
		\label{I-3-2}
	\end{equation}

	{\it Computation of $I_{33}$.}
	\begin{equation}
		I_{33} = - 2 s^3 \lambda^2 \iint_{\Ocal_T} |\nabla  \psi|^2 \xi |\nabla \varphi|^2 |w|^2.
		\label{I-3-3}
	\end{equation}
	
	Combining the above computations \eqref{I-1-1}--\eqref{I-3-3}, we obtain the following:
	\begin{align}
		& \iint_{\Ocal_T} P_1 w P_2 w
		\notag
		\\
		=
		& \frac{1}{2} \int_\Ocal \overline \sigma(0) |\nabla w(0)|^2 +  \frac{1}{2} \int_\Ocal |w(0)|^2 \overline \sigma(0) \left(-s^2  |\nabla \varphi(0)|^2 - s \overline \sigma(0) \partial_t \varphi(0) \right)
		\label{I-t=0}
		 \\
		 & -2 s \iint_{\Ocal_T} D^2 \varphi(\nabla w, \nabla w) + s \iint_{\Ocal_T} (\Delta \varphi + 2 \lambda^2 |\nabla  \psi|^2 \xi )  |\nabla w|^2
		\label{I-Gradw}		
		\\
		& + \iint_{\Ocal_T} |w|^2
		\Bigg(
			s^3 \left( - \div (|\nabla \varphi|^2 \nabla \varphi)  - 2 \lambda^2 |\nabla  \psi|^2 \xi |\nabla \varphi|^2 \right)
		\label{I-w-s3}
		\\
		&\phantom{+ \int_0^T \int_\Ocal |w|}	+
			s^2 \overline \sigma \left( - \partial_t \left(|\nabla \varphi|^2\right)  - (\Delta \varphi+ 2\lambda^2 |\nabla  \psi|^2 \xi) \partial_t \varphi \right)
		\label{I-w-s2}
		\\
		&\phantom{+ \int_0^T \int_\Ocal |w|}	
			+
			s \overline{\sigma}^2\left(- \frac{1}{ 2} \partial_{tt} \varphi \right)
		\Bigg)
		\label{I-w-s}
		 \\
		 & +s \int_0^T \int_{\partial \Ocal} \partial_{\bf n} \varphi |\partial_{\bf n} w |^2 		
		 \label{I-boundary}
		 + I_{\mathcal{R}}.
	\end{align}
	where
	\begin{multline}
	I_{\mathcal{R}} =
	\frac{1}{2} \iint_{\Ocal_T} \partial_t \overline{\sigma} |\nabla w|^2 + 2 s \lambda^2 \iint_{\Ocal_T} \nabla(|\nabla  \psi|^2 \xi ) w \cdot \nabla w
	- s\iint_{\Ocal_T} \overline \sigma \partial_t \overline \sigma \partial_t \varphi |w|^2
		\label{Def-mathcalR}
	\\
	 - \frac{s^2}{2} \iint_{\Ocal_T} \partial_t \overline \sigma |\nabla \varphi|^2 |w|^2
	- s^2 \iint_{\Ocal_T} \nabla \overline \sigma \nabla \varphi \partial_t \varphi |w|^2
	- \iint_{\Ocal_T} \partial_t w \nabla \overline \sigma \cdot \nabla w.
	\end{multline}
	
	{\bf Positivity.} Our main goal now is to check that the coefficients in the above integrals are positive, except perhaps on the observation set $\omega_T$. At this step, we will strongly rely upon the choice of the weight function $\varphi$ in \eqref{Phi-Xi}, and on the formula
	\begin{equation}
	\label{ExpressDtPhi}
		\partial_t \varphi = \frac{\partial_t \theta}{\theta} \varphi - \lambda \partial_t  \psi \xi, \quad \partial_t \xi = \frac{\partial_t \theta}{\theta} \xi + \lambda \partial_t  \psi \xi.
	\end{equation}
	
	In the following, to simplify notations, we will denote by $C$ generic positive large constants that do not depend on $s$ or $\lambda$ and by $c$ generic positive small constants independent of $s$ and $\lambda$. The constants may change from line to line.

	{\it Positivity of the terms \eqref{I-t=0} at $t = 0$.} Explicit computations yield
	$$
		- \partial_t \varphi(0)
		= \frac{\mu}{T_0} ( \lambda e^{6\lambda(m+1)} - e^{\lambda \psi(0)}) +  2\lambda \partial_t \psi(0) e^{\lambda \psi(0)}
		 \geq c s \lambda^3 e^{\lambda (12m +2 )}
	$$
	whereas
	$$
		|\nabla \varphi(0)|^2 \leq C \lambda^2 |\xi(0)|^2 \leq C \lambda^2 e^{2\lambda (6m+1)}.
	$$
	Thus, with \eqref{Positivity}, for some $\lambda_1>0$, taking $\lambda \geq \lambda_1\geq 1$,
	\begin{equation}
		\label{Positivity-t=0-coeff}
		\inf_{\overline \Ocal} \left\{ -s^2 |\nabla \varphi(0)|^2 - s \overline \sigma(0) \partial_t \varphi(0)\right\} \geq c s^2 \lambda^3 e^{2 \lambda (6 m +1)},
	\end{equation}
	and, following,
	\begin{equation}
		\label{Positivity-t=0}
		 \frac{1}{2} \int_\Ocal \overline \sigma(0) |w(0)|^2 \left(-s^2 |\nabla \varphi(0)|^2 - s \overline \sigma(0)\partial_t \varphi(0) \right)
		  \geq cs^2 \lambda^3 e^{2 \lambda (6 m +1)} \int_\Ocal |w(0)|^2.
	\end{equation}
	{\it Positivity of the terms \eqref{I-Gradw} involving the gradient.} For $\eta \in \R^N$, we have
	\begin{multline}
		\label{Comp-Dw}
		- 2 s D^2 \varphi( \eta, \eta) + s (\Delta \varphi+ 2 \lambda^2 |\nabla  \psi|^2 \xi) |\eta|^2
		\\
		= 2 s \lambda^2 \xi |\nabla  \psi \cdot \eta|^2 + s \lambda^2 \xi |\nabla  \psi |^2 |\eta|^2 + 2 s \lambda \xi D^2  \psi(\eta, \eta)  - s \lambda \xi \Delta  \psi |\eta|^2.
	\end{multline}
	Using \eqref{Ass-Psi}, we get the existence of $\lambda_2 = \lambda_2(\alpha, \| D^2  \psi\|_\infty)\geq \lambda_1$ such that for all $\lambda \geq \lambda_2$ and $\eta \in \R^N$,
	\begin{equation}
		\label{Est-GradW-below-OutsideOcal}
		\forall (t,x) \in \Ocal_T \setminus \tilde \omega_T,
		\quad
		- 2 s D^2 \varphi( \eta, \eta) + s (\Delta \varphi+ 2 \lambda^2 |\nabla  \psi|^2 \xi) |\eta|^2 \geq  c s \lambda^2 |\eta|^2 \xi,
	\end{equation}
	whereas there exists a positive constant $C= C(\alpha, \|D^2  \psi\|_\infty)$ such that
	\begin{equation*}
		\forall \eta \in \R^N,\, \forall (t,x) \in \tilde \omega_T,\\
		\quad
		 - 2 s D^2 \varphi( \eta, \eta) + s (\Delta \varphi+ 2 \lambda^2 |\nabla  \psi|^2 \xi) |\eta|^2 \geq c s \lambda^2 \xi |\eta|^2 - C s \lambda^2 \xi |\eta|^2.
	\end{equation*}
	Hence we obtain, for all $\lambda \geq \lambda_1$,
	\begin{multline}
		\label{Positivity-GradW}
		 -2 s \iint_{\Ocal_T} D^2 \varphi(\nabla w, \nabla w) + s \iint_{\Ocal_T} (\Delta \varphi+ 2 \lambda^2 |\nabla  \psi|^2 \xi) |\nabla w|^2
		\\
		\geq
		cs \lambda^2 \iint_{\Ocal_T}  \xi |\nabla w|^2  - C s \lambda^2 \iint_{\tilde \omega_T}  \xi |\nabla w|^2.
	\end{multline}

	{\it Positivity of the terms \eqref{I-w-s3} involving $w$ with scale $s^3$.} Using $\nabla \varphi = -\lambda \nabla \psi \xi$, we have
	\begin{align*}
		 - \div (|\nabla \varphi|^2 \nabla \varphi) &= {3\lambda^4} |\nabla \psi|^4 \xi^3 + \lambda^3 \xi^3 \div(|\nabla \psi|^2 \nabla \psi),
		\\
		  \lambda^2|\nabla  \psi|^2 \xi |\nabla \varphi|^2 &= \lambda^4 |\nabla \psi|^4 \xi^3.
	\end{align*}
	Hence
	\begin{equation}
		\label{Comp-w-s3}
		- \div (|\nabla \varphi|^2 \nabla \varphi)  - 2 \lambda^2|\nabla  \psi|^2 \xi |\nabla \varphi|^2
		=   \lambda^4  |\nabla  \psi|^4 \xi^3  +   \lambda^3 \xi^3\div(|\nabla \psi|^2 \nabla \psi).
	\end{equation}
	Using \eqref{Ass-Psi}, we thus get the existence of $\lambda_3 = \lambda_3(\alpha, \|D^2  \psi\|_\infty)\geq \lambda_2$ such that for $\lambda \geq \lambda_3$,
	\begin{equation}
		\label{Est-W-s3-below-OustideOcal}
		\forall (t,x) \in  \Ocal_T \setminus \omega_T,\quad
		- \div (|\nabla \varphi|^2 \nabla \varphi)  - 2 \lambda^2|\nabla  \psi|^2 \xi |\nabla \varphi|^2 \geq c \lambda^4 \xi^3.
	\end{equation}
	whereas there exists a positive constant $C = C (\alpha, \|D^2  \psi\|_\infty )$ such that
	\begin{equation}
		\label{Est-W-s3-below-InsideOcal}
		\forall (t,x) \in {\tilde \omega_T},\quad
				- \div (|\nabla \varphi|^2 \nabla \varphi)  - 2 \lambda^2|\nabla  \psi|^2 \xi |\nabla \varphi|^2 \geq c \lambda^4 \xi^3 - C \lambda^4 \xi^3.
	\end{equation}
	We thus obtain, for all $\lambda \geq \lambda_3$,
	\begin{multline}
		\label{Positivity-w-s3}
		s^3  \iint_{\Ocal_T} |w|^2 \left( - \div (|\nabla \varphi|^2 \nabla \varphi)  - 2 \lambda^2|\nabla  \psi|^2 \xi |\nabla \varphi|^2 \right)
		\\
		\geq
		c s^3 \lambda^4 \iint_{\Ocal_T}  \xi^3 |w|^2 - C s^3 \lambda^4 \iint_{{\tilde \omega_T}} \xi^3 | w|^2.
	\end{multline}

	{\it Terms \eqref{I-w-s2} involving $w$ in the scale $s^2$.} We have to estimate
	$$
		 - \partial_t \left( |\nabla \varphi|^2 \right) -( \Delta \varphi +2 \lambda^2|\nabla  \psi|^2 \xi) \partial_t \varphi.
	$$
	Explicit computations yield:
	\begin{eqnarray}
		\lefteqn{ - \partial_t \left( |\nabla \varphi|^2 \right) -( \Delta \varphi +2 \lambda^2|\nabla  \psi|^2 \xi) \partial_t \varphi} \notag
		 \\
		  &\qquad =&
		  -  \lambda^3 \xi^2 \partial_t  \psi |\nabla  \psi|^2  - 2 \lambda^2 \xi^2 \nabla  \psi \cdot \nabla \partial_t  \psi - \lambda^2 \xi^2 \partial_t  \psi \Delta  \psi
		  \label{coeff-w-s2-l1}
		   \\
		  & &+\frac{\partial_t \theta}{\theta} \left( - \lambda^2 \xi \varphi |\nabla  \psi|^2 + \lambda \xi \Delta  \psi \varphi -2 \lambda^2 \xi^2 |\nabla\psi|^2 \right). \label{coeff-w-s2-l2}
	\end{eqnarray}
	Before going further, let us remark that, using $\xi \geq 1$, there exists a positive constant $C$, only depending on the $C^2$-norm of $\psi$ such that for all $\lambda \geq 1$, for all $(t,x) \in (0,T) \times \Ocal$,
	\begin{equation*}
		\left|-  \lambda^3 \xi^2 \partial_t  \psi |\nabla  \psi|^2  - 2 \lambda^2 \xi^2 \nabla  \psi \cdot \nabla \partial_t  \psi - \lambda^2 \xi^2 \partial_t  \psi \Delta  \psi
		-2 \lambda^2 \xi^2 |\nabla\psi|^2
		\right|
		\leq
		C \lambda^3 \xi^3.
	\end{equation*}
	This estimate is sufficient to handle the terms in \eqref{coeff-w-s2-l1}.
	
	We will then focus on the terms in \eqref{coeff-w-s2-l2}. First remark that on $(T_0, T- 2T_1)$, $\partial_t \theta \equiv 0$, so the term in \eqref{coeff-w-s2-l2} simply vanishes.
	
	On $(T-2T_1, T)$, we use the fact that there exists a constant $C>0$ such that
	$$
		\forall t \in (T- 2T_1, T), \qquad |\partial_t \theta| \leq C \theta^2.
	$$
	Hence there exists $C = C (\|\nabla  \psi\|_\infty, \| \Delta  \psi\|_\infty) $ such that for all $(t,x) \in (T-2 T_1, T)\times \Ocal$,
		\begin{equation}
		\left| \frac{\partial_t \theta}{\theta} \left( - \lambda^2 \xi \varphi |\nabla  \psi|^2 + \lambda \xi \Delta  \psi \varphi -2 \lambda^2 \xi^2 |\nabla\psi|^2\right)\right|
		\\
		\leq
		C \lambda^2 \theta \xi \varphi
		\leq
		C \lambda^3 \xi^3,
		\label{coeff-w-s2-t=T}
	\end{equation}
	where for the last inequality we have used $|\theta\varphi|\leq \lambda \xi^2$, which is a consequence of \eqref{Phi-bounds}.

	On $(0,T_0)$, we are going to use that $\partial_t \theta \leq 0$ and $\theta \in [1, 2]$ and thus the term in \eqref{coeff-w-s2-l2} has the good sign outside ${\tilde \omega_T}$. Indeed, using \eqref{Ass-Psi}, we can find $\lambda_4 = \lambda_4( \alpha, \| \Delta  \psi\|_\infty) \geq \lambda_3$ such that for all $\lambda \geq \lambda_4$, for all $(t,x) \in (0,T_0) \times \Ocal$ such that $(t,x) \notin {\tilde \omega_T}$,
	\begin{equation*}
		-\left( - \lambda^2 \xi \varphi |\nabla  \psi|^2 + \lambda \xi \Delta  \psi \varphi -2 \lambda^2 \xi^2 |\nabla\psi|^2 \right)
		 \geq
		c \lambda^2  \xi \varphi,
	\end{equation*}
	whereas it is bounded by $C\lambda^2 \xi \varphi$ everywhere in $\Ocal_T$. We thus derive, for all $\lambda \geq \lambda_4$,
	\begin{multline}
		\label{Positivity-w-s2}
		s^2 \iint_{\Ocal_T} |w|^2 \overline \sigma \left( - \partial_t \left( |\nabla \varphi|^2 \right) -( \Delta \varphi +2 \lambda^2|\nabla  \psi|^2 \xi) \partial_t \varphi\right)
		\geq
		c s^2 \lambda^2  \int_0^{T_0}\int_\Ocal |\partial_t \theta| \xi  \varphi  |w|^2
		\\
		 - C  s^2 \lambda^3 \iint_{\Ocal_T} \xi^3 |w|^2
		- C s^2 \lambda^2  \iint_{{\tilde \omega_T} \cap \{ t \in (0,T_0)\} }  |\partial_t \theta| \xi  \varphi |w|^2.
	\end{multline}
	
	{\it Term \eqref{I-w-s} involving $w$ in the scale $s$.} We have to estimate $-\partial_{tt} \varphi$.
	\begin{equation}
		 \partial_{tt} \varphi
		=    \frac{\partial_{tt} \theta}{\theta} \varphi - 2 \lambda \frac{\partial_t \theta}{\theta} \partial_t  \psi \xi - \lambda \partial_{tt}  \psi \xi -\lambda^2 (\partial_t  \psi)^2 \xi
		\label{Est-dtt-phi}
	\end{equation}
	Let us first remark that we immediately have
	$$
		\left| - \lambda \partial_{tt}  \psi \xi -\lambda^2 (\partial_t  \psi)^2 \xi\right| \leq C \lambda^2 \xi^3.
	$$
	
	For $t \in (0,T_0)$, we further have
	$$
		\forall t \in (0,T_0), \quad \left| \partial_{tt} \theta \right| \leq C s^2 \lambda^4 e^{\lambda (12m - 8)}, \quad | \partial_{t} \theta| \leq C s \lambda^2 e^{\lambda (6m-4)} ,
	$$
	so that, on $(0,T_0)$
	$$
		|\partial_{tt} \varphi| \leq C s^2 \lambda^5 e^{\lambda (12m-8)} e^{6\lambda (m+1)} + C s \lambda^3 e^{\lambda (6m-4)} \xi + C \lambda^2 \xi^3 \leq C s^2 \lambda^2 \xi^3.
	$$
	
	For $t \in (T-2T_1, T)$, we have
	$$
		\forall t \in (T-2T_1, T), \quad |\partial_{tt} \theta| \leq C \theta^3 \quad \hbox{ and } \quad |\partial_t \theta| \leq C \theta^2.
	$$
	Hence, using \eqref{Phi-bounds} and $\theta \varphi \leq \lambda \xi^2$, for some positive constant $C = C(\|\partial_t  \psi\|_\infty)$,
	\begin{equation*}
		\forall (t,x) \in (T-2T_1,T)\times \Ocal, \quad
		| \partial_{tt} \varphi| \leq  C \theta^2 \varphi + C \lambda \theta \xi + C \lambda^2 \xi^3
		\leq C \lambda^2 \xi^3.
	\end{equation*}
	
	Combining all these estimates, we get
	\begin{equation}
		\label{Positivity-w-s1}
		s \iint_{\Ocal_T} \overline{\sigma}^2 |w|^2 \left(- \frac{1}{ 2} \partial_{tt} \varphi \right)
		\geq
		- C s^3 \lambda^2 \iint_{\Ocal_T} \xi^3 |w|^2.
	\end{equation}

	{\it Positivity of the terms \eqref{I-w-s3}--\eqref{I-w-s2}--\eqref{I-w-s} involving $w$.} Here we combine the estimates in \eqref{Positivity-w-s3}, \eqref{Positivity-w-s2}, \eqref{Positivity-w-s1} in order to derive suitable estimates for the sum of the terms in \eqref{I-w-s3}--\eqref{I-w-s2}--\eqref{I-w-s}. To simplify notations, let us set $I_w$ the sum of the terms in \eqref{I-w-s3}--\eqref{I-w-s2}--\eqref{I-w-s}:
	\begin{multline}
		\label{Def-Iw}
		 I_w  \ov  \iint_{\Ocal_T} |w|^2
		\Bigg(
			s^3 \left( - \div (|\nabla \varphi|^2 \nabla \varphi)  - 2 \lambda^2 |\nabla  \psi|^2\xi |\nabla \varphi|^2 \right)
		\\
			+
			s^2 \overline{\sigma} \left( - \partial_t \left(|\nabla \varphi|^2\right)  -( \Delta \varphi + 2 \lambda^2 |\nabla  \psi|^2 \xi )\partial_t \varphi \right)
			+
			s \overline{\sigma}^2 \left(- \frac{1}{ 2} \partial_{tt} \varphi   \right)
		\Bigg).
	\end{multline}
	Putting together \eqref{Positivity-w-s3}, \eqref{Positivity-w-s2}, \eqref{Positivity-w-s1}, we deduce that
	there exist $s_1 \geq 1$ and $\lambda_5\geq \lambda_4$ such that for $s \geq s_1$ and $\lambda \geq \lambda_5$,
	\begin{align}
		I_w
		\geq &
		c s^3 \lambda^4 \iint_{\Ocal_T}  \xi^3 |w|^2 	+ c s^2 \lambda^2 \int_0^{T_0}\int_\Ocal |\partial_t \theta| \xi \varphi |w|^2
		\\
		&
		- C s^3 \lambda^4 \iint_{{\tilde \omega_T}} \xi^3 |w|^2
		- C s^2 \lambda^2 \iint_{{\tilde \omega_T} \cap \{ t \in (0,T_0)\} }  |\partial_t \theta| \xi \varphi  |w|^2.
	\end{align}

	{\it Positivity of the boundary terms \eqref{I-boundary}.} Here, we only have to remark that $\partial_{\bf n} \varphi \geq 0$ since $\partial_{\bf n} \psi \leq 0$ by construction, see \eqref{Psi}.
	
	{\it A bound on $I_{\mathcal{R}}$ in \eqref{Def-mathcalR}}
	We also provide an upper bound on $I_{\mathcal{R}}$.
	
	First, we shall of course use the immediate estimate
	$$
		\frac{1}{2} \iint_{\Ocal_T} \partial_t \overline{\sigma} |\nabla w|^2 \leq C \iint_{\Ocal_T}  |\nabla w|^2.
	$$
	Using $\nabla(|\nabla  \psi|^2 \xi )\leq C \lambda \xi$, one easily checks that
	\begin{equation}
		\left| 2 s \lambda^2 \iint_{\Ocal_T} \nabla(|\nabla  \psi|^2 \xi ) w \cdot \nabla w
		\right| \leq 	
		C s^2 \lambda^4 \iint_{\Ocal_T} \xi^3 |w|^2 + C \lambda^2 \iint_{\Ocal_T} \xi |\nabla w|^2.
	\end{equation}
	
	Using \eqref{ExpressDtPhi}, we have
	$$
	|\partial_t \varphi| \leq
		\left\{
			\begin{array}{ll}
				 s \lambda^2 e^{\lambda (6m-4)}  \lambda e^{6 \lambda(m+1)} + C \lambda \xi &\text{ on } (0,T_0),
				 \\
				  C \lambda \xi  &\text{ on } (T_0, T - 2T_1),
				  \\
				  \theta \lambda e^{6 \lambda(m+1)} + C \lambda \xi & \text{ on } (T-2T_1,T),
			\end{array}
		\right.
	$$
	so that $|\partial_t \varphi|  \leq C s \lambda \xi^3$ everywhere. Hence
	\begin{equation}
		\left|
		s\iint_{\Ocal_T} \overline \sigma \partial_t \overline \sigma \partial_t \varphi |w|^2
		\right|
		\leq
		C s^2 \lambda \iint_{\Ocal_T} \xi^3 |w|^2.
	\end{equation}
	
Moreover, using $|\nabla \varphi| \leq C \lambda \xi$, \eqref{ExpressDtPhi} and $\theta \varphi\leq \lambda\xi^2$ we also obtain
	\begin{align*}
		& \left| \frac{s^2}{2} \iint_{\Ocal_T} \partial_t \overline \sigma |\nabla \varphi|^2 |w|^2 \right|
		\leq
		C s^2 \lambda^2 \iint_{\Ocal_T}  \xi^2 |w|^2,
		\\
		&\left|
			s^2 \iint_{\Ocal_T} \nabla \overline \sigma \nabla \varphi \partial_t \varphi |w|^2
		\right|
		\leq
			C s^2\lambda \int_0^{T_0}\int_\Ocal \xi \varphi |\partial_t \theta|   |w|^2
			+
			C s^2 \lambda^2 \iint_{\Ocal_T} \xi^3  |w|^2.
	\end{align*}

	Finally, we also have
	\begin{equation}
		\left| \iint_{\Ocal_T} \partial_t w \nabla \overline \sigma \cdot \nabla w\right|
		\leq
		C\frac{1}{s \lambda} \iint_{\Ocal_T} \frac{1}{\xi} |\partial_t w|^2 + C s \lambda \iint_{\Ocal_T} \xi |\nabla w|^2,
	\end{equation}
	and combining all the above estimates,
	\begin{multline}
	\label{Est-mathcal-R}
		|I_{\mathcal{R}}|
		\leq
		\frac{C}{s \lambda} \iint_{\Ocal_T} \frac{1}{\xi} |\partial_t w|^2
		+
		C s \lambda \iint_{\Ocal_T} \xi |\nabla w|^2
		\\
		+		
		C s^2\lambda  \int_0^{T_0}\int_\Ocal \xi |\partial_t \theta| \varphi  |w|^2
		+
		C s^2 \lambda^4 \iint_{\Ocal_T} \xi^3  |w|^2.
	\end{multline}
	
	{\it A lower bound for the cross-product $\iint P_1 w P_2 w$.} This step simply consists in putting together all the above estimates: for all $s \geq s_1$ and $\lambda \geq \lambda_5$,
	\begin{multline*}
		2 \iint_{\Ocal_T} P_1 w P_2 w
		\geq
		  \int_\Ocal |\nabla w(0)|^2 +c s^2 \lambda^3 e^{12 \lambda m + 2 } \int_\Ocal |w(0)|^2
		\\
		\begin{array}{l}
			 \ds + c s \lambda^2 \iint_{\Ocal_T}  \xi |\nabla w|^2  -  C s \lambda^2 \iint_{{\tilde \omega_T}}  \xi |\nabla w|^2
			\smallskip\\
			 \ds +c s^3 \lambda^4 \iint_{\Ocal_T}  \xi^3 |w|^2 	+  c s^2 \lambda^2 \int_0^{T_0}\int_\Ocal |\partial_t \theta| \xi \varphi |w|^2
			\smallskip\\
			 \ds -C s^3 \lambda^4 \iint_{{\tilde \omega_T}} \xi^3 |w|^2
			- C s^2 \lambda^2 \iint_{{\tilde \omega_T} \cap \{ t \in (0,T_0)\} }  |\partial_t \theta| \xi \varphi |w|^2 - |I_{\mathcal{R}}|.
		\end{array}
	\end{multline*}
	
	Thus, using \eqref{Est-mathcal-R}, for some $s_2 \geq s_1$ and $\lambda_6 \geq \lambda_5$, for all $s \geq s_2$ and $\lambda \geq \lambda_6$
	\begin{multline}
		\label{Est-Cross-Prod}
		2 \iint_{\Ocal_T} P_1 w P_2 w
		\geq
		 \int_\Ocal |\nabla w(0)|^2 +c s^2 \lambda^3 e^{12 \lambda m + 2 } \int_\Ocal |w(0)|^2
		\\
		\begin{array}{l}
			 \ds + c s \lambda^2 \iint_{\Ocal_T}  \xi |\nabla w|^2  -  C s \lambda^2 \iint_{{\tilde \omega_T}}  \xi |\nabla w|^2
			 \ds +c s^3 \lambda^4 \iint_{\Ocal_T}  \xi^3 |w|^2
			 	\smallskip\\
			\ds +	c s^2 \lambda^2 \int_0^{T_0}\int_\Ocal |\partial_t \theta| \xi \varphi |w|^2
			 -C s^3 \lambda^4 \iint_{{\tilde \omega_T}} \xi^3 |w|^2
			 			\smallskip\\
			 \ds
			- C s^2 \lambda^2 \iint_{{\tilde \omega_T} \cap \{ t \in (0,T_0)\} }  |\partial_t \theta| \xi \varphi |w|^2
			-\frac{C}{s \lambda} \int_0^T \int_\Ocal \frac{1}{\xi} |\partial_t w|^2.
		\end{array}
	\end{multline}
	
	{\bf Conclusion.}
		We first derive a Carleman estimate on $w$ with gradient observations, and then explains how to remove this term using a suitable multiplier.

	{\it A Carleman estimate on $w$ with gradient observations.} According to estimates \eqref{Square-Pi} and \eqref{Est-Cross-Prod}, for all $s \geq s_2$ and $\lambda \geq \lambda_6$,
	\begin{multline*}
		\iint_{\Ocal_T} \left( |P_1 w|^2 + |P_2 w|^2 \right)
		+  c \int_\Ocal |\nabla w(0)|^2 +c s^2 \lambda^3 e^{12 \lambda m + 2 } \int_\Ocal |w(0)|^2
		\\
		\begin{array}{l}
			 \ds + c s \lambda^2 \iint_{\Ocal_T} \ \xi |\nabla w|^2
			 \ds + cs^3 \lambda^4 \iint_{\Ocal_T}  \xi^3 |w|^2 	
			 +
			c s^2 \lambda^2 \int_0^{T_0}\int_\Ocal |\partial_t \theta| \varphi \xi  |w|^2
			\smallskip\\
			\ds  \leq C \iint_{\Ocal_T} |f|^2 e^{-2 s \varphi} + C \iint_{\Ocal_T} |Rw|^2
			+  C s \lambda^2 \iint_{{\tilde \omega_T}}  \xi |\nabla w|^2
			\smallskip
			\\
			\ds
			 + C s^3 \lambda^4 \iint_{{\tilde \omega_T}} \xi^3 |w|^2
			+ C s^2 \lambda^2 \iint_{{\tilde \omega_T} \cap \{ t \in (0,T_0)\} }  |\partial_t \theta| \xi \varphi |w|^2
			\ds +\frac{C}{s \lambda} \iint_{\Ocal_T} \frac{1}{\xi} |\partial_t w|^2.
		\end{array}
	\end{multline*}
	To handle the term $\| R w\|_{L^2}^2$, we recall that $Rw $ is given by \eqref{R}, hence
	$$
		 \iint_{\Ocal_T} |Rw|^2 \leq C s^2 \lambda^4 \iint_{\Ocal_T}  \xi^3 |w|^2.
	$$
	where $C = C(\|\nabla  \psi\|_\infty, \|\Delta  \psi\|_\infty)$ is a positive constant.

	Also note that
	$$
		\frac{1}{s \lambda} \iint_{\Ocal_T} \frac{1}{\xi} |\partial_t w|^2
		\leq
		\frac{C}{s\lambda} \iint_{\Ocal_T} |P_1 w|^2 + C s \lambda \iint_{\Ocal_T} \xi |\nabla w|^2 + C s \lambda^3 \iint_{\Ocal_T} \xi^3 |w|^2.
	$$
	In particular, for some $s_3 \geq s_2$, for all $s \geq s_3 $ and $\lambda \geq \lambda_6$,
	\begin{multline}
		\label{Est-Complete}
		\iint_{\Ocal_T} \left( |P_1 w|^2 + |P_2 w|^2 \right)
		+  c \int_\Ocal |\nabla w(0)|^2 +c s^2 \lambda^3 e^{12 \lambda m + 2 } \int_\Ocal |w(0)|^2
		\\
		\begin{array}{l}
			 \ds + c s \lambda^2 \iint_{\Ocal_T} \ \xi |\nabla w|^2
			 \ds + cs^3 \lambda^4 \iint_{\Ocal_T}  \xi^3 |w|^2 	+
			c s^2 \lambda^2 \int_0^{T_0}\int_\Ocal |\partial_t \theta| \xi  \varphi |w|^2
			\smallskip\\
			\ds  \leq C \iint_{\Ocal_T} |f|^2 e^{-2 s \varphi}
			+  C s \lambda^2 \iint_{{\tilde \omega_T}}  \xi |\nabla w|^2
			\smallskip
			\\
			\ds
			 + C s^3 \lambda^4 \iint_{{\tilde \omega_T}} \xi^3 |w|^2
			+ C s^2 \lambda^2 \iint_{{\tilde \omega_T} \cap \{ t \in (0,T_0)\} }  |\partial_t \theta| \xi  \varphi |w|^2.
		\end{array}
	\end{multline}
	In \eqref{Est-Complete}, the observation is done on ${\tilde \omega_T}$ and concerns both $w$ and $\nabla w$. Below, we shall explain that this observation can be done only on $w$ provided we take an observation set slightly larger.

	{\it A Carleman estimate on $w$ without gradient observations.} Recall that ${\tilde \omega_T} \Subset \widehat{ \omega_T}$, then there exists a nonnegative smooth function $\eta = \eta(t,x)$ taking value in $[0,1]$ such that $\eta = 1$ on ${\tilde \omega_T}$, and $\eta = 0$ in $(0,T) \times \Ocal \setminus \widehat{ \omega_T}$. We then compute the scalar product of $P_2 w$ and $\eta s \lambda^2 \xi w$:
	\begin{multline*}
		\iint_{\Ocal_T} P_2 w ( \eta s \lambda^2 \xi w)
		= s \lambda^2 \iint_{\Ocal_T} \eta \xi |\nabla w|^2  - \frac{s \lambda^2}{2} \iint_{\Ocal_T} \Delta(\eta \xi) |w|^2
		\\
		 - s^2 \lambda^2 \iint_{\Ocal_T} \eta \overline \sigma \partial_t \varphi \xi |w|^2
		 - s^3 \lambda^2 \iint_{\Ocal_T} \eta |\nabla \varphi|^2 \xi |w|^2.
	\end{multline*}	
	In particular, using \eqref{ExpressDtPhi} and \eqref{Phi-bounds},
	\begin{multline*}
		s \lambda^2 \iint_{\Ocal_T} \eta \xi |\nabla w|^2 + c s^2 \lambda^2 \int_0^{T_0}\int_\Ocal \overline \sigma \eta |\partial_t \theta|  \xi \varphi |w|^2
		\\
		\leq \iint_{\Ocal_T} P_2 w ( \eta s \lambda^2 \xi w)  +  s \lambda^2  \int_{\Ocal_T} |\Delta (\eta \xi) | |w|^2
		\\
		+ s^2 \lambda^3 e^{6\lambda (m+1)} \int_{T-2T_1}^T\int_\Ocal \eta \overline \sigma |\partial_t \theta|  \xi |w|^2
		+s^2 \lambda^3 \iint_{\Ocal_T} \eta \overline \sigma | \partial_t  \psi| \xi^2 |w|^2
		\\
		+ s^3 \lambda^4 \iint_{\Ocal_T} \eta |\nabla  \psi|^2 \xi^3 |w|^2.
	\end{multline*}
	Of course, this implies that
	\begin{multline*}
		 s \lambda^2 \iint_{{\tilde \omega_T}}  \xi |\nabla w|^2 + s^2 \lambda^2 \iint_{{\tilde \omega_T} \cap \{ t \in (0,T_0)\} } \overline \sigma \eta |\partial_t \theta|  \xi \varphi |w|^2
		\\
		\leq \frac{1}{\sqrt{s}}  \iint_{\Ocal_T} |P_2 w|^2 + C s^{5/2} \lambda^4 \iint_{\Ocal_T}  \eta^2 \xi^2| w|^2
		  + 2 C s \lambda^2 \iint_{\Ocal_T} |\Delta (\eta \xi) | |w|^2
		  \\
		+ 2 CÂ s^2 \lambda^3 e^{6\lambda (m+1)}\int_{T-2T_1}^T\int_\Ocal \eta |\partial_t \theta|  \xi  |w|^2
		+ 2 C s^2 \lambda^3 \iint_{\Ocal_T} \eta | \partial_t  \psi| \xi^2 |w|^2
		\\
		+ 2 C s^3 \lambda^4 \iint_{\widehat {\omega_T}}  |\nabla  \psi|^2 \xi^3 |w|^2.
	\end{multline*}

	But there exists a constant $C = C(\|\eta\|_{L^\infty(C^2)},\| \nabla  \psi\|_\infty,  \| \Delta  \psi\|_\infty, \| \partial_t  \psi\|_\infty )$ such that
	$$
		 |\Delta (\eta \xi) | \leq C \lambda^2 \xi^2, \quad \sup_{[T-2T_1,T)} \left\{\frac{|\partial_t \theta|}{\theta^2} \right\} \leq C,
	$$	
	hence, using the fact that $\eta$ is supported on $\widehat{ \omega_T}$,
	\begin{multline*}
		s^{5/2} \lambda^4 \iint_{\Ocal_T}  \eta^2 \xi^2| w|^2
		+
		 s \lambda^2 \iint_{\Ocal_T} |\Delta (\eta \xi) | |w|^2
		  +
		  s^2 \lambda^3 \iint_{\Ocal_T} \eta | \partial_t  \psi| \xi^2 |w|^2
		 \\
		 \leq
		 C s^3 \lambda^4 \iint_{{ \widehat \omega_T}} \xi^3 |w|^2,
	\end{multline*}
	whereas
	\begin{eqnarray*}
		s^2 \lambda^3 e^{6\lambda (m+1)} \int_{T-2T_1}^T\int_\Ocal \eta|\partial_t \theta|  \xi  |w|^2
		 & \leq &
		C s^2 \lambda^3 e^{6\lambda (m+1)} \int_{T-2T_1}^T\int_\Ocal \eta \theta^2 \xi  |w|^2
		\\
		& \leq &
		C s^3 \lambda^4 \iint_{{ \widehat \omega_T}} \xi^3 |w|^2.
	\end{eqnarray*}
	
	Hence, by combining above estimates with \eqref{Est-Complete}, for some $s_4\geq s_3$ and $\lambda_7 \geq \lambda_6$, there exists a constant $C$ such that for all $s \geq s_4$ and $\lambda\geq \lambda_7$,
	\begin{multline}
		\label{Est-Complete-Obs-w}
		  \int_\Ocal |\nabla w(0)|^2 + s^2 \lambda^3 e^{\lambda( 12 m+2) } \int_\Ocal |w(0)|^2 + s \lambda^2\iint_{\Ocal_T}  \xi |\nabla w|^2
		\\
			 \ds + s^3 \lambda^4 \iint_{\Ocal_T} \xi^3 |w|^2 	+s^2 \lambda^2 \int_0^{T_0}\int_\Ocal|\partial_t \theta| \xi \varphi |w|^2
		\\
			\ds  \leq C \iint_{\Ocal_T} |f|^2 e^{-2 s \varphi}
			 + C s^3 \lambda^4 \iint_{\widehat{\omega}_T} \xi^3 |w|^2.
	\end{multline}
	
	{\it Back to the function $z$.} We now go back to the function $z = w e^{s \varphi}$. For that, let us first remark that there exists a constant $C = C(\|\nabla  \psi\|_\infty)$ such that for all $(t,x) \in (0,T) \times \Ocal$,
	\begin{align*}
		& |z|^2 e^{-2 s \varphi} = |w|^2,
		\\
		& |\nabla z|^2 e^{-2 s \varphi} \leq 2 |\nabla w|^2 + 2 s^2 |\nabla \varphi|^2 |w|^2 \leq 2 |\nabla w|^2 + 2 C s^2 \lambda^2 \xi^2 |w|^2.
	\end{align*}
	
	We immediately deduce from \eqref{Est-Complete-Obs-w} that for all $s \geq s_4$ and $\lambda \geq \lambda_7$, for some positive constant $C$,
	\begin{multline}
		\label{Est-Complete-Obs-z}
		  \int_\Ocal |\nabla z(0)|^2 e^{-2 s \varphi(0)} + s^2 \lambda^3 e^{\lambda( 12 m+2) } \int_\Ocal |z(0)|^2 e^{-2 s \varphi(0)} + s \lambda^2 \iint_{\Ocal_T}  \xi |\nabla z|^2  e^{-2 s \varphi}
		\\
			 \ds + s^3 \lambda^4\iint_{\Ocal_T} \xi^3 |z|^2 e^{-2 s \varphi} 	+  s^2 \lambda^2 \int_0^{T_0}\int_\Ocal|\partial_t \theta| \xi  \varphi |z|^2 e^{-2 s \varphi}
		\\
			\ds  \leq C \iint_{\Ocal_T} |f|^2 e^{-2 s \varphi}
			 + C s^3 \lambda^4 \iint_{\widehat{ \omega_T}} \xi^3 |z|^2 e^{-2 s \varphi}.
	\end{multline}
	We conclude the proof of Theorem \ref{CarlemanThm} by setting $s_0 =s_4$ and $\lambda_0 = \lambda_7$.




\subsection{Proof of Theorem \ref{Thm-Est-Y-Carl-Norms}}\label{Sec-Proof-Control-Heat}

We divide the proof in several steps.

{\bf A duality approach.}
To solve the control problem \eqref{Heat-control}--\eqref{Null-Control-Req}, we first rewrite the control problem under a weak form. Multiplying $y$ solution of \eqref{Heat-control} by smooth functions $z$ on $[0,T] \times \overline{\Ocal}$ such that $z = 0$ on $[0,T] \times \partial \Ocal$, we get:
\begin{equation}
	\label{Duality}
	\int_\Ocal \overline{\sigma}(T)y(T) z(T) + \iint_{\Ocal_T} y (- \overline \sigma \partial_t z - \Delta z)
		=  \iint_{\Ocal_T} f z  + \iint_{{ \widehat{\omega}_T}} h z.
\end{equation}

In particular, since $\overline{\sigma}(T)>0$, the null-controllability requirement \eqref{Null-Control-Req} is satisfied if and only if for all smooth functions $z$ on $[0,T] \times \overline{\Ocal}$ such that $z = 0$ on $[0,T] \times \partial \Ocal$
\begin{equation}
	\label{Duality-ControlReq}
	  \iint_{\Ocal_T} y (- \overline \sigma \partial_t z - \Delta z) \\
		= \iint_{\Ocal_T} f z + \iint_{{ \widehat{\omega}_T}} h z.
\end{equation}

The trick now is to introduce a functional $J$ whose Euler Lagrange equation coincide with \eqref{Duality-ControlReq}: For smooth functions $z$ on $[0,T] \times \overline{\Ocal}$ such that $z = 0$ on $[0,T] \times \partial \Ocal$, we define
\begin{equation}
	\label{Func-J}
	J(z) =  \frac{1}{2} \iint_{\Ocal_T} |(- \overline \sigma \partial_t - \Delta) z|^2 e^{-2 s \varphi}
		+  \frac{s^3 \lambda^4}{2} \iint_{{ \widehat{\omega}_T}} \xi^3 |z|^2 e^{-2 s \varphi}
		- \iint_{\Ocal_T} f z.
\end{equation}

But the set of smooth functions $z$ on $[0,T] \times \overline{\Ocal}$ such that $z = 0$ on $[0,T] \times \partial \Ocal$ is not a Banach space. We thus introduce
\begin{equation}
	\label{CarlemanSpace}
	X_{obs} = \overline{\{ z \in C^{\infty}([0,T] \times \overline\Ocal) \hbox{\, such that } z = 0 \hbox{ on } [0,T] \times \partial \Ocal \}}^{\norm{\cdot }_{obs}}
\end{equation}
where $\norm{\cdot}_{obs}$ is the Hilbert norm defined by
\begin{equation}
	\norm{z}_{obs}^2 =
		\iint_{\Ocal_T} |(- \overline \sigma \partial_t - \Delta) z|^2 e^{-2 s \varphi}
		+  s^3 \lambda^4 \iint_{{ \widehat{\omega}_T }} \xi^3 |z|^2 e^{-2 s \varphi}.
\end{equation}
The set $X_{obs}$ is then endowed with the Hilbert structure given by $\norm{\cdot}_{obs}$. Note that here we use the fact that $\norm{\cdot}_{obs}$ is a norm, which is a consequence of the Carleman estimate \eqref{CarlemanEst}. Also note that $X_{obs}$ and $\norm{\cdot}_{obs}$ strongly depends on $s$ and $\lambda$ and we shall follow these dependences carefully in the sequel.

The functional $J$ can be extended as a continuous functional on $X_{obs}$ provided \eqref{Conditions-f-F}. Indeed, due to \eqref{CarlemanEst}, we easily have, for some constant $C >0$ independent of $s$ and $\lambda$,
\begin{equation}
	\label{Continuity-J}
	\left|
		\iint_{\Ocal_T} f z
	\right|
	 \leq C \norm{z}_{obs}
	 \left(
	 \frac{1}{s^3\lambda^4 }\iint_{\Ocal_T} \xi^{-3} |f|^2 e^{2 s \varphi}
	\right)^{1/2}.
\end{equation}
It follows that, if condition \eqref{Conditions-f-F} is satisfied, the functional $J$ can be uniquely extended as a continuous functional (still denoted the same) on $X_{obs}$. Besides, \eqref{Continuity-J} also implies the coercivity of $J$ on $X_{obs}$. Since it is also strictly convex on $X_{obs}$ since $\norm{\cdot}_{obs}$ is an Hilbert norm, $J$ admits a unique minimizer $Z$ on $X_{obs}$.

Setting
\begin{equation}
	\label{Link-Y-Z}
	Y = (- \overline \sigma \partial_t - \Delta ) Z e^{-2 s \varphi} \quad \hbox{ and } \quad H = -s^3 \lambda^4 \xi^3 Z e^{-2s \varphi} 1_{\widehat{\omega}_T},
\end{equation}
writing the Euler Lagrange equation of $J$ at $Z$, for all smooth functions $z$ on $[0,T] \times \overline{\Ocal}$ such that $z = 0$ on $[0,T] \times \partial \Ocal$,
\begin{equation}
	\label{Duality-ControlReq-Y}
	0 =  \iint_{\Ocal_T} Y (- \overline \sigma \partial_t z - \Delta z) -  \iint_{\widehat{\omega}_T} H z
	- \iint_{\Ocal_T} f z,
\end{equation}
which coincides with \eqref{Duality-ControlReq}.

In particular, \eqref{Duality-ControlReq-Y} holds  for all smooth functions $z$ on $[0,T] \times \overline{\Ocal}$ such that $z = 0$ on $[0,T] \times \partial \Ocal$ with $z(T) \equiv 0$, which implies that $Y$ solves the equation \eqref{Heat-control} with $h = H$ in the sense of transposition. By uniqueness of solutions in the sense of transposition, this is the solution of \eqref{Heat-control} in the classical sense. In particular, since $H \in L^2(\Ocal_T)$, $Y$ is $C([0,T];L^2(\Ocal))$. Then, using again \eqref{Duality-ControlReq-Y}, we remark that it coincides with \eqref{Duality-ControlReq}, hence $Y$ solves the control requirement \eqref{Null-Control-Req}.

Besides, using \eqref{Continuity-J} and the fact that $J(Z) \leq J(0) = 0$,
\begin{equation}
	\label{Estimate-Y-V}
	s^3 \lambda^4 \iint_{\Ocal_T} |Y|^2 e^{2 s \varphi} +  \iint_{\widehat{\omega}_T}\xi^{-3} |H|^2 e^{2 s \varphi}
	\leq
	 C\iint_{\Ocal_T} \xi^{-3} |f|^2 e^{2 s \varphi}.
\end{equation}

{\bf Estimates on $\nabla Y$.}
In the previous step, we found $(Y,H)$ satisfying the equations
\begin{equation}
	\label{Heat-control-Y}
	\left\lbrace
		\begin{array}{rlll}
			\partial_t (\overline \sigma Y) - \Delta Y & =& f + H {1}_{ \widehat{\omega}_T}, \quad & \hbox{ in } \Ocal_T,
			\\
			Y &=& 0, \quad & \hbox{ in } \Gamma_T,
			\\
			Y(0,\cdot) &=& 0, \quad & \hbox{ in }\Ocal,
			\\
			Y(T,\cdot) &=& 0, \quad & \hbox{ in }\Ocal.
		\end{array}
	\right.
\end{equation}
and the estimates \eqref{Estimate-Y-V}.

Our goal now is to obtain an estimate on $\nabla Y$. In order to do this, for $\varepsilon >0$, we introduce
$$
	\varphi_\varepsilon (t,x) \ov \theta_\varepsilon(t) \left( \lambda e^{6 \lambda (m+1)} - e^{\psi(t,x)} \right), \quad \xi_\varepsilon(t) \ov  \theta_\varepsilon(t)e^{\psi(t,x)}
$$
and $\theta_\varepsilon$ is given by:
$$
	\theta_{\varepsilon} \ov \theta_{\varepsilon}(t) \quad
	 \hbox{ such that }
		\left\{
			\begin{array}{l}
			\ds \forall t \in [0,T_0],\, \theta_{\varepsilon}(t) = 1+ \left( 1- \frac{t}{T_0} \right)^\mu,
			\smallskip
			\\
			\ds \forall t \in [T_0, T- 2T_1+\varepsilon], \, \theta_{\varepsilon}(t) = 1,
			\smallskip
			\\
			\ds \forall t \in [T-2T_1+ \varepsilon,T), \,  \theta_{\varepsilon}(t) = \theta(t - \varepsilon),
			\smallskip
			\\
			\ds \mu \hbox{ as in \eqref{Def-mu}}.
			\end{array}
		\right.
$$
We then multiply the equation \eqref{Heat-control-Y} by $\xi_\varepsilon^{-2} Y e^{2 s \varphi_\varepsilon}$:
\begin{multline*}
	 - \frac{1}{2} \iint_{\Ocal_T} |Y|^2 \partial_t \left( \overline \sigma \xi_\varepsilon^{-2} e^{2 s \varphi_\varepsilon} \right)
	 + \iint_{\Ocal_T} |Y|^2 \partial_t \overline \sigma \xi_\varepsilon^{-2} e^{2 s \varphi_\varepsilon}
	\\
	+
	\iint_{\Ocal_T} \xi_\varepsilon^{-2} |\nabla Y|^2 e^{2 s \varphi_\varepsilon} - \frac{1}{2} \iint_{\Ocal_T} |Y|^2 \Delta \left( \xi_\varepsilon^{-2} e^{2 s \varphi_\varepsilon} \right)
	\\
	\ds = \iint_{\Ocal_T} f \xi_\varepsilon^{-2} Y e^{2 s \varphi_\varepsilon}
	+ \iint_{\widehat{\omega}_T} H \xi_\varepsilon^{-2} Y e^{2 s \varphi_\varepsilon}.
\end{multline*}
Following, multiplying by $s \lambda^2$,
\begin{multline}
	\label{Est-GradY-1}
	s \lambda^2 \iint_{\Ocal_T} \xi_\varepsilon^{-2} |\nabla Y|^2 e^{2 s \varphi_\varepsilon}  - \frac{s \lambda^2}{2} \int_0^{T_0}\int_\Ocal \overline \sigma |Y|^2 \partial_t \left( \xi_\varepsilon^{-2} e^{2 s \varphi_\varepsilon} \right)
	\\
	=
	\frac{s \lambda^2}{2} \int_{T_0}^T\int_\Ocal \overline \sigma |Y|^2 \partial_t \left( \xi_\varepsilon^{-2} e^{2 s \varphi_\varepsilon} \right)
	+
	 \frac{s \lambda^2}{2} \iint_{\Ocal_T} |Y|^2 \Delta \left( \xi_\varepsilon^{-2} e^{2 s \varphi_\varepsilon} \right)
	\\
	+
	s \lambda^2  \iint_{\Ocal_T} f \xi_\varepsilon^{-2} e^{2 s \varphi_\varepsilon}  Y
	   +
	s \lambda^2 \iint_{\widehat{\omega}_T } H \xi_\varepsilon^{-2} Y e^{2 s \varphi_\varepsilon}
	- \frac{s \lambda^2}{2} \iint_{\Ocal_T} |Y|^2 \partial_t \overline \sigma \xi_\varepsilon^{-2} e^{2 s \varphi_\varepsilon}.
\end{multline}
We then compute explicitly:
\begin{multline}
	\label{Term-Dt-t=0}
	 - e^{-2 s \varphi_\varepsilon}\partial_t \left( \xi_\varepsilon^{-2} e^{2 s \varphi_\varepsilon} \right)
	=
	2 s \lambda \xi_\varepsilon^{-1} \partial_t  \psi
- 2 s \xi_\varepsilon^{-2} \frac{\partial_t \theta_\varepsilon}{\theta_\varepsilon} \varphi_\varepsilon
+ 2 \frac{\partial_t \theta_\varepsilon}{\theta_\varepsilon} \xi_\varepsilon^{-2}
+ 2 \lambda \partial_t  \psi \xi_\varepsilon^{-2}.
\end{multline}
On $(0,T_0)$, we remove the dependence in $\varepsilon >0$ as $\theta_\varepsilon = \theta$ on $(0,T_0)$. Using \eqref{Phi-bounds}, $\partial_t \theta\leq 0$ and $\theta\in [1,2]$ in $[0,T_0]$ we have, for all $s\geq s_0$ and $t \in (0,T_0)$,
$$
	- 2 s \xi^{-2} \frac{\partial_t \theta}{\theta} \varphi
+  2 \frac{\partial_t \theta}{\theta} \xi^{-2}
\geq   c s |\partial_t \theta| \xi^{-2} \varphi,
$$
whereas
$$
	\left| 2 s \lambda \xi^{-1} \partial_t  \psi   + 2 \lambda \partial_t  \psi \xi^{-2} \right|	
	\leq
	C s \lambda \xi^{-1}.
$$
Hence
\begin{multline}
\label{Est-Dy-Dt-t=0}
	- \frac{s \lambda^2}{2} \int_0^{T_0}\int_\Ocal \overline \sigma |Y|^2 \partial_t \left( \xi_\varepsilon^{-2} e^{2 s \varphi_\varepsilon} \right)
	\geq
	c s^2 \lambda^2 \int_0^{T_0}\int_\Ocal \overline \sigma |\partial_t \theta| \xi^{-2} \varphi |Y|^2 e^{2 s \varphi}
	\\
	 - C s^2 \lambda^3 \int_0^{T_0}\int_\Ocal |Y|^2 e^{2 s \varphi}.
\end{multline}

On $(T_0, T)$, from the identity \eqref{Term-Dt-t=0}, using $|\partial_t \theta_\varepsilon| \leq C \theta_\varepsilon^2$, we derive
$$
	\left|
	 2 s \lambda \xi_\varepsilon^{-2} \partial_t  \psi  - 2 s \xi_\varepsilon^{-2} \frac{\partial_t \theta_\varepsilon}{\theta_\varepsilon} \varphi + 2 \frac{\partial_t \theta_\varepsilon}{\theta_\varepsilon} \xi_\varepsilon^{-2}  + 2 \lambda \partial_t  \psi \xi_\varepsilon^{-2}\right|
	 \leq C s \lambda
$$
We thus obtain
\begin{equation}
\label{Est-Dy-Dt-t=T}
	\left|\frac{s \lambda^2}{2} \int_{T_0}^T\int_\Ocal \overline \sigma |Y|^2 \partial_t \left( \xi_\varepsilon^{-2} e^{2 s \varphi_\varepsilon} \right) \right|
	\leq C s^2 \lambda^3 \int_{T_0}^T\int_\Ocal |Y|^2 e^{2 s \varphi_\varepsilon}.
\end{equation}

Straightforward computations yield
$
	\left| \Delta \left( \xi_\varepsilon^{-2} e^{2 s \varphi_\varepsilon} \right) \right|
	\leq
	C s^2 \lambda^2 e^{2s \varphi_\varepsilon},
$
from which we get
\begin{equation}
	\label{Est-Dy-Y-OK}
	\left| \frac{s \lambda^2}{2} \iint_{\Ocal_T} |Y|^2 \Delta \left( \xi_\varepsilon^{-2} e^{2 s \varphi_\varepsilon} \right) \right|
	\leq
	C s^3 \lambda^4 \iint_{\Ocal_T} |Y|^2  e^{2s \varphi_\varepsilon}.
\end{equation}

Using Cauchy-Schwarz estimates,
\begin{eqnarray}
	\lefteqn{\left|
	s \lambda^2  \iint_{\Ocal_T} \left((f  + H 1_{\widehat{\omega}_T})\xi_\varepsilon^{-2} e^{2 s \varphi_\varepsilon} \right) Y
	\right|}
		\label{Est-f-F-Y-OK}
	\\
	& \leq &
		s^3 \lambda^4 \iint_{\Ocal_T} |Y|^2 e^{2 s \varphi_\varepsilon}
		+
		 \frac{C}{s} \iint_{\Ocal_T} \xi_\varepsilon^{-4} |f|^2 e^{2 s \varphi_\varepsilon}
		 +
		 \frac{C}{s} \iint_{\widehat{\omega}_T} \xi_\varepsilon^{-4} |H|^2 e^{2 s \varphi_\varepsilon}.
	\notag
\end{eqnarray}

Since we obviously have
\begin{equation}
\label{Est-Dy-Dt-t=T-dtsigma}
	\left| s \lambda^2 \iint_{\Ocal_T} |Y|^2 \partial_t \overline \sigma \xi_\varepsilon^{-2} e^{2 s \varphi_\varepsilon} \right|
	\leq C s^3 \lambda^4 \iint_{\Ocal_T} |Y|^2 e^{2 s \varphi_\varepsilon},
\end{equation}
combining estimates \eqref{Est-Dy-Dt-t=0}--\eqref{Est-Dy-Dt-t=T}--\eqref{Est-Dy-Y-OK}--\eqref{Est-f-F-Y-OK}--\eqref{Est-Dy-Dt-t=T-dtsigma} and plugging \eqref{Est-GradY-1}, we obtain
\begin{multline*}
	s \lambda^2 \iint_{\Ocal_T} \xi_\varepsilon^{-2} |\nabla Y|^2 e^{2 s \varphi_\varepsilon}+ s^2 \lambda^2 \int_0^{T_0}\int_\Ocal \overline \sigma |\partial_t \theta| \xi^{-2}  |Y|^2 e^{2 s \varphi}
	\\
	\leq
	C s^3 \lambda^4 \iint_{\Ocal_T} |Y|^2 e^{2s \varphi_\varepsilon} + C \iint_{\widehat{\omega}_T} \xi_\varepsilon^{-3} |H|^2 e^{2s \varphi_\varepsilon}
	+
	C \iint_{\Ocal_T} \xi_\varepsilon^{-3} |f|^2 e^{2s \varphi_\varepsilon}.
\end{multline*}
Since the constant $C$ is independent of $\varepsilon >0$, we can pass to the limit $\varepsilon \to 0$, and using \eqref{Estimate-Y-V} and the fact that $\overline \sigma$ is bounded from below away from $0$, we get:
\begin{equation}
	s \lambda^2 \iint_{\Ocal_T} \xi^{-2} |\nabla Y|^2 e^{2 s \varphi}+ s^2 \lambda^2 \int_0^{T_0}\int_\Ocal |\partial_t \theta| \xi^{-2} \varphi |Y|^2 e^{2 s \varphi}
	\leq
	C \iint_{\Ocal_T} \xi^{-3} |f|^2 e^{2 s \varphi}.
	\label{Estimate-NablaY}
\end{equation}

{\bf Estimates on $\Delta Y$, $\partial_t Y$.}
Multiplying the equation \eqref{Heat-control-Y} by $-\xi_\varepsilon^{-4} \Delta Y e^{2 s \varphi_\varepsilon}/s$,
\begin{multline}
	\label{Identity-Delta-Y}
	-\frac{1}{2s} \iint_{\Ocal_T} \partial_t (\overline \sigma \xi_\varepsilon^{-4} e^{2 s \varphi_\varepsilon}) |\nabla Y|^2
	+
	\frac{1}{s}\iint_{\Ocal_T} \partial_t Y \nabla Y \cdot \nabla (\overline \sigma \xi_\varepsilon^{-4} e^{2 s \varphi_\varepsilon} )
	\\
	+
	\frac{1}{s}\iint_{\Ocal_T} \xi_\varepsilon^{-4} |\Delta Y|^2 e^{2 s \varphi_\varepsilon}
	=
	- \frac{1}{s}  \iint_{\Ocal_T} (f+V 1_{\widehat{\omega}_T}- \partial_t \overline \sigma Y) \xi_\varepsilon^{-4} \Delta Y e^{2 s \varphi_\varepsilon}.
\end{multline}

As in \eqref{Term-Dt-t=0}, we compute explicitly $- \partial_t (\xi_\varepsilon^{-4} e^{2 s \varphi_\varepsilon})$. Arguing as in \eqref{Est-Dy-Dt-t=0}, we get
\begin{multline}
	-\frac{1}{2s} \int_0^{T_0}\int_\Ocal \overline \sigma \partial_t (\xi_\varepsilon^{-4} e^{2 s \varphi_\varepsilon}) |\nabla Y|^2
	\geq
	c \int_0^{T_0}\int_\Ocal \overline \sigma |\partial_t \theta| \xi^{-4} \varphi |\nabla Y|^2e^{2 s \varphi}
	\\
	- C \lambda \int_0^{T_0}\int_\Ocal \xi^{-2} |\nabla Y|^2 e^{2 s \varphi}.
\end{multline}
Besides, arguing as in \eqref{Est-Dy-Dt-t=T}, we get
\begin{equation}
	\left|
	-\frac{1}{2s} \int_{T_0}^T\int_\Ocal \overline \sigma \partial_t (\xi_\varepsilon^{-4} e^{2 s \varphi_\varepsilon}) |\nabla Y|^2
	\right|
	\leq
	C s \lambda^2 \iint_{\Ocal_T} \xi_\varepsilon^{-2} |\nabla Y|^2 e^{2 s \varphi_\varepsilon}.
\end{equation}
One can also easily check that
\begin{equation}
	\left|
	-\frac{1}{2s} \iint_{\Ocal_T} \partial_t \overline \sigma  (\xi_\varepsilon^{-4} e^{2 s \varphi_\varepsilon}) |\nabla Y|^2
	\right|
	\leq
	C s \lambda^2 \iint_{\Ocal_T} \xi_\varepsilon^{-2} |\nabla Y|^2 e^{2 s \varphi_\varepsilon}.
\end{equation}
We then estimate the cross-term of \eqref{Identity-Delta-Y}:
\begin{multline}
	\left|
	\frac{1}{s}\iint_{\Ocal_T} \partial_t Y \nabla Y \cdot \nabla (\overline \sigma \xi_\varepsilon^{-4} e^{2 s \varphi_\varepsilon} )
	\right|
	\\
	\leq
	\frac{\overline \sigma_{\rm min}^2}{8  s} \iint_{\Ocal_T} \xi_\varepsilon^{-4} |\partial_t Y|^2 e^{2s \varphi_\varepsilon}
	+
	C s \lambda^2 \iint_{\Ocal_T}  \xi_\varepsilon^{-2} |\nabla Y|^2 e^{2 s \varphi_\varepsilon},
\end{multline}
where $\overline \sigma_{\rm min}\ov\min_{\overline{\Ocal_T}}\overline \sigma$. From the equation \eqref{Heat-control-Y},
\begin{equation}
	\label{DtY-ParEq}
	\partial_t Y = \frac{1}{\overline \sigma}\left( \Delta Y + f + H  1_{ \omega_T} - \partial_t \overline \sigma Y\right),
\end{equation}
and thus we deduce
\begin{multline}
	\frac{\overline \sigma_{\rm min}^2}{8  s}  \iint_{\Ocal_T} \xi_\varepsilon^{-4} |\partial_t Y|^2 e^{2s \varphi_\varepsilon}
	\leq
	\frac{1}{4s } \iint_{\Ocal_T} \xi_\varepsilon^{-4} |\Delta Y|^2 e^{2s \varphi_\varepsilon}
	+
	\frac{C}{s} \iint_{\Ocal_T}  |Y|^2 e^{2 s \varphi_\varepsilon}
	\\
	+
	\frac{C}{s} \iint_{\Ocal_T} \xi_\varepsilon^{-3} |f|^2   e^{2 s \varphi_\varepsilon}
	+
	\frac{C}{s} \iint_{\widehat{\omega_T } } \xi_\varepsilon^{-3} |H|^2  e^{2 s \varphi_\varepsilon}.
	\label{DtY-par-DeltaY}
\end{multline}
And of course,
\begin{multline}
	\left|
		- \frac{1}{s}  \iint_{\Ocal_T} (f+H 1_{ \omega_T}- \partial_t \overline \sigma Y) \xi_\varepsilon^{-4} \Delta Y e^{2 s \varphi_\varepsilon}
	\right|
	\leq
	\frac{1}{4s} \iint_{\Ocal_T} \xi_\varepsilon^{-4} |\Delta Y|^2 e^{2 s \varphi_\varepsilon}
	\\
	+ \frac{C}{s} \iint_{\Ocal_T}\xi_\varepsilon^{-3} |f|^2  e^{2 s \varphi_\varepsilon}
	+
	\frac{C}{s} \iint_{\widehat{\omega_T } }  \xi_\varepsilon^{-3} |H|^2  e^{2 s \varphi_\varepsilon}
	+ \frac{C}{s} \iint_{\Ocal_T} |Y|^2  e^{2 s \varphi_\varepsilon}.
\end{multline}
Combining all the above estimates, we get
\begin{multline*}
	\frac{1}{2s} \iint_{\Ocal_T} \xi_\varepsilon^{-4} |\Delta Y|^2 e^{2 s \varphi_\varepsilon}
	\leq
	C s \lambda^2  \iint_{\Ocal_T}  \xi_\varepsilon^{-2} |\nabla Y|^2 e^{2 s \varphi_\varepsilon}
	+
	C s^3 \lambda^4 \iint_{\Ocal_T} |Y|^2  e^{2 s \varphi_\varepsilon}
	\\
	+
	C  \iint_{\Ocal_T} \xi_\varepsilon^{-3} |f|^2  e^{2 s \varphi_\varepsilon}
	+ C \iint_{\widehat{\omega_T } }  \xi_\varepsilon^{-3} |H|^2 e^{2 s \varphi_\varepsilon}.
\end{multline*}
Since the constant $C$ does not depend on $\varepsilon >0$, we can pass to the limit $\varepsilon \to 0$:
\begin{multline*}
	\frac{1}{2s} \iint_{\Ocal_T} \xi^{-4} |\Delta Y|^2 e^{2 s \varphi}
	\leq
	C s \lambda^2  \iint_{\Ocal_T} \xi^{-2} |\nabla Y|^2 e^{2 s \varphi}
	+
	C s^3 \lambda^4 \iint_{\Ocal_T} |Y|^2 e^{2 s \varphi}
	\\
	+
	C  \iint_{\Ocal_T} \xi^{-3} |f|^2 e^{2 s \varphi}
	+
	C \iint_{\widehat{\omega_T }} \xi^{-3} |H|^2 e^{2 s \varphi}.
\end{multline*}
Using now estimates \eqref{Estimate-Y-V}, \eqref{Estimate-NablaY} and \eqref{DtY-par-DeltaY}, we get
\begin{equation}
	\label{Estimate-Delta-Dt-Y}
	\frac{1}{s} \iint_{\Ocal_T} \xi^{-4} (|\partial_t Y|^2 + |\Delta Y|^2) e^{2 s \varphi}
	\leq
	C  \iint_{\Ocal_T} \xi^{-3} |f|^2 e^{2 s \varphi}.
\end{equation}

{\bf Estimates on $\partial_{\bf n} Y$ in $L^2(\Gamma_T)$.}
Let $ \eta: \overline{\Ocal} \mapsto \R^N$ such that $\eta \in C^2(\overline{\Ocal};\R^N) $ and $\eta= \vec{n}$ on $\partial\Ocal$. Since $Y$ vanishes on $\Gamma_T$, we have the following identity: for all $\varepsilon >0$,
\begin{multline*}
	\frac{1}{2}\int_{\Gamma_T} \xi_\varepsilon^{-3} |\partial_{\bf n} Y|^2 e^{2s \varphi_\varepsilon} = \iint_{\Ocal_T} \xi_\varepsilon^{-3} \Delta Y  \eta \cdot \nabla Y e^{2 s \varphi_\varepsilon}
	\\
	  + \iint_{\Ocal_T} D\left(\xi_\varepsilon^{-3} \eta e^{2s \varphi_\varepsilon}\right) (\nabla Y, \nabla Y) - \frac{1}{2}  \iint_{\Ocal_T} \div (\eta \xi_\varepsilon^{-3} e^{2 s \varphi_\varepsilon} )  |\nabla Y|^2.
\end{multline*}
Hence
$$
	\lambda \int_{\Gamma_T} \xi_\varepsilon^{-3} |\partial_{\bf n} Y|^2 e^{2s \varphi_\varepsilon}
	\leq
	\frac{1}{s} \iint_{ \Ocal_T} \xi_\varepsilon^{-4} |\Delta Y|^2 e^{2s \varphi_\varepsilon}
	+
	C s \lambda^2 \iint_{\Ocal_T} \xi_\varepsilon^{-2} |\nabla Y|^2 e^{2s \varphi_\varepsilon}.
$$
Passing to the limit in $\varepsilon \to 0$ and using \eqref{Estimate-NablaY} and \eqref{Estimate-Delta-Dt-Y} we thus obtain
\begin{equation}
	\label{Estimate-Dnu-Y}
	\lambda \int_0^T\int_{\partial \Ocal} \xi^{-3} |\partial_{\bf n} Y|^2 e^{2s \varphi}
	\leq
	C  \iint_{\Ocal_T} \xi^{-3} |f|^2 e^{2 s \varphi}.
\end{equation}

{\bf Conclusion.} Estimates \eqref{Estimate-Y-V}, \eqref{Estimate-NablaY}, \eqref{Estimate-Delta-Dt-Y} and \eqref{Estimate-Dnu-Y} yield \eqref{Est-Y-Gal}.
\section{Regularity of the weight function}\label{Appendix-Reg-Weight}

\begin{proof}[Proof of Lemma \ref{Lem-Reg-Weight}]
	The first remark is that the flow $\overline{X}_e$ is ${\bf C}^2([0,T]\times[0,T]\times \R^2)$ since $\overline{\y}_e \in {\bf C}^2([0,T]\times \R^2)$.
	
	In order to study the regularity of $\widehat\psi$, we will introduce the function $t_{\rm out}= t_{\rm out}(t,x)$ defined for $(t,x) \in (0,T)\times \Ocal$ as the supremum of the time $\tau \in (t, T]$ for which $\forall t' \in (t,\tau), \, \overline{X}_e(t',t,x) \in \Ocal$. It is not difficult to check that this time $t_{\rm out}$ can also be characterized as the solution of
\begin{equation}
	\label{Eq-t-in}
  \left\{
		\begin{array}{rlll}
			\partial_t t_{\rm out}+\overline {\bf y}_e\cdot \nabla t_{\rm out} &=&0 &\mbox{ in } \Ocal_T,
			\\
			t_{\rm out}(t) &=&t&\mbox{ on } \Gamma_T,
			\\
		        t_{\rm out}(T)&=& T \quad &\mbox{ in }\Ocal.
		\end{array}
	\right.
	\end{equation}
	
	For convenience, we also set
	\begin{equation}\label{Defxin}
		x_{\rm out}(t,x) = \overline{X}_e(t_{\rm out}(t,x),t,x).
	\end{equation}
	We first prove that $t_{\rm out}$ is continuous in $\Ocal_T$. In order to do that, let us remark that $\overline{X}_e$ is ${\bf C}^2([0,T]\times[0,T]\times \R^2)$ and for all $(t, \tau) \in [0,T]^2$, $\overline{X}_e(t,\tau,\cdot)$ is a $C^2$ diffeomorphism of $\R^2$. In particular, $\Ocal_T$ can be decomposed into
	\begin{multline}
		\label{Decomp-Ocal-T}
		\Ocal_{T}  = \Ocal_{T,1} \cup  \Ocal_{T,2} \cup \Sigma_T,
		\\
		 \hbox{ with }
		\left\{\begin{array}{lcl}
			\Ocal_{T,1} & = &
			\{(t,x) \in (0,T) \times \Ocal,\,  x \in \overline{X}_e(t, T, \Ocal)\} ,
			\\
			\Ocal_{T,2} & = &
			\{(t,x) \in (0,T) \times \Ocal,\, x \in \overline{X}_e(t,T, \R^2 \setminus\overline{\Ocal})\},
			\\
			\Sigma_{T} & =	&\{(t,x) \in (0,T) \times \Ocal,\, x \in \overline{X}_e(t,T, \partial \Ocal)\}.
		\end{array}\right.
	\end{multline}
	In \eqref{Decomp-Ocal-T}, $\Ocal_{T,1}$ and $\Ocal_{T,2}$ are open sets whereas $\Sigma_T = \overline{\Ocal_{T,1}} \cap \overline{\Ocal_{T,2}}$ is closed and of dimension $2$.
	For $(t,x) \in \Ocal_{T,1} \cup \Sigma_T$, $t_{\rm out}(t,x) = T$ and $t_{\rm out}$ is thus continuous on $\overline{\Ocal_{T,1}}$. The continuity on $\overline{\Ocal_{T,2}}$ is more involved. If $(t,x) \in \overline{\Ocal_{T,2}}$, then $x_{\rm out}(t,x)$ belongs to $\partial\Ocal$. Due to the condition \eqref{Hyp-y-e}, for any $\varepsilon >0$, there exists a neighborhood $\mathscr{V}_\varepsilon$ of $(t_{\rm out}(t,x), x_{\rm out}(t,x))$ in $[0,T] \times \overline{\Ocal}$ such that $|t_{\rm out}(t',x')-t_{\rm out}(t,x)|< \varepsilon$ for all $(t',x') \in \mathscr{V}_\varepsilon$. In particular, for some $t_\varepsilon \in (0,T)$ close to $t_{\rm out}(t,x)$,  $\mathscr{V}_\varepsilon$ is a neighborhood of $(t_\varepsilon,\overline{X}_e(t_\varepsilon, t_{\rm out}(t,x), x_{\rm out}(t,x))) = (t_\varepsilon,\overline{X}_e(t_\varepsilon,t,x))$. Following, $\{ \overline{X}_e(t-t_\varepsilon +t',t',x'), \, (t',x') \in  \mathscr{V}_\varepsilon\}$ is a neighborhood of $(t,\overline{X}_e(t,t_\varepsilon, \overline{X}_e(t_\varepsilon,t,x) ))  = (t,x)$ on which $t_{\rm out}$ is at distance at most $\varepsilon$ of $t_{\rm out}(t,x)$.
	
	Thus, $t_{\rm out}$ is continuous in $\overline{\Ocal_T}$. As $\widehat\psi$ solution of \eqref{Hat-Psi-Eq} can be written as
	\begin{equation}
		\widehat\psi(t,x) =
			\left\{
				\begin{array}{ll}
					\widehat \psi_T(x_{\rm out}(t,x)) \quad & \hbox{ if } t_{\rm out} (t,x) = T,
					\\
					t_{\rm out}(t,x)-T \quad & \hbox{ if } t_{\rm out} (t,x) < T,
				\end{array}
			\right.
		\label{WideHat-Psi-Explicit}
	\end{equation}	
	the continuity of $\widehat\psi$ in $\overline{\Ocal_T}$ follows from the first compatibility condition in \eqref{Boundary-Conditions-Psi-T}. Also note that $\widehat\psi$ is obviously $C^2$ in $\Ocal_{T,1}$.
	
	We then focus on the $C^1$ regularity of $\widehat\psi$. In order to do this, we remark that $\nabla t_{\rm out}$ solves
	\begin{equation}
	\label{Eq-nabla t-in}
	 \left\{
		\begin{array}{rlll}
			\partial_t \nabla t_{\rm out}+(\overline {\bf y}_e\cdot \nabla)\nabla t_{\rm out} + D\overline\y_e \nabla t_{\rm out} &=&0 &\mbox{ in } \Ocal_T,
			\\
			\nabla t_{\rm out}(t,x) &=& \displaystyle - \frac{\n(x)}{\overline{\y}_e(t,x)\cdot \n(x)} &\mbox{ on } \Gamma_T,
			\\
		        \nabla t_{\rm out}(T)&=& 0 \quad &\mbox{ in }\Ocal.
		\end{array}
	\right.
	\end{equation}
	In particular, $\nabla t_{\rm out}$ can be computed for any $(t,x) \in \Ocal_{T,2}$ by solving for $\tau$ between $t$ and $t_{\rm out}(t,x)$ the ODE
	\begin{multline*}
		\frac{d}{d\tau} \left( \nabla t_{\rm out}(\tau,\overline{X}_e(\tau,t,x)) \right) = -D\overline{\y}_e(\tau, \overline{X}_e(\tau,t,x)) \nabla t_{\rm out}(\tau,\overline{X}_e(\tau,t,x)),
		\,  \tau \in (t,t_{\rm out}(t,x)),
		\\
		\hbox{with } \nabla t_{\rm out} (t_{\rm out}(t,x),x_{\rm out}(t,x)) = - \frac{\n(x_{\rm out}(t,x))}{\overline{\y}_e(t_{\rm out}(t,x),x_{\rm out}(t,x))\cdot \n(x_{\rm out}(t,x))}.		
	\end{multline*}

\par\noindent One then easily obtains that $\nabla t_{\rm out}$ is $C^0$ on $\Ocal_{T,2}$ and from the equation \eqref{Eq-t-in} we deduce that $t_{\rm out}$ is $C^1$ in $\Ocal_{T,2}$. From there, we derived immediately from \eqref{WideHat-Psi-Explicit} that $\widehat\psi$ is $C^1$ on $\Ocal_{T,2}$ and that it can be extended as a $C^1$ funtion on $\overline{\Ocal_{T,2}}$ as follows: $\nabla \widehat\psi$ can be computed for any $(t,x) \in \Sigma_T$ by solving for $\tau$ between $t$ and $T$ the ODE:
\begin{gather}
		\frac{d}{d\tau} \left( \nabla \widehat\psi(\tau,\overline{X}_e(\tau,t,x)) \right) = -D\overline{\y}_e(\tau, \overline{X}_e(\tau,t,x)) \nabla \widehat\psi(\tau,\overline{X}_e(\tau,t,x)),
		\,  \tau \in (t,T), \label{Eqhatpsi1}
		\\
	\displaystyle 	\hbox{with } \nabla \widehat\psi (T,\overline{X}_e(T,t,x)) = - \frac{\n(\overline{X}_e(T,t,x))}{\overline{\y}_e(T,\overline{X}_e(T,t,x))\cdot \n(\overline{X}_e(T,t,x))}.	\label{Eqhatpsi2}	
	\end{gather}
	On the other hand, $\widehat \psi$ solves the equation \eqref{Eq-hat-psi-CritPoints}, and can be extended as a $C^1$ function on $\overline{\Ocal_{T,1}}$. For $(t,x) \in \Sigma_T$, this yields $\nabla\widehat\psi(t,x)$ as the solution of the ODE \eqref{Eqhatpsi1} with $\nabla \widehat\psi(T, \overline{X}_e(T,t,x))$ given. But, as $\widehat\psi(T)$ is constant on the boundary and satisfies the second compatibility condition in \eqref{Hyp-y-e}, we get again \eqref{Eqhatpsi2} for $(t,x) \in \Sigma_T$. Following, $\nabla \widehat\psi$ is continuous across $\Sigma_T$, hence on $\overline{\Ocal_T}$. Using the equation \eqref{Hat-Psi-Eq}, $\widehat\psi$ belongs to $C^1(\overline{\Ocal_T})$.
	
	The proof of the $C^2$ regularity follows the same path and is left to the reader.
\end{proof}

\bibliographystyle{plain}

\end{document}